\documentclass[12pt]{amsart}
\usepackage{}

\usepackage{amsmath}
\usepackage{amsfonts}
\usepackage{amssymb}
\usepackage[all]{xy}           

\usepackage{bbding}
\usepackage{txfonts}
\usepackage{amscd}

\usepackage[shortlabels]{enumitem}
\usepackage{ifpdf}
\ifpdf
  \usepackage[colorlinks,final,
  hyperindex]{hyperref}
\else
  \usepackage[colorlinks,final,
  hyperindex]{hyperref}
\fi
\usepackage{tikz}
\usepackage[active]{srcltx}

\makeatletter

\numberwithin{equation}{section}

\newtheorem{thm}{Theorem}[section]
\newtheorem{lem}[thm]{Lemma}
\newtheorem{cor}[thm]{Corollary}
\newtheorem{pro}[thm]{Proposition}
\newtheorem{ex}[thm]{Example}

\newtheorem{defi}[thm]{Definition}

\newcommand{\kl}{\mathfrak l}
\newcommand{\kr}{\mathfrak r}

\newcommand{\fl}{\mathbf l}
\newcommand{\fr}{\mathbf r}

\newcommand{\End}{\mathrm{End}}

\newcommand{\Img}{\mathrm{Im}}

\newcommand{\Hom}{\mathrm{Hom}}

\def\id{\mathop {\fam0 id}\nolimits}

\begin{document}

\title[Averaging antisymmetric infinitesimal bialgebra and perm bialgebras]
{Averaging antisymmetric infinitesimal bialgebra\\ and induced perm bialgebras}

\author{Bo Hou}
\address{School of Mathematics and Statistics, Henan University, Kaifeng 475004,
China}
\email{bohou1981@163.com}
\vspace{-5mm}

\author{Zhanpeng Cui}
\address{School of Mathematics and Statistics, Henan University, Kaifeng 475004,
China}
\email{czp15824833068@163.com}


\begin{abstract}
We establish a bialgebra theory for averaging algebras, called averaging antisymmetric
infinitesimal bialgebras by generalizing the study of antisymmetric infinitesimal bialgebras
to the context of averaging algebras. They are characterized by double constructions of
averaging Frobenius algebras as well as matched pairs of averaging algebras. Antisymmetric
solutions of the Yang-Baxter equation in averaging algebras provide averaging antisymmetric
infinitesimal bialgebras. The notions of an $\mathcal{O}$-operator of an averaging algebra
and an averaging dendriform algebra are introduced to construct antisymmetric solutions of
the Yang-Baxter equation in an averaging algebra and hence averaging antisymmetric
infinitesimal bialgebras. Moreover, we introduce the notion of factorizable
averaging antisymmetric infinitesimal bialgebras and show that a factorizable
averaging antisymmetric infinitesimal bialgebra leads to a factorization of the
underlying averaging algebra. We establish a one-to-one correspondence between
factorizable averaging antisymmetric infinitesimal bialgebras and symmetric averaging
Frobenius algebras with a Rota-Baxter operator of nonzero weight.
Finally, we apply the study of averaging antisymmetric infinitesimal bialgebras to
perm bialgebras, extending the construction of perm algebras from commutative
averaging algebras to the context of bialgebras, which is consistent with the well
constructed theory of perm bialgebras.
\end{abstract}


\maketitle

\tableofcontents 



\vspace{-4mm}

\section{Introduction}\label{sec:intr}

The notion of averaging operator was first implicitly studied by Reynolds
in the turbulence theory of fluid dynamics \cite{Rey}. Kamp\'{e} de F\'eriet
introduced explicitly the averaging operator in the context
of turbulence theory and functional analysis \cite{KdF,Mil}.
Moy investigated averaging operators from the viewpoint of conditional expectation
in probability theory \cite{Moy}. Kelley and Rota studied
the role of averaging operators in Banach algebras \cite{Kel,Rot}.
The algebraic study on averaging operators began in \cite{Cao}. Cao constructed explicitly
free unitary commutative averaging algebras and discovered the Lie algebra structures
induced naturally from averaging operators. In \cite{PG}, Guo and Pei studied averaging
operators from an algebraic and combinatorial point of view, and constructed free nonunital
averaging algebras in terms of a class of bracketed words called averaging words.
In \cite{GZ}, Gao and Zhang contain an explicit construction of free unital averaging
algebras in terms of bracketed polynomials and the main tools were rewriting systems
and Gr\"{o}bner-Shirshov bases. The averaging operators attract much attention also
because of their closely connected with Reynolds operators, symmetric operators and
Rota-Baxter operators \cite{Bon,GM,Tri}.

An averaging algebra is an algebra with an averaging operator.
In recent years, the properties of averaging algebra have been widely studied.
Sheng, Tang and Zhu studied embedding tensors (another name of averaging operators
in physics) for Lie algebras and they construct a cohomology theory for such operators
on Lie algebras using derived brackets as a main tool \cite{STZ}.
In \cite{WZ,Das,DS} the authors study the cohomological theory, homotopy theory and
non-abelian extensions for averaging associative algebras. The averaging operators
on various algebraic structures and the induced structures have been studied in \cite{Das1}.
The aim of this paper is to develop a bialgebra theory for averaging (associative)
algebras and get some applications. The notion of averaging antisymmetric infinitesimal
bialgebras is introduced. Some special averaging antisymmetric infinitesimal
bialgebras are studied. As an application, we generalize the typical construction
of perm algebras from averaging algebras to the context of bialgebras.

A bialgebra structure consists of an algebra structure and a coalgebra structure coupled by certain compatibility conditions. Such structures have connections with
other structures arising from mathematics and physics. Lie bialgebras are the algebra
structures of Poisson-Lie groups and play an important role in the study of
quantized universal enveloping algebras \cite{Dri,CP}. Antisymmetric infinitesimal
bialgebras for associative algebras were introduced by Bai in order to establish
the connection with the double constructions of Frobenius algebras and matched
pairs of associative algebras in \cite{Bai}. In this paper, we establish a bialgebra
theory for averaging associative algebras, called averaging antisymmetric infinitesimal
bialgebras, by extending the study of antisymmetric infinitesimal bialgebras in \cite{Bai}
to the context of averaging algebras. Explicitly, averaging antisymmetric
infinitesimal bialgebras are characterized equivalently by matched pairs of averaging
algebras and double constructions of averaging Frobenius algebras, as the
generalizations of matched pairs of algebras and double constructions of Frobenius
algebras respectively to the context of averaging algebras. The coboundary cases
lead to introduce the notion of $\beta$-Yang-Baxter equation in an averaging
algebra, whose antisymmetric solutions are used to construct averaging antisymmetric
infinitesimal bialgebras. The notions of $\mathcal{O}$-operators of averaging
algebras and averaging dendriform algebras are introduced to construct antisymmetric
solutions of the $\beta$-Yang-Baxter equation in averaging algebras and
hence give rise to averaging antisymmetric infinitesimal bialgebras.
We summarize these results in the following diagram:\\[-5mm]
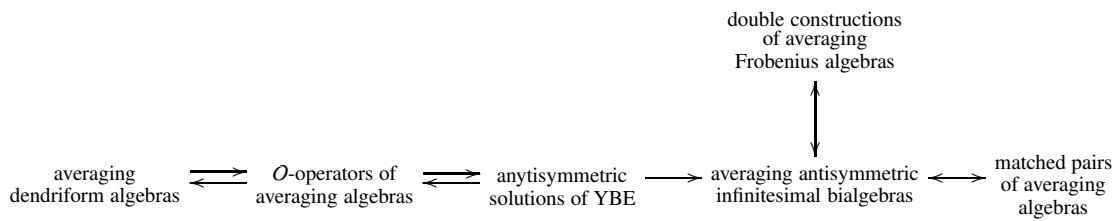
$$
{\tiny \xymatrix@C=0.75cm{& & &
\txt{double constructions \\ of averaging \\ Frobenius algebras}\ar@{<->}[d]_-{}\\
\txt{averaging \\ dendriform algebras} \ar@{->}[r]<0.7ex>^-{}&
\txt{$\mathcal{O}$-operators of \\ averaging algebras}
\ar@{->}[r]<0.5ex>^-{}
\ar@{->}[l]<0.5ex>^-{} &
\txt{anytisymmetric \\ solutions of YBE}
\ar@{->}[l]<0.5ex>^-{}
\ar@{->}[r]^-{} &
\txt{averaging antisymmetric\\ infinitesimal bialgebras} &
\txt{matched pairs \\ of averaging \\ algebras} \ar@{<->}[l]_-{}}}
$$

Quasitriangular Lie bialgebras and triangular Lie bialgebras are important Lie
bialgebras classes. Another important Lie bialgebras are factorizable Lie bialgebras,
which is introduced in \cite{RS}. Factorizable Lie bialgebras are used to
establish the relation between classical $r$-matrices and certain factorization
problems in Lie algebras. Recently, factorizable Lie bialgebras and factorizable
antisymmetric infinitesimal bialgebras have been further studied in \cite{LS,SW}.
Here we study
factorizable averaging antisymmetric infinitesimal bialgebras. We show that the
factorizable averaging antisymmetric infinitesimal bialgebras give rise to a natural
factorization of the underlying averaging algebras. The importance
of factorizable averaging antisymmetric infinitesimal bialgebras in the study of
averaging antisymmetric infinitesimal bialgebras can also be observed from the fact
that the double space of an arbitrary averaging antisymmetric infinitesimal bialgebras
admits a factorizable averaging antisymmetric infinitesimal bialgebra structure.

Furthermore, as an application, we construct perm bialgebras from some special averaging
antisymmetric infinitesimal bialgebras. We have studied a bialgebra theory for perm
algebras in \cite{Hou}, also see \cite{LZB}. For any commutative averaging algebra,
the averaging operator induces a perm algebra structure on the original vector space.
It is natural to consider extending such a relationship to the context of bialgebras.
We establish the explicit relationships between averaging antisymmetric infinitesimal
bialgebras and the induced perm bialgebras, as well as the equivalent interpretation in
terms of the corresponding Manin triples and matched pairs. We show that a solution
$\beta$-YBE in a commutative averaging algebra is also a solution of YBE in the induced
perm algebra under certain conditions.

The paper is organized as follows. In Section \ref{sec:alg}, we recall some facts
on averaging algebras and bimodules over averaging algebras. In Section \ref{sec:bialg},
we give the general notions of matched pairs of averaging algebras, double construction
of averaging Frobenius algebras and averaging antisymmetric infinitesimal bialgebras.
Whenever the underlying linear spaces of the two averaging algebras are dual to each
other, their equivalence is interpreted. In Section \ref{sec:cob}, we consider a
special class of averaging antisymmetric infinitesimal bialgebras, the coboundary
averaging antisymmetric infinitesimal bialgebras. This study also leads to the
introduction of the Yang-Baxter equation in an averaging algebra, whose antisymmetric
solutions give averaging antisymmetric infinitesimal bialgebras.
We also introduce the notions of $\mathcal{O}$-operators of averaging algebras and
averaging dendriform algebras, and give constructions of antisymmetric solutions of
the Yang-Baxter equation in an averaging algebra from these structures.
In Section \ref{sec:fact}, we introduce the notion of factorizable
averaging antisymmetric infinitesimal bialgebras and show that a factorizable
averaging antisymmetric infinitesimal bialgebra leads to a factorization of the
underlying averaging algebra. We establish a one-to-one correspondence between
factorizable averaging antisymmetric infinitesimal bialgebras and symmetric averaging
Frobenius algebras with a Rota-Baxter operator of nonzero weight.
In Section \ref{sec:permbia}, proceeding from the typical construction of perm algebras
from commutative averaging algebras, we construct perm bialgebras from commutative and
cocommutative averaging antisymmetric infinitesimal bialgebras. The explicit
relationships between them, as well as the equivalent interpretation in terms
of the corresponding Manin triples and matched pairs, are established.

Throughout this paper, we fix $\Bbbk$ a field and characteristic zero.
All the vector spaces and algebras are of finite dimension over $\Bbbk$,
and all tensor products are also taking over $\Bbbk$.

\section{Averaging algebras and their bimodule} \label{sec:alg}

In this section, we recall the background on averaging algebras and bimodules over
averaging algebras, for that details, see \cite{Das,WZ}.

\begin{defi}\label{def:aver}
Let $(A, \cdot)$ be an associative algebra. An {\rm averaging operator} on $A$
is a linear map $\alpha: A\rightarrow A$ such that
$$
\alpha(a_{1})\alpha(a_{2})=\alpha(\alpha(a_{1})a_{2})=\alpha(a_{1}\alpha(a_{2})),
$$
for all $a_{1}, a_{2}\in A$. A triple $(A, \cdot, \alpha)$ consisting of an associative algebra
$A$ and an averaging operator $\alpha: A\rightarrow A$ is called an {\rm averaging
(associative) algebra}. We often denote this averaging algebra by $(A, \alpha)$ for simply.
\end{defi}

Given two averaging algebras $(A, \alpha)$ and $(A', \alpha')$, a {\it homomorphism of
averaging algebras} from $(A, \alpha)$ to $(A', \alpha')$ is a homomorphism of algebras
$f: A\rightarrow A'$ satisfying $f\alpha=\alpha'f$. A homomorphism
$f: (A, \alpha)\rightarrow(A', \alpha')$ is said to be an {\it isomorphism} if
$f$ is a bijection. A subalgebra $B$ of associative algebra $A$ with a linear map
$\beta: B\rightarrow B$ is called a subalgebra of averaging algebras $(A, \alpha)$
if $\beta$ is just the restriction of $\alpha$ on $B$.

\begin{ex}\label{ex:2dim}
$(i)$ Let $A=\Bbbk\{e\}$ be a 1-dimensional associative algebra. Note that
any scalar multiple transformation on $A$ is an averaging operator, we get that
every linear map $\alpha: A\rightarrow A$ is an averaging operator.

$(ii)$ Let $A=\Bbbk\{e_{1}, e_{2}\}$ be a 2-dimensional associative algebra
with non-zero product $e_{1}e_{2}=e_{1}$ and  $e_{2}e_{2}=e_{2}$.
Then, up to isomorphism, all the non-zero averaging operators are given by
\begin{enumerate}\itemsep=0pt
\item[$(\mathrm{a})$] $\alpha(e_{1})=0$,\; $\alpha(e_{2})=e_{1}$;
\item[$(\mathrm{b})$] $\alpha(e_{1})=0$,\; $\alpha(e_{2})=e_{2}$;
\item[$(\mathrm{c})$] $\alpha(e_{1})=0$,\; $\alpha(e_{2})=e_{2}+ae_{1}$,\, $0\neq a\in\Bbbk$;
\item[$(\mathrm{d})$] $\alpha(e_{1})=e_{1}$,\; $\alpha(e_{2})=e_{2}$.
\end{enumerate}
%
\end{ex}

\begin{ex}\label{ex:dered}
Let $(A, \alpha)$ be an averaging algebra. We define two new binary operations
$\bullet, \star: A\otimes A\rightarrow A$ by
$$
a_{1}\bullet a_{2}=\alpha(a_{1})a_{2},\qquad\quad\mbox{and}\qquad\quad
a_{1}\star a_{2}=a_{1}\alpha(a_{2}),
$$
for all $a_{1}, a_{2}\in A$, then $A_{1}=(A, \bullet)$ and $A_{2}=(A, \star)$
are associative algebras, $(A_{1}, \alpha)$ and $(A_{2}, \alpha)$ are averaging algebras,
and $(A, \bullet, \star)$ is a dialgebra \cite{Agu}.
\end{ex}

\begin{defi}\label{def:perm}
A (left) {\rm perm algebra} is a vector space $P$ with a bilinear operation $(p_{1}, p_{2})
\mapsto p_{1}p_{2}$, such that $p_{1}(p_{2}p_{3})=(p_{1}p_{2})p_{3}=(p_{2}p_{1})p_{3}$,
for any $p_{1}, p_{2}, p_{3}\in P$.
\end{defi}

Clearly, perm algebra is a special class of associative algebra.
An averaging algebra is called commutative if it as an associative algebra
is commutative. We can construct perm algebras from commutative averaging algebras.

\begin{pro} \label{pro:perm}
For any commutative averaging algebra $(A, \alpha)$, we have a perm algebra
$(A, \bullet)$, where the product $\bullet$ is defined in Example \ref{ex:dered}.
\end{pro}

Let $A$ be an associative algebra, $M$ be a vector space. $M$ is called a
bimodule over $A$, if there are linear maps $\kl, \kr: A\rightarrow\End_{\Bbbk}(M)$,
such that $(a_{1}a_{2})m=a_{1}(a_{2}m)$, $a_{1}(ma_{2})=(a_{1}m)a_{2}$ and
$m(a_{1}a_{2})=(ma_{1})a_{2}$, for any $a_{1}, a_{2}\in A$
and $m\in M$, where $a_{1}m:=\kl(a_{1})(m)$ and $ma_{1}:=\kr(a_{1})(m)$.
We denote this bimodule by $(M, \kl, \kr)$. Clearly, for any
associative algebra $A$, $(A, \fl_{A}, \fr_{A})$ is a bimodule over itself,
where $\fl_{A}(a_{1})a_{2}=a_{1}a_{2}$ and $\fr_{A}(a_{1})a_{2}=a_{2}a_{1}$.

\begin{defi}\label{def:mod}
Let $(A, \alpha)$ be an averaging algebra. A {\rm bimodule} $(M, \kl, \kr, \beta)$
over the averaging algebra $(A, \alpha)$ is a bimodule $(M, \kl, \kr)$ over associative
algebra $A$ endowed with an operator $\beta: M\rightarrow M$, such that for any $a\in A$,
$m\in M$, the following equalities hold:
\begin{align}
\kl(\alpha(a))(\beta(m))&=\beta(\kl(\alpha(a))(m))=\beta(\kl(a)(\beta(m))), \label{adm1} \\
\kr(\alpha(a))(\beta(m))&=\beta(\kr(\alpha(a))(m))=\beta(\kr(a)(\beta(m))). \label{adm2}
\end{align}
\end{defi}

Given two bimodules $(M, \beta)$ and $(N, \beta')$ over averaging algebra $(A, \alpha)$,
a {\it homomorphism} from $(M, \beta)$ and $(N, \beta')$ is a bimodule homomorphism
$f: M\rightarrow N$ over associative $A$ such that $f\beta=\beta'f$.
If the homomorphism $f: (M, \beta)\rightarrow(N, \beta')$ is a bijection, we call
that $f$ is an isomorphism, and $(M, \beta)$ and $(N, \beta')$ are isomorphic.
The averaging algebra $(A, \alpha)$ itself is naturally a bimodule over
itself, called the {\it regular bimodule}. For general bimodule over averaging algebra,
we have the following proposition.

\begin{pro} \label{pro:mod}
Let $(A, \alpha)$ be an averaging algebra, $M$ be a vector space, $\kl,\, \kr:
A\rightarrow\End_{\Bbbk}(M)$ and $\beta: M\rightarrow M$ be linear maps.
Then $A\oplus M$ with the multiplication
$$
(a_{1}, m_{1})(a_{2}, m_{2})=\big(a_{1}a_{2},\ \ \kl(a_{1})(m_{2})+\kr(a_{2})(m_{1})\big),
$$
and linear map $\alpha\oplus\beta: A\oplus M\rightarrow A\oplus M$,
$(a, m)\mapsto(\alpha(a), \beta(m))$, for any $(a, m), (a_{1}, m_{1}), (a_{2}, m_{2})
\in A\oplus M$, is an averaging algebra if and only if $(M, \kl, \kr, \beta)$
is a bimodule over $(A, \alpha)$. This averaging algebra structure on $A\oplus M$ is
called the {\rm semidirect product} of $(A, \alpha)$ by bimodule $(M, \beta)$,
and denoted by $(A\ltimes M, \alpha\oplus\beta)$.	
\end{pro}

Let $V$ be a vector space. Denote the standard pairing between the dual space
$V^{\ast}$ and $V$ by
\begin{align*}
\langle-,-\rangle:\quad V^{\ast}\otimes V\rightarrow \Bbbk, \qquad\quad
\langle \xi,\; v \rangle:=\xi(v),
\end{align*}
for any $\xi\in V^{\ast}$ and $v\in V$.
Let $V$, $W$ be two vector spaces. For a linear map $\varphi: V\rightarrow W$,
the transpose map $\varphi^{\ast}: W^{\ast}\rightarrow V^{\ast}$ is defined by
\begin{align*}
\langle \varphi^{\ast}(\xi),\; v \rangle:=\langle\xi,\; \varphi(v)\rangle,
\end{align*}
for any $v\in V$ and $\xi\in W^{\ast}$.
Let $A$ be an associative algebra and $V$ be a vector space.
For a linear map $\psi: A\rightarrow\End_{\Bbbk}(V)$, the linear map
$\psi^{\ast}: A\rightarrow\End_{\Bbbk}(V^{\ast})$ is defined by
\begin{align*}
\langle\psi^{\ast}(a)(\xi),\; v\rangle:=\langle\xi,\; \psi(a)(v)\rangle,
\end{align*}
for any $a\in A$, $v\in V$, $\xi\in V^{\ast}$. That is, $\psi^{\ast}(a)=\psi(a)^{\ast}$
for all $a\in A$. It is easy to see that, for each bimodule $(M, \kl, \kr)$ over
associative algebra $A$, the triple $(M^{\ast}, \kr^{\ast}, \kl^{\ast})$ is again
a bimodule over $A$.

\begin{pro}\label{pro:dual}
Let $(A, \alpha)$ be an averaging algebra, $(M, \kl, \kr)$ be a bimodule
over associative algebra $A$, and $\beta: M\rightarrow M$ be a linear maps.
Then the quadruple $(M^{\ast}, \kr^{\ast}, \kl^{\ast}, \beta^{\ast})$ is a bimodule
over the averaging algebra $(A, \alpha)$ if and only if $(M, \kl, \kr, \beta)$
is a bimodule over $(A, \alpha)$.
\end{pro}

\begin{proof}
By the definition of bimodule over an averaging algebra, we get
$(M^{\ast}, \kr^{\ast}, \kl^{\ast}, \beta^{\ast})$ is a bimodule over $(A, \alpha)$
if and only if
\begin{enumerate}\itemsep=0pt
\item[$(i)$] $\kl^{\ast}(\alpha(a))(\beta^{\ast}(\xi))=\beta^{\ast}(\kl^{\ast}
(\alpha(a))(\xi))=\beta^{\ast}(\kl^{\ast}(a)(\beta^{\ast}(\xi)))$, and
\item[$(ii)$] $\kr^{\ast}(\alpha(a))(\beta^{\ast}(\xi))
=\beta^{\ast}(\kr^{\ast}(a)(\beta^{\ast}(\xi)))=\beta^{\ast}(\kr^{\ast}(\alpha(a))(\xi))$,
\end{enumerate}
for any $a\in A$ and $\xi\in M^{\ast}$. Note that
$\kl^{\ast}(\alpha(a))\beta^{\ast}-\beta^{\ast}\kl^{\ast}(\alpha(a))
=(\beta\kl(\alpha(a))-\kl(\alpha(a))\beta)^{\ast}$ and
$\kl^{\ast}(\alpha(a))\beta^{\ast}-\beta^{\ast}\kl^{\ast}(a)\beta^{\ast}
=(\beta\kl(\alpha(a))-\beta\kl(a)\beta)^{\ast}$, we get
$(i)$ holds if and only if Eq. \eqref{adm1} holds.
Similarly, we also have $(ii)$ holds if and only if Eq. \eqref{adm2} holds.
Thus the conclusion follows.
\end{proof}

Thus, for an averaging algebra $(A, \alpha)$, the quadruple $(A^{\ast}, \fr_{A}^{\ast},
\fl_{A}^{\ast}, \alpha^{\ast})$ is a bimodule over $(A, \alpha)$, which is called
the {\it coregular bimodule}.

\section{Averaging antisymmetric infinitesimal bialgebras}\label{sec:bialg}
In this section, we introduce the notions of a double construction of averaging
Frobenius algebra and an averaging antisymmetric infinitesimal bialgebra, and give
their equivalence in terms of matched pairs of averaging algebras.

\subsection{Matched pairs of averaging algebras}\label{subsec:mat-pair}
We first recall the concept of a matched pair of associative algebras.

\begin{defi} \label{def:mat-pa}
A {\rm matched pair of associative algebras} consists of two associative algebras
$A$, $B$, and linear maps $\kl_{A}, \kr_{A}: A\rightarrow\End_{\Bbbk}(B)$ and
$\kl_{B}, \kr_{B}: B\rightarrow\End_{\Bbbk}(A)$, such that $(A\oplus B, \ast)$ is
also an associative algebra, where $\ast$ is defined by
\begin{align*}
(a_{1}, b_{1})\ast(a_{2}, b_{2})=\Big(a_{1}a_{2}+\kl_{B}(b_{1})(a_{2})+\kr_{B}(b_{2})(a_{1}),
\quad  b_{1}b_{2}+\kl_{A}(a_{1})(b_{2})+\kr_{A}(a_{2})(b_{1})\Big),
\end{align*}
for all $a_{1}, a_{2}\in A$ and $b_{1}, b_{2}\in B$. The matched pair is denoted by
$(A, B, \kl_{A}, \kr_{A}, \kl_{B}, \kr_{B})$ and the resulting algebra is denoted
by $A\bowtie B$.
\end{defi}

For a matched pair of associative algebras $(A, B, \kl_{A}, \kr_{A}, \kl_{B}, \kr_{B})$,
it is easy to see that $(A, \kl_{B}, \kr_{B})$ is a bimodule over $B$ and
$(B, \kl_A, \kr_A)$ is a bimodule over $A$.

\begin{defi}
Let $(A, \alpha)$ and $(B, \beta)$ be two averaging algebras.
Suppose that $\kl_{A}, \kr_{A}: A\rightarrow\End_{\Bbbk}(B)$ and $\kl_{B},
\kr_{B}: B\rightarrow\End_{\Bbbk}(A)$ are linear maps.
If the following conditions are satisfied:
\begin{enumerate}\itemsep=0pt
\item[$(i)$] $(A, \kl_{B}, \kr_{B}, \alpha)$ is a bimodule over averaging algebra
     $(B, \beta)$;
\item[$(ii)$] $(B, \kl_{A}, \kr_{A}, \beta)$ is a bimodule over averaging algebra
     $(A, \alpha)$;
\item[$(iii)$] $(A, B, \kl_{A}, \kr_{A}, \kl_{B}, \kr_{B})$ is a matched pair of
     associative algebras,
\end{enumerate}
then $((A, \alpha), (B, \beta), \kl_{A}, \kr_{A}, \kl_{B}, \kr_{B})$ is called a
{\rm matched pair of averaging algebras}.
\end{defi}

\begin{pro}\label{pro:matda}
Let $(A, \alpha)$ and $(B, \beta)$ be two averaging algebras.
Suppose that $(A, B, \kl_{A}, \kr_{A}, \kl_{B},$ $ \kr_{B})$ is a matched pair
of associative algebras. Then $(A\bowtie B, \alpha\oplus\beta)$ is an averaging
algebra if and only if $((A, \alpha), (B, \beta), \kl_{A}, \kr_{A}, \kl_{B},
\kr_{B})$ is a matched pair of averaging algebras.
Further, any averaging algebra whose underlying vector space is the linear direct sum
of two averaging subalgebras is obtained from a matched pair of these
two averaging subalgebras.
\end{pro}

\begin{proof}
Suppose that $((A, \alpha), (B, \beta), \kl_{A}, \kr_{A}, \kl_{B}, \kr_{B})$ is a
matched pair of averaging algebras. Then, for any $a_{1}, a_{2}\in A$,
$b_{1}, b_{2}\in B$, we have
\begin{align*}
(\alpha\oplus\beta)(a_{1}, b_{1})\ast(\alpha\oplus\beta)(a_{2}, b_{2})
&=\Big(\alpha(a_{1})\alpha(a_{2})+\kl_{B}(\beta(b_{1}))(\alpha(a_{2}))
+\kr_{B}(\beta(b_{2}))(\alpha(a_{1})),\\[-2mm] &\qquad\beta(b_{1})\beta(b_{2})
+\kl_{A}(\alpha(a_{1}))(\beta(b_{2}))+\kr_{A}(\alpha(a_{2}))(\beta(b_{1}))\Big),\\
(\alpha\oplus\beta)((\alpha\oplus\beta)(a_{1}, b_{1})\ast(a_{2}, b_{2}))
&=\Big(\alpha(\alpha(a_{1})a_{2})+\alpha(\kl_{B}(\beta(b_{1}))(a_{2}))
+\alpha(\kr_{B}(b_{2})(\alpha(a_{1}))),\\[-2mm]  &\qquad  \beta(\beta(b_{1})b_{2})
+\beta(\kl_{A}(\alpha(a_{1}))(b_{2}))+\beta(\kr_{A}(a_{2})(\beta(b_{1})))\Big),\\
(\alpha\oplus\beta)((a_{1}, b_{1})\ast(\alpha\oplus\beta)(a_{2}, b_{2}))
&=\Big(\alpha(a_{1}\alpha(a_{2}))+\alpha(\kl_{B}(b_{1})(\alpha(a_{2})))
+\alpha(\kr_{B}(\beta(b_{2}))(a_{1})),\\[-2mm]  &\qquad  \beta(b_{1}\beta(b_{2}))
+\beta(\kl_{A}(a_{1})(\beta(b_{2})))+\beta(\kr_{A}(\alpha(a_{2}))(b_{1}))\Big).
\end{align*}
Since $(A, \kl_{B}, \kr_{B}, \alpha)$ is a bimodule over $(B, \beta)$ and
$(B, \kl_{A}, \kr_{A}, \beta)$ is a bimodule over $(A, \alpha)$, we get
$(\alpha\oplus\beta)(a_{1}, b_{1})\ast(\alpha\oplus\beta)(a_{2}, b_{2})
=(\alpha\oplus\beta)((\alpha\oplus\beta)(a_{1}, b_{1})\ast(a_{2}, b_{2}))
=(\alpha\oplus\beta)((a_{1}, b_{1})\ast(\alpha\oplus\beta)(a_{2}, b_{2}))$. Thus,
$\alpha\oplus\beta$ is an averaging operator on $A\bowtie B$, and so that,
$(A\bowtie B, \alpha\oplus\beta)$ is an averaging algebra.

Conversely, if $\alpha\oplus\beta$ is an averaging operator on $A\bowtie B$,
i.e., $(\alpha\oplus\beta)(a_{1}, b_{1})\ast(\alpha\oplus\beta)(a_{2}, b_{2})
=(\alpha\oplus\beta)((\alpha\oplus\beta)(a_{1}, b_{1})\ast(a_{2}, b_{2}))
=(\alpha\oplus\beta)((a_{1}, b_{1})\ast(\alpha\oplus\beta)(a_{2}, b_{2}))$, for any
$a_{1}, a_{2}\in A$ and $b_{1}, b_{2}\in B$. From the above calculation,
taking $a_{1}=b_{2}=0$ and $a_{2}=b_{1}=0$ in the above equation respectively, we get
that $(A, \kl_{B}, \kr_{B}, \alpha)$ is a bimodule over $(B, \beta)$ and
$(B, \kl_{A}, \kr_{A}, \beta)$ is a bimodule over $(A, \alpha)$. Hence,
$((A, \alpha), (B, \beta), \kl_{A}, \kr_{A}, \kl_{B}, \kr_{B})$ is a matched pair
of averaging algebras. Finally, the second part follows straightforwardly.
\end{proof}

\subsection{Double constructions of averaging Frobenius algebras}\label{subsec:double}
We recall the concept of a double construction of Frobenius algebra \cite{Bai}.

\begin{defi}\label{def:form}
Let $\mathfrak{B}(-,-)$ be a bilinear form on an associative algebra $A$.
\begin{enumerate}\itemsep=0pt
\item[-] $\mathfrak{B}(-,-)$ is called {\rm nondegenerate} if $\mathfrak{B}(a_{1},\;
        a_{2})=0$ for any $a_{2}\in A$, then $a_{1}=0$;
\item[-] $\mathfrak{B}(-,-)$ is called {\rm invariant} if $\mathfrak{B}(a_{1}a_{2},\; a_{3})
        =\mathfrak{B}(a_{1},\; a_{2}a_{3})$, for any $a_{1}, a_{2}, a_{3}\in A$;
\item[-] $\mathfrak{B}(-,-)$ is called {\rm symmetric} if $\mathfrak{B}(a_{1},\; a_{2})
        =\mathfrak{B}(a_{2},\; a_{1})$, for any $a_{1}, a_{2}\in A$.
\end{enumerate}
A {\rm Frobenius algebra} $(A, \mathfrak{B})$ is an associative algebra $A$ with a
nondegenerate invariant bilinear form $\mathfrak{B}(-,-)$. A Frobenius algebra
$(A, \mathfrak{B})$ is called {\rm symmetric} if $\mathfrak{B}(-,-)$ is symmetric.
\end{defi}

Let $A$ be an associative algebra. Suppose that there is an associative algebra structure
$\cdot$ on its dual space $A^{\ast}$ and an associative algebra structure $\ast$ on the
direct sum $A\oplus A^{\ast}$ of the underlying vector spaces $A$ and $A^{\ast}$,
which contains both $A$ and $A^{\ast}$ as subalgebras. Then the associative algebra is
just the associative algebra $A\bowtie A^{\ast}$, corresponding to the matched pair
$(A, A^{\ast}, \fr_{A}^{\ast}, \fl_{A}^{\ast}, \fr_{A^{\ast}}^{\ast},
\fl_{A^{\ast}}^{\ast})$. Define a bilinear form on $A\oplus A^{\ast}$ by
\begin{align*}
\mathfrak{B}_{d}\Big((a_{1},\xi_{1}),\; (a_{2}, \xi_{2})\big)
:=\langle\xi_{2},\; a_{1}\rangle+\langle\xi_{1},\; a_{2}\rangle,
\end{align*}
for any $a_{1}, a_{2}\in A$ and $\xi_{1}, \xi_{2}\in A^{\ast}$.
If $(A \oplus A^{\ast}, \mathfrak{B}_{d})$ is a symmetric Frobenius algebra,
then it is called a {\rm double construction of Frobenius algebra associated to $A$ and
$A^{\ast}$}, and denoted by $(A\bowtie A^{\ast}, \mathfrak{B}_{d})$.

We extend these notions to the context of averaging algebras.

\begin{defi}\label{def:aver-Frob}
An {\rm averaging Frobenius algebra} is a triple $(A, \alpha, \mathfrak{B})$,
where $(A, \alpha)$ is an averaging algebra and $(A, \mathfrak{B})$ is a Frobenius algebra.
It is called {\rm symmetric} if the bilinear form $\mathfrak{B}(-,-)$ is symmetric.
A linear map $\hat{\alpha}: A\rightarrow A$ is called {\rm the adjoint linear operator
of $\alpha$ under the nondegenerate bilinear form $\mathfrak{B}(-,-)$}, if
$$
\mathfrak{B}(\alpha(a_{1}),\; a_{2})=\mathfrak{B}(a_{1},\; \hat{\alpha}(a_{2})), \qquad
$$
for any $a_{1}, a_{2}\in A$.
\end{defi}

\begin{pro}\label{pro:Fro-mod}
Let $(A, \alpha, \mathfrak{B})$ be a symmetric averaging Frobenius algebra,
and $\hat{\alpha}$ be the adjoint of $\alpha$ with respect to $\mathfrak{B}(-,-)$.
Then, $(A^{\ast}, \fr_{A}^{\ast}, \fl_{A}^{\ast}, \hat{\alpha}^{\ast})$ is
a bimodule over the averaging algebra $(A, \alpha)$, and as bimodules over
$(A, \alpha)$, $(A, \fl_{A}, \fr_{A}, \alpha)$ and $(A^{\ast}, \fr_{A}^{\ast},
\fl_{A}^{\ast}, \hat{\alpha}^{\ast})$ are isomorphic.
Moreover, let $(A, \alpha)$ be an averaging algebra, and $\beta: A\rightarrow A$ be
a linear map. If $(A^{\ast}, \fr_{A}^{\ast}, \fl_{A}^{\ast}, \beta^{\ast})$ is
a bimodule over $(A, \alpha)$, and it is isomorphic to the regular bimodule
$(A, \fl_{A}, \fr_{A}, \alpha)$, then there exists a nondegenerate invariant
bilinear form $\mathfrak{B}$ on $A$ such that $\beta=\hat{\alpha}$.
\end{pro}

\begin{proof}
First, suppose that $(A, \alpha, \mathfrak{B})$ is a symmetric averaging Frobenius algebra.
For any $a_{1}, a_{2}$, $a_{3}\in A$, note that
\begin{align*}
\mathfrak{B}(\alpha(a_{1})\hat{\alpha}(a_{2}),\; a_{3})
&=\mathfrak{B}(\alpha(a_{3}\alpha(a_{1})),\; a_{2}),\\
\mathfrak{B}(\hat{\alpha}(\alpha(a_{1})a_{2}),\; a_{3})
&=\mathfrak{B}(\alpha(a_{3})\alpha(a_{1}),\; a_{2}),\\
\mathfrak{B}(\hat{\alpha}(a_{1}\hat{\alpha}(a_{2})),\; a_{3})
&=\mathfrak{B}(\alpha(\alpha(a_{3})a_{1}),\; a_{2}),
\end{align*}
we get $\fl(\alpha(a_{1}))(\hat{\alpha}(a_{2}))=\hat{\alpha}(\fl(\alpha(a_{1}))(a_{2}))
=\hat{\alpha}(\fl(a_{1})(\hat{\alpha}(a_{2})))$. Similarly, we have
$\fr(\alpha(a_{1}))(\hat{\alpha}(a_{2}))=$ $\hat{\alpha}(\fr(\alpha(a_{1}))(a_{2}))
=\hat{\alpha}(\fr(a_{1})(\hat{\alpha}(a_{2})))$. Thus, $(A, \fl_{A}, \fr_{A},
\hat{\alpha})$ is a bimodule over $(A, \alpha)$. By Proposition \ref{pro:dual},
we get $(A^{\ast}, \fr_{A}^{\ast}, \fl_{A}^{\ast}, \hat{\alpha}^{\ast})$ is
a bimodule over $(A, \alpha)$.

Define a linear map $\varphi: A\rightarrow A^{\ast}$ by
$$
\varphi(a_{1})(a_{2})=\mathfrak{B}(a_{1}, a_{2}),
$$
for any $a_{1}, a_{2}\in A$. Then, $\varphi$ is a linear isomorphism.
Moreover, for any $a_{1}, a_{2}, a_{3}\in A$, we have
\begin{align*}
\langle\varphi(\fl_{A}(a_{1})(a_{2})),\; a_{3}\rangle&=\mathfrak{B}(a_{1}a_{2},\; a_{3})
=\langle\varphi(a_{2}),\; a_{3}a_{1}\rangle
=\langle\fr_{A}^{\ast}(a_{1})(\varphi(a_{2})),\; a_{3}\rangle, \\
\langle\varphi(\fr_{A}(a_{1})(a_{2})),\; a_{3}\rangle&=\mathfrak{B}(a_{2}a_{1},\; a_{3})
=\langle\varphi(a_{2}),\; a_{1}a_{3}\rangle
=\langle\fl_{A}^{\ast}(a_{1})(\varphi(a_{2})),\; a_{3}\rangle, \\
\langle\varphi(\alpha(a_{1})),\; a_{2}\rangle&=\mathfrak{B}(\alpha(a_{1}),\; a_{2})
=\mathfrak{B}(a_{1},\; \hat{\alpha}(a_{2}))
=\langle\hat{\alpha}^{\ast}(\varphi(a_{1})),\; a_{2}\rangle.
\end{align*}
Hence, $\varphi$ is an isomorphism.

Second, suppose that $\varphi: A\rightarrow A^{\ast}$ is the isomorphism
from $(A, \fl_{A}, \fr_{A}, \alpha)$ to $(A^{\ast}, \fr_{A}^{\ast}, \fl_{A}^{\ast},
\beta^{\ast})$. Define a bilinear form $\mathfrak{B}(-,-)$ on $A$ by
$$
\mathfrak{B}(a_{1}, a_{2}):=\langle\varphi(a_{1}),\; a_{2}\rangle,
$$
for any $a_{1}, a_{2}\in A$.
Then by a similar argument as above, we show that $\mathfrak{B}(-,-)$ is a nondegenerate
invariant bilinear form on $A$ such that $\beta=\hat{\alpha}$.
\end{proof}

\begin{defi}\label{def:doucon}
Let $(A, \cdot, \alpha)$ be an averaging algebra. Suppose that there is a linear map
$\beta: A\rightarrow A$ and a bilinear map $\cdot': A^{\ast}\otimes A^{\ast}\rightarrow
A^{\ast}$ such that $(A^{\ast}, \cdot', \beta^{\ast})$ is an averaging algebra.
A {\rm double construction of averaging Frobenius algebra associated
to $(A, \cdot, \alpha)$ and $(A^{\ast}, \cdot', \beta^{\ast})$} is a double
construction of Frobenius algebra $(A \bowtie A^{\ast}, \mathfrak{B}_d)$
associated to $A$ and $A^{\ast}$ such that $(A\bowtie A^{\ast},
\alpha\oplus\beta^{\ast})$ is an averaging algebra, which is denoted by
$(A\bowtie A^{\ast}, \alpha\oplus\beta^{\ast}, \mathfrak{B}_{d})$.
\end{defi}

\begin{lem}\label{lem:douadm}
Let $(A\bowtie A^{\ast}, \alpha\oplus\beta^{\ast}, \mathfrak{B}_{d})$ be a double
construction of averaging Frobenius algebra associated to $(A, \cdot, \alpha)$
and $(A^{\ast}, \cdot' \beta^{\ast})$. Then,
\begin{enumerate}\itemsep=0pt
\item[$(i)$] The adjoint $\widehat{\alpha\oplus\beta^{\ast}}$ of $\alpha\oplus\beta^{\ast}$
     with respect to $\mathfrak{B}_{d}(-,-)$ is $\beta\oplus\alpha^{\ast}$, and $(A\oplus
     A^{\ast}, \fl_{A\bowtie A^{\ast}}, \fr_{A\bowtie A^{\ast}}, \beta\oplus\alpha^{\ast})$
     is a bimodule over $(A\bowtie A^{\ast}, \alpha\oplus\beta^{\ast})$;
\item[$(ii)$] $(A, \fl_{A}, \fr_{A}, \beta)$ is a bimodule over $(A, \alpha)$;
\item[$(iii)$] $(A^{\ast}, \fl_{A^{\ast}}, \fr_{A^{\ast}}, \alpha^{\ast})$ is a
     bimodule over $(A^{\ast}, \beta^{\ast})$.
\end{enumerate}
\end{lem}

\begin{proof}
$(i)$ For any $a_{1}, a_{2}\in A$ and $\xi_{1}, \xi_{2}\in A^{\ast}$, we have
\begin{align*}
\mathfrak{B}_{d}((\alpha\oplus\beta^{\ast})(a_{1}, \xi_{1}),\; (a_{2}, \xi_{2}))
&=\langle\alpha(a_{1}),\; \xi_{2}\rangle+\langle a_{2},\; \beta^{\ast}(\xi_{1})\rangle\\
&=\langle a_{1},\; \alpha^{\ast}(\xi_{2})\rangle+\langle\beta(a_{2}),\; \xi_{1}\rangle
=\mathfrak{B}_{d}((a_{1}, \xi_{1}),\; (\beta\oplus\alpha^{\ast})(a_{2}, \xi_{2})).
\end{align*}
That is to say, the adjoint of $\alpha\oplus\beta^{\ast}$ with respect to
$\mathfrak{B}_{d}$ is $\beta\oplus\alpha^{\ast}$. Moreover, by Proposition \ref{pro:Fro-mod},
we get $(A\oplus A^{\ast}, \fl_{A\bowtie A^{\ast}}, \fr_{A\bowtie A^{\ast}},
\beta\oplus\alpha^{\ast})$ is a bimodule over $(A\bowtie A^{\ast}, \alpha\oplus\beta^{\ast})$.

$(ii)$ and $(iii)$ Since $(A\oplus A^{\ast}, \fl_{A\bowtie A^{\ast}}, \fr_{A\bowtie A^{\ast}},
\beta\oplus\alpha^{\ast})$ is a bimodule over $(A\bowtie A^{\ast}, \alpha\oplus\beta^{\ast})$,
for any $a_{1}, a_{2}\in A$ and $\xi_{1}, \xi_{2}\in A^{\ast}$, we have
\begin{align*}
(\alpha(a_{1}), \beta^{\ast}(\xi_{1}))\ast(\beta(a_{2}), \alpha^{\ast}(\xi_{2}))
&=(\beta\oplus\alpha^{\ast})\big((\alpha(a_{1}), \beta^{\ast}(\xi_{1}))
\ast(a_{2}, \xi_{2})\big)\\
&=(\beta\oplus\alpha^{\ast})\big((a_{1}, \xi_{1})
\ast(\beta(a_{2}), \alpha^{\ast}(\xi_{2})\big), \\
(\beta(a_{2}), \alpha^{\ast}(\xi_{2}))\ast(\alpha(a_{1}), \beta^{\ast}(\xi_{1}))
&=(\beta\oplus\alpha^{\ast})\big((a_{2}, \xi_{2})
\ast(\alpha(a_{1}), \beta^{\ast}(\xi_{1})\big) \\
&=(\beta\oplus\alpha^{\ast})\big((\beta(a_{2}), \alpha^{\ast}(\xi_{2}))
\ast(a_{1}, \xi_{1})\big).
\end{align*}
Taking $\xi_{1}=\xi_{2}=0$ in the above equations,
we get that $(A, \fl_{A}, \fr_{A}, \beta)$ is a bimodule over $(A, \alpha)$,
and taking $a_{1}=a_{2}=0$, we get that $(A^{\ast}, \fl_{A^{\ast}}, \fr_{A^{\ast}},
\alpha^{\ast})$ is a bimodule over $(A^{\ast}, \beta^{\ast})$.
\end{proof}

\begin{pro}\label{pro:dou-mat}
Let $(A, \cdot, \alpha)$ be an averaging algebra. Suppose that there is a linear map
$\beta: A\rightarrow A$ and a bilinear map $\cdot': A^{\ast}\otimes A^{\ast}\rightarrow
A^{\ast}$ such that $(A^{\ast}, \cdot', \beta^{\ast})$ is an averaging algebra.
Then there is a double construction of averaging Frobenius algebra
$(A\bowtie A^{\ast}, \alpha\oplus\beta^{\ast}, \mathfrak{B}_{d})$ associated to
$(A, \alpha)$ and $(A^{\ast}, \beta^{\ast})$ if and only if
$((A, \alpha), (A^{\ast}, \beta^{\ast}), \fr_{A}^{\ast}, \fl_{A}^{\ast},
\fr_{A^{\ast}}^{\ast}$, $\fl_{A^{\ast}}^{\ast})$ is a matched pair of averaging algebras.
\end{pro}

\begin{proof}
If $(A\bowtie A^{\ast}, \mathfrak{B}_{d})$ is a double construction of averaging Frobenius
algebra associated to $(A, \alpha)$ and $(A^{\ast}, \beta^{\ast})$, by
\cite[Theorem 2.2.1]{Bai}, $(A, A^{\ast}, \fr_{A}^{\ast}, \fl_{A}^{\ast},
\fr_{A^{\ast}}^{\ast}, \fl_{A^{\ast}}^{\ast})$ is a matched pair of associative algebras.
And by Lemma \ref{lem:douadm}, $(A^{\ast}, \fr_{A}^{\ast}, \fl_{A}^{\ast}, \beta^{\ast})$
is a bimodule over $(A, \alpha)$ and $(A, \fr_{A^{\ast}}^{\ast}, \fl_{A^{\ast}}^{\ast},
\alpha)$ is a bimodule over $(A^{\ast}, \beta^{\ast})$, respectively.
Hence $((A, \alpha), (A^{\ast}, \beta^{\ast}), \fr_{A}^{\ast}, \fl_{A}^{\ast},
\fr_{A^{\ast}}^{\ast}, \fl_{A^{\ast}}^{\ast})$ is a matched pair of averaging algebras.

Conversely, if $((A, \alpha), (A^{\ast}, \beta^{\ast}), \fr_{A}^{\ast}, \fl_{A}^{\ast},
\fr_{A^{\ast}}^{\ast}, \fl_{A^{\ast}}^{\ast})$ is a matched pair of averaging algebras,
by \cite[Theorem 2.2.1]{Bai} again, $(A\bowtie A^{\ast}, \mathfrak{B}_{d})$ is a
Frobenius algebra. Moreover by Proposition \ref{pro:matda}, $(A\bowtie A^{\ast},
\alpha\oplus\beta^{\ast})$ is an averaging algebra.
Hence $(A\bowtie A^{\ast}, \alpha\oplus\beta^{\ast}, \mathfrak{B}_{d})$ is a
double construction of averaging Frobenius algebra associated to $(A, \alpha)$
and $(A^{\ast}, \beta^{\ast})$.
\end{proof}

\subsection{Averaging antisymmetric infinitesimal bialgebras}\label{subsec:bialg}
Recall that a {\it coassociative coalgebra} $(A, \Delta)$ is a vector space $A$
with a linear map $\Delta: A\rightarrow A\otimes A$ satisfying the coassociative law:
$$
(\Delta\otimes\id)\Delta=(\id\otimes\Delta)\Delta.
$$
A coassociative coalgebra $(A, \Delta)$ is called {\it cocommutative} if
$\Delta=\tau\Delta$, where $\tau: A\otimes A\rightarrow A\otimes A$ is the
flip operator defined by $\tau(a_{1}\otimes a_{2}):=a_{2}\otimes a_{1}$ for
all $a_{1}, a_{2}\in A$.

\begin{defi}[\cite{Bai}]\label{def:bialg}
An {\rm antisymmetric infinitesimal bialgebra} or simply an {\rm ASI bialgebra} is
a triple $(A, \cdot, \Delta)$ consisting of a vector space $A$ and linear maps
$\cdot: A\otimes A\rightarrow A$ and $\Delta: A\rightarrow A\otimes A$ such that
\begin{enumerate}\itemsep=0pt
\item[$(i)$] $(A, \cdot)$ is an associative algebra;
\item[$(ii)$] $(A, \Delta)$ is a coassociative coalgebra;
\item[$(iii)$] for any $a_{1}, a_{2}\in A$,
\begin{align}
&\qquad\qquad\Delta(a_{1}a_{2})=(\fr_{A}(a_{2})\otimes\id)\Delta(a_{1})
+(\id\otimes\,\fl_{A}(a_{1}))\Delta(a_{2}),                          \label{bialg1} \\
&\big(\fl_{A}(a_{1})\otimes\id-\id\otimes\,\fr_{A}(a_{1})\big)\Delta(a_{2})
=\tau\big(\big(\id\otimes\,\fr_{A}(a_{2})
-\fl_{A}(a_{2})\otimes\id\big)\Delta(a_{1})\big).                  \label{bialg2}
\end{align}
\end{enumerate}
\end{defi}

\begin{defi}\label{def:coderivation}
Let $(A, \Delta)$ be a coassociative coalgebra.
A linear map $\beta: A\rightarrow A$ is called an {\rm averaging operator} on
$(A, \Delta)$ if $(\beta\otimes\beta)\Delta=(\beta\otimes\id)\Delta\beta
=(\id\otimes\beta)\Delta\beta$.

An {\rm averaging coalgebra} is a triple $(A, \Delta, \beta)$,
consisting of a coassociative coalgebra $(A, \Delta)$ and an averaging operator
$\beta: A\rightarrow A$. An averaging coalgebra $(A, \Delta, \beta)$ is called
{\rm cocommutative} if $(A, \Delta)$ is cocommutative.
\end{defi}

The notion of an averaging coalgebra is the dualization of the notion of an
averaging algebra, that is, $(A, \Delta, \beta)$ is an (cocommutative) averaging
coalgebra if and only if $(A^{\ast}, \Delta^{\ast}, \beta^{\ast})$ is an
(commutative) averaging algebra, where $\langle\Delta^{\ast}(\xi_{1}, \xi_{2}),\; a\rangle
=\langle\xi_{1}\otimes\xi_{2},\; \Delta(a)\rangle=\sum_{(a)}\xi_{1}(a_{(1)})\xi_{1}(a_{(1)})$
for any $\xi_{1}, \xi_{2}\in A^{\ast}$ and $a\in A$, if $\Delta(a)=\sum_{(a)}
a_{(1)}\otimes a_{(2)}$.

\begin{defi}\label{def:avebialg}
An {\rm averaging antisymmetric infinitesimal bialgebra} or simply an {\rm averaging
ASI bialgebra} is a quadruple $(A, \Delta, \alpha, \beta)$ satisfying
\begin{enumerate}\itemsep=0pt
\item[-] $(A, \cdot, \alpha)$ is an averaging algebra;
\item[-] $(A, \Delta, \beta)$ is an averaging coalgebra;
\item[-] $(A, \cdot, \Delta)$ is an ASI bialgebra;
\item[-] $(A, \fl_{A}, \fr_{A}, \beta)$ is a bimodule over $(A, \cdot,$ $ \alpha)$, and
     $(A^{\ast}, \fl_{A^{\ast}}, \fr_{A^{\ast}}, \alpha^{\ast})$ is a bimodule over
     $(A^{\ast}, \Delta^{\ast}, $ $ \beta^{\ast})$.
\end{enumerate}
An averaging ASI bialgebra $(A, \Delta, \alpha, \beta)$ is called {\rm commutative
and cocommutative} if $A$ is a commutative associative algebra and $(A, \Delta)$ is
a cocommutative coassociative coalgebra.
\end{defi}

\begin{ex}\label{ex:bialg}
$(i)$ Let $(A, \alpha)$ be the 2-dimensional averaging algebra considered in
Example \ref{ex:2dim}, that is, $A=\Bbbk\{e_{1}, e_{2}\}$, that non-zero product is
given by $e_{1}e_{2}=e_{1}$, $e_{2}e_{2}=e_{2}$, and the averaging operator is given by
$\alpha(e_{1})=0$ and $\alpha(e_{2})=e_{2}$. Now we define a comultiplication on $A$
by $\Delta(e_{1})=e_{1}\otimes e_{1}$ and $\Delta(e_{2})=e_{2}\otimes e_{1}$, and define
linear map $\beta: A\rightarrow A$ by $\beta(e_{1})=e_{2}$ and $\beta(e_{2})=0$. Then
we get an averaging ASI bialgebra $(A, \Delta, \alpha, \beta)$.

$(ii)$ Consider the 2-dimensional associative algebra $(A=\Bbbk\{e_{1}, e_{2}\}, \cdot)$,
where the non-zero product is given by $e_{1}e_{1}=e_{1}$, $e_{1}e_{2}=e_{2}=e_{2}e_{1}$.
Define a linear map $\alpha: A\rightarrow A$ by $\alpha(e_{1})=e_{1}$, $\alpha(e_{2})=0$.
Then $(A, \alpha)$ is a commutative averaging algebra. Define a comultiplication on $A$
by $\Delta(e_{1})=0$ and $\Delta(e_{2})=e_{2}\otimes e_{2}$, and define
linear map $\beta: A\rightarrow A$ by $\beta(e_{1})=0$ and $\beta(e_{2})=e_{2}$. Then
we get a commutative and cocommutative averaging ASI bialgebra $(A, \Delta, \alpha, \beta)$.
\end{ex}

Since the multiplication $\Delta^{\ast}$ of averaging algebra $(A^{\ast}, \beta^{\ast})$ is
the dual of comultiplication $\Delta$, we have the following lemma.

\begin{lem}\label{lem:dual-mod}
Let $(A, \Delta, \beta)$ be an averaging coalgebra, and $\alpha: A\rightarrow A$
be a linear map. Then $(A^{\ast}, \fl_{A^{\ast}}, \fr_{A^{\ast}}, \alpha^{\ast})$
is a bimodule over $(A^{\ast}, \Delta^{\ast}, \beta^{\ast})$ if and only if for any
$$
(\beta\otimes\alpha)\Delta=(\beta\otimes\id)\Delta\alpha
=(\id\otimes\,\alpha)\Delta\alpha,\qquad\quad
(\alpha\otimes\beta)\Delta=(\id\otimes\,\beta)\Delta\alpha=(\alpha\otimes\id)\Delta\alpha.
$$
\end{lem}

\begin{proof}
First, by the definition, $(A^{\ast}, \fl_{A^{\ast}}, \fr_{A^{\ast}}, \alpha^{\ast})$
is a bimodule over $(A^{\ast}, \beta^{\ast})$ if and only if
\begin{align*}
\beta^{\ast}(\xi_{1})\cdot_{A^{\ast}}\alpha^{\ast}(\xi_{2})
&=\alpha^{\ast}(\beta^{\ast}(\xi_{1})\cdot_{A^{\ast}}\xi_{2})
=\alpha^{\ast}(\xi_{1}\cdot_{A^{\ast}}\alpha^{\ast}(\xi_{2})), \\
\alpha^{\ast}(\xi_{2})\cdot_{A^{\ast}}\beta^{\ast}(\xi_{1})
&=\alpha^{\ast}(\xi_{2}\cdot_{A^{\ast}}\beta^{\ast}(\xi_{1}))
=\alpha^{\ast}(\alpha^{\ast}(\xi_{2})\cdot_{A^{\ast}}\xi_{1}),
\end{align*}
for any $\xi_{1}, \xi_{2}\in A^{\ast}$. Rewriting the above equations in
terms of the comultiplication, we get this lemma.
\end{proof}

Let $(A, \cdot, \Delta)$ be an ASI bialgebra. A pair $(\alpha, \beta)$
in Definition \ref{def:avebialg} is called {\it a pair of averaging operators}
on $(A, \cdot, \Delta)$. A linear map $\alpha: A\rightarrow A$
is called an {\it averaging operator} on $(A, \cdot, \Delta)$, if $\alpha$ is both
an averaging operator on associative algebra $A$ and an averaging operator on
coassociative coalgebra $(A, \Delta)$. Clearly, $(\alpha, \alpha)$ is a pair of
averaging operators on ASI bialgebra $(A, \cdot, \Delta)$, and so that,
$(A, \Delta, \alpha, \alpha)$ is an averaging ASI bialgebra, if $\alpha$ is an
averaging operator on ASI bialgebra $(A, \cdot, \Delta)$.

\begin{pro}\label{pro:bi-mat}
Let $(A, \alpha)$ be an averaging algebra. Suppose that there are linear maps $\beta:
A\rightarrow A$ and $\Delta: A\rightarrow A\otimes A$ such that $(A, \Delta, \beta)$
is an averaging coalgebra. Then the quadruple $(A, \Delta, \alpha, \beta)$ is an averaging
ASI bialgebra if and only if $((A, \alpha), (A^{\ast}, \beta^{\ast}), \fr_{A}^{\ast},
\fl_{A}^{\ast}, \fr_{A^{\ast}}^{\ast}, \fl_{A^{\ast}}^{\ast})$ is a matched pair of
averaging algebras, where $(A^{\ast}, \beta^{\ast})$ is the dual algebra of
$(A, \Delta, \beta)$.
\end{pro}

\begin{proof}
If the quadruple $(A, \Delta, \alpha, \beta)$ is an averaging ASI bialgebra,
then $(A, \cdot, \Delta)$ is an ASI bialgebra, and so that $(A, A^{\ast},
\fr_{A}^{\ast}, \fl_{A}^{\ast}, \fr_{A^{\ast}}^{\ast}, \fl_{A^{\ast}}^{\ast})$ is a
matched pair of associative algebras. Moreover, by the definition of averaging ASI
bialgebra again, we get $(A^{\ast}, \fr_{A}^{\ast}, \fl_{A}^{\ast}, \beta^{\ast})$
is a bimodule over $(A, \alpha)$, and $(A, \fr_{A^{\ast}}^{\ast}, \fl_{A^{\ast}}^{\ast},
\alpha)$ is a bimodule over $(A^{\ast}, \beta^{\ast})$. Hence, $((A, \alpha),
(A^{\ast}, \beta^{\ast}), \fr_{A}^{\ast}, \fl_{A}^{\ast}, \fr_{A^{\ast}}^{\ast},
\fl_{A^{\ast}}^{\ast})$ is a matched pair of averaging algebras.

Conversely, if $((A, \alpha), (A^{\ast}, \beta^{\ast}), \fr_{A}^{\ast}, \fl_{A}^{\ast},
\fr_{A^{\ast}}^{\ast}, \fl_{A^{\ast}}^{\ast})$ is a matched pair of averaging algebras,
then $(A, A^{\ast}$, $\fr_{A}^{\ast}, \fl_{A}^{\ast}, \fr_{A^{\ast}}^{\ast},
\fl_{A^{\ast}}^{\ast})$ is a matched pair of associative algebras, and so that,
$(A, \cdot, \Delta)$ is an ASI bialgebra. Moreover, by the definition of matched pair of
averaging algebras again, we get $(A^{\ast}, \fr_{A}^{\ast}, \fl_{A}^{\ast}, \beta^{\ast})$
is a bimodule over $(A, \alpha)$, and $(A, \fr_{A^{\ast}}^{\ast}, \fl_{A^{\ast}}^{\ast},
\alpha)$ is a bimodule over $(A^{\ast}, \beta^{\ast})$. Thus,
$(A, \Delta, \alpha, \beta)$ is an averaging ASI bialgebra.
\end{proof}

Combining Propositions \ref{pro:dou-mat} and \ref{pro:bi-mat},
we have the following conclusion.

\begin{thm}\label{thm:equ}
Let $(A, \cdot, \alpha)$ be an averaging algebra. Suppose that there is a linear map
$\beta: A\rightarrow A$ and a bilinear map $\cdot': A^{\ast}\otimes A^{\ast}\rightarrow
A^{\ast}$ such that $(A^{\ast}, \cdot', \beta^{\ast})$ is an averaging algebra.
Let $\Delta: A\rightarrow A\otimes A$ denote the linear dual of the multiplication
on $A^{\ast}$. Then the following conditions are equivalent:
\begin{enumerate}\itemsep=0pt
\item[$(i)$] There is a double construction of averaging Frobenius algebra associated to
     $(A, \cdot, \alpha)$ and $(A^{\ast}, \cdot', \beta^{\ast})$;
\item[$(ii)$] $((A, \alpha), (A^{\ast}, \beta^{\ast}), \fr_{A}^{\ast}, \fl_{A}^{\ast},
     \fr_{A^{\ast}}^{\ast}, \fl_{A^{\ast}}^{\ast})$ is a matched pair of averaging algebras;
\item[$(iii)$] $(A, \Delta, \alpha, \beta)$ is an averaging ASI bialgebra.
\end{enumerate}
\end{thm}

\section{Coboundary averaging antisymmetric infinitesimal bialgebras}\label{sec:cob}
In this section, we study the coboundary averaging ASI bialgebras, and
introduce the notion of Yang-Baxter equation in an averaging algebra. The notions
of $\mathcal{O}$-operators of averaging algebras and averaging dendriform algebras
are introduced to provide antisymmetric solutions of Yang-Baxter equation in
semidirect product averaging algebras and hence give rise to averaging ASI bialgebras.

\subsection{Coboundary averaging ASI bialgebras and Yang-Baxter equation}\label{subsec:cob}
An averaging ASI bialgebra is called coboundary if it as an ASI bialgebra is coboundary.

\begin{defi}\label{def:cob}
An averaging ASI bialgebra $(A, \Delta, \alpha, \beta)$ is called {\rm coboundary}
if there exists an element $r\in A\otimes A$, such that
\begin{align}
\Delta(a):=(\id\otimes\,\fl_{A}(a)-\fr_{A}(a)\otimes\id)(r),          \label{cobo}
\end{align}
for any $a\in A$. In this case, we also call $(A, \Delta, \alpha, \beta)$ is
{\rm an averaging ASI bialgebra induced by $r$}.
\end{defi}

\begin{pro}[{\cite[Theorem 2.3.5]{Bai}}]\label{pro:coba}
Let $A$ be an associative algebra and $r\in A\otimes A$.
Define a linear map $\Delta: A\rightarrow A\otimes A$ by Eq. \eqref{cobo}.
Then $(A, \cdot, \Delta)$ is an ASI bialgebra if and only if for any $a_{1}, a_{2}\in A$,
\begin{align}
&\big(\fl_{A}(a_{1})\otimes\id-\id\otimes\,\fr_{A}(a_{1})\big)
\big(\id\otimes\,\fl_{A}(a_{2})-\fr_{A}(a_{2})\otimes\id\big)
\big(r+\tau(r)\big)=0,                                        \label{coba1} \\
&\;\;\big(\id\otimes\id\otimes\,\fl_{A}(a_{1})-\fr_{A}(a_{1})\otimes\id\otimes\id\big)
(r_{12}r_{13}+r_{13}r_{23}-r_{23}r_{12})=0.                         \label{coba2}
\end{align}
Here for $r=\sum_{i}x_{i}\otimes y_{i}\in A\otimes A$, we denote
$r_{12}r_{13}=\sum_{i,j}x_{i}x_{j}\otimes y_{i}\otimes y_{j}$,
$r_{13}r_{23}=\sum_{i,j}x_{i}\otimes x_{j}\otimes y_{i}y_{j}$ and
$r_{23}r_{12}=\sum_{i,j}x_{j}\otimes x_{i}y_{j}\otimes y_{i}$.
\end{pro}

\begin{lem}\label{lem:coba}
Let $(A, \alpha)$ be an averaging algebra, $r\in A\otimes A$, and $\beta: A\rightarrow A$
be a linear map such that $(A, \fl_{A}, \fr_{A}, \beta)$ is a bimodule over $(A, \alpha)$.
If the linear map $\Delta: A\rightarrow A\otimes A$ by Eq. (\ref{cobo})
defines a coassociative coalgebra structure on $A$, then $\beta$ is an averaging
operator on $(A, \Delta)$ if and only if for any $a\in A$,
\begin{align}
(\id\otimes\,\fl_{A}(\beta(a)))(\beta\otimes\id-\id\otimes\,\alpha)(r)
-(\fr_{A}(\beta(a))\otimes\id)(\alpha\otimes\id-\id\otimes\,\beta)(r)&=0, \label{coba3}\\
(2\id\otimes\,\beta\fl_{A}(a)-\id\otimes\,\fl_{A}(\beta(a)))
(\beta\otimes\id-\id\otimes\,\alpha)(r)\qquad\qquad\qquad\qquad\;&       \label{coba4}\\[-1mm]
+(2\beta\fr_{A}(a)\otimes\id-\fr_{A}(\beta(a))\otimes\id)
(\alpha\otimes\id-\id\otimes\,\beta)(r)&=0.                           \nonumber
\end{align}
If $(A, \Delta, \beta)$ is an averaging coalgebra, $(A, \fr_{A^{\ast}}^{\ast},
\fl_{A^{\ast}}^{\ast}, \alpha)$ is a bimodule over $(A^{\ast}, \Delta^{\ast}, \beta^{\ast})$
if and only if for any $a\in A$,
\begin{align}
&\qquad\qquad\qquad\big(\id\otimes\,\fl_{A}(\alpha(a))-\fr_{A}(\alpha(a))\otimes\id\big)
(\beta\otimes\id-\id\otimes\,\alpha)(r)=0,                                    \label{coba5}\\
&\big(2\id\otimes\,\alpha\fl_{A}(a)+2\beta\fr_{A}(a)\otimes\id
-\id\otimes\,\fl_{A}(\alpha(a))-\fr_{A}(\alpha(a))\otimes\id\big)
(\beta\otimes\id-\id\otimes\,\alpha)(r)=0,                                    \label{coba6}\\
&\qquad\qquad\qquad\big(\id\otimes\,\fl_{A}(\alpha(a))-\fr_{A}(\alpha(a))\otimes\id\big)
(\alpha\otimes\id-\id\otimes\,\beta)(r)=0,                                    \label{coba7}\\
&\big(2\id\otimes\,\beta\fl_{A}(a)+2\alpha\fr_{A}(a)\otimes\id-\id\otimes\,\fl_{A}(\alpha(a))
-\fr_{A}(\alpha(a))\otimes\id\big)(\alpha\otimes\id-\id\otimes\,\beta)(r)=0.  \label{coba8}
\end{align}
\end{lem}

\begin{proof}
First, since $(A, \fl_{A}, \fr_{A}, \beta)$ is a bimodule over $(A, \alpha)$, for any $a\in A$,
we get $\fr_{A}(\beta(a))\alpha=\beta\fr_{A}(a)\alpha=\beta\fr_{A}(\beta(a))$ and
$\fl_{A}(\beta(a))\alpha=\beta\fl_{A}(a)\alpha=\beta\fl_{A}(\beta(a))$, and so that,
\begin{align*}
&\;(\beta\otimes\id)\Delta\beta(a)-(\id\otimes\beta)\Delta\beta(a)\\
=&\;(\beta\otimes\fl_{A}(\beta(a))-\beta\fr_{A}(\beta(a))\otimes\id)(r)
-(\id\otimes\beta\fl_{A}(\beta(a))-\fr_{A}(\beta(a))\otimes\beta)(r)\\
=&\;(\id\otimes\,\fl_{A}(\beta(a)))(\beta\otimes\id-\id\otimes\,\alpha)(r)
-(\fr_{A}(\beta(a))\otimes\id)(\alpha\otimes\id-\id\otimes\,\beta)(r),
\end{align*}
and
\begin{align*}
&\;2(\beta\otimes\beta)\Delta(a)-(\beta\otimes\id)\Delta\beta(a)
-(\id\otimes\beta)\Delta\beta(a)\\
=&\;2(\beta\otimes\beta\fl_{A}(a)-\beta\fr_{A}(a)\otimes\beta)(r)
-(\beta\otimes\fl_{A}(\beta(a))-\beta\fr_{A}(\beta(a))\otimes\id)(r)\\[-1mm]
&\qquad\qquad\qquad\qquad\qquad\qquad\quad
-(\id\otimes\beta\fl_{A}(\beta(a))-\fr_{A}(\beta(a))\otimes\beta)(r)\\
=&\;(2\id\otimes\,\beta\fl_{A}(a)-\id\otimes\,\fl_{A}(\beta(a)))
(\beta\otimes\id-\id\otimes\,\alpha)(r)\\[-1mm]
&\qquad\qquad\qquad+(2\beta\fr_{A}(a)\otimes\id-\fr_{A}(\beta(a))\otimes\id)
(\alpha\otimes\id-\id\otimes\,\beta)(r),
\end{align*}
Thus, we get $\beta$ is an averaging operator on $(A, \Delta)$ if and only if
Eqs. \eqref{coba3} and \eqref{coba4} hold.

Next, note that $(A, \fr_{A^{\ast}}^{\ast}, \fl_{A^{\ast}}^{\ast}, \alpha)$ is a
bimodule over $(A^{\ast}, \Delta^{\ast}, \beta^{\ast})$ if and only if
$(A^{\ast}, \fl_{A^{\ast}}, \fr_{A^{\ast}}, \alpha^{\ast})$ is a
bimodule over $(A^{\ast}, \Delta^{\ast}, \beta^{\ast})$, if and only if
$$
(\beta\otimes\alpha)\Delta=(\beta\otimes\id)\Delta\alpha
=(\id\otimes\alpha)\Delta\alpha,\qquad\quad
(\alpha\otimes\beta)\Delta=(\id\otimes\beta)\Delta\alpha
=(\alpha\otimes\id)\Delta\alpha,
$$
by Lemma \ref{lem:dual-mod}, similar to the calculation above, we can get that
$(\beta\otimes\alpha)\Delta=(\beta\otimes\id)\Delta\alpha=(\id\otimes\alpha)\Delta\alpha$
if and only if Eqs. \eqref{coba5}-\eqref{coba6} hold, and $(\alpha\otimes\beta)\Delta=
(\id\otimes\beta)\Delta\alpha=(\alpha\otimes\id)\Delta\alpha$ if and only if Eqs.
\eqref{coba7}-\eqref{coba8} hold. The proof is complete.
\end{proof}

Let $(A, \alpha)$ be an averaging algebra, $r\in A\otimes A$, and $\beta: A\rightarrow A$
be a linear map such that $(A, \fl_{A}, \fr_{A}, \beta)$ is a bimodule over $(A, \alpha)$.
Define a linear map $\Delta$ by Eq. \eqref{cobo}. Then $(A, \Delta, \alpha, \beta)$
is an averaging ASI bialgebra if and only if Eqs. \eqref{coba1}-\eqref{coba8} hold.
In particular, $(A, \Delta, \alpha, \alpha)$ is an averaging ASI bialgebra if and only if
Eqs. \eqref{coba1}-\eqref{coba4} for $\beta=\alpha$ hold.

\begin{pro}\label{pro:doublebia}
Let $(A, \Delta, \alpha, \beta)$ be an averaging ASI bialgebra,
$\tilde{\Delta}: A^{\ast}\rightarrow A^{\ast}\otimes A^{\ast}$ be the linear dual
of the multiplication of $A$ and $\Delta^{\ast}: A^{\ast}\otimes A^{\ast}\rightarrow
A^{\ast}$ be the linear dual of $\Delta$. Then $(A^{\ast}, -\tilde{\Delta},\, \beta^{\ast},
\alpha^{\ast})$ is an averaging ASI bialgebra. Further, there is an averaging
ASI bialgebra structure on the direct sum $A \oplus A^{\ast}$,
containing the two averaging ASI bialgebras as averaging ASI sub-bialgebras.
\end{pro}

\begin{proof}
By \cite[Remark 2.2.4]{Bai}, $(A^{\ast}, \Delta^{\ast}, -\tilde{\Delta})$ is an ASI bialgebra.
Note that $(A^{\ast}, \Delta^{\ast}, \beta^{\ast})$ is an averaging algebra since
$(A, \Delta, \beta)$ is an averaging coalgebra, $(A^{\ast}, -\tilde{\Delta},\,
\alpha^{\ast})$ is an averaging coalgebra since $(A, \alpha)$ is an averaging algebra,
and $(A, \fl_{A}, \fr_{A}, \beta)$ is a bimodule over $(A, \alpha)$,
$(A^{\ast}, \fl_{A^{\ast}}, \fr_{A^{\ast}}, \alpha^{\ast})$ is a bimodule over
$(A^{\ast}, \beta^{\ast})$, we get $(A^{\ast}, -\tilde{\Delta},\, \beta^{\ast},
\alpha^{\ast})$ is also an averaging ASI bialgebra.

Let $\{e_{1}, e_{2}, \cdots, e_{n}\}$ be a basis of $A$, $\{e_{1}^{\ast}, e_{2}^{\ast},
\cdots, e_{n}^{\ast}\}$ be the dual basis, and $r=\sum_{i=1}^{n}e_{i}\otimes e_{i}^{\ast}
\in A\otimes A^{\ast}\subset(A\oplus A^{\ast})\otimes(A\oplus A^{\ast})$.
Since $(A, \Delta, \alpha, \beta)$ is an averaging ASI bialgebra, there is a
corresponding matched pair $((A, \alpha), (A^{\ast}, \beta^{\ast}), \fr_{A}^{\ast},
\fl_{A}^{\ast}, \fr_{A^{\ast}}^{\ast}, \fl_{A^{\ast}}^{\ast})$.
Let $(A\bowtie A^{\ast}, \alpha\oplus\beta^{\ast})$ be the averaging algebra
structure on $A\oplus A^{\ast}$ obtained from this matched pair. By Lemma \ref{lem:douadm},
we get $(A\oplus A^{\ast}, \fl_{A\bowtie A^{\ast}}, \fr_{A\bowtie A^{\ast}},
\beta\oplus\alpha^{\ast})$ is a bimodule over $(A\bowtie A^{\ast}, \alpha\oplus\beta^{\ast})$.
Define
$$
\Delta_{A\bowtie A^{\ast}}(x)=(\id\otimes\,\fl_{A\bowtie A^{\ast}}(x)
-\fr_{A\bowtie A^{\ast}}(x)\otimes\id)(r),
$$
for any $x\in A\bowtie A^{\ast}$. Then
\begin{align*}
&\;\big((\alpha\oplus\beta^{\ast})\otimes\id-\id\otimes(\beta\oplus\alpha^{\ast})\big)(r)\\
=&\;\sum_{i=1}^{n}\big(\alpha(e_{i})\otimes e_{i}^{\ast}
-e_{i}\otimes\alpha^{\ast}(e_{i}^{\ast})\big)
=\sum_{i=1}^{n}\alpha(e_{i})\otimes e_{i}^{\ast}-\sum_{i=1}^{n}\sum_{j=1}^{n}
e_{i}\otimes\langle\alpha^{\ast}(e_{i}^{\ast}),\, e_{j}\rangle e_{j}^{\ast}\\[-2mm]
=&\;\sum_{i=1}^{n}\alpha(e_{i})\otimes e_{i}^{\ast}-\sum_{i=1}^{n}\sum_{j=1}^{n}
\langle e_{i}^{\ast},\, \alpha(e_{j})\rangle e_{i}\otimes e_{j}^{\ast}
=\sum_{i=1}^{n}\alpha(e_{i})\otimes e_{i}^{\ast}-
\sum_{i=j}^{n}\alpha(e_{j})\otimes e_{j}^{\ast}=0.
\end{align*}
Similarly $\big((\beta\oplus\alpha^{\ast})\otimes\id-\id\otimes(\alpha\oplus\beta^{\ast})
\big)(r)=0$. Hence, Eqs. \eqref{coba3}-\eqref{coba8} hold. By \cite[Theorem 2.3.6]{Bai},
Eqs. \eqref{coba1}-\eqref{coba2} hold. Therefore, $(A\bowtie A^{\ast},
\Delta_{A\bowtie A^{\ast}}, \alpha\oplus\beta^{\ast}, \beta\oplus\alpha^{\ast})$ is an
averaging ASI bialgebra. Obviously it contains $(A, \Delta, \alpha, \beta)$ and
$(A^{\ast}, -\tilde{\Delta},\, \beta^{\ast}, \alpha^{\ast})$ as averaging
ASI sub-bialgebras.
\end{proof}

Proposition \ref{pro:doublebia} provides a method for constructing averaging ASI bialgebra.
More exactly, for any averaging ASI bialgebra $(A, \Delta, \alpha, \beta)$,
we get a new averaging ASI bialgebra $(A\bowtie A^{\ast}, \Delta_{A\bowtie A^{\ast}},
\alpha\oplus\beta^{\ast}, \beta\oplus\alpha^{\ast})$, which is called the {\it
double averaging ASI bialgebra} of $(A, \Delta, \alpha, \beta)$.
Moreover, as a direct conclusion, we have

\begin{cor}\label{cor:cobia}
Let $(A, \alpha)$ be an averaging algebra, $r\in A\otimes A$, and $\beta: A\rightarrow A$
be a linear map such that $(A, \fl_{A}, \fr_{A}, \beta)$ is a bimodule over $(A, \alpha)$.
Then the linear map $\Delta: A\rightarrow A\otimes A$ by Eq. \eqref{cobo}
makes $(A, \Delta, \alpha, \beta)$ is an averaging ASI bialgebra if Eq. \eqref{coba1}
and the following equations hold:
\begin{align}
r_{12}r_{13}+r_{13}r_{23}-r_{23}r_{12}&=0,                    \label{ybe1} \\
(\alpha\otimes\id-\id\otimes\beta)(r)&=0,                     \label{ybe2} \\
(\beta\otimes\id-\id\otimes\alpha)(r)&=0.                     \label{ybe3}
\end{align}
\end{cor}

Recall that an element $r\in A\otimes A$ is called {\it antisymmetric} if $\tau(r)=-r$,
and is called {\it symmetric} if $\tau(r)=r$. If $r$ is symmetric or antisymmetric,
Eq. \eqref{ybe2} is equivalent to Eq. \eqref{ybe3}.

\begin{defi}\label{def:YBE}
Let $(A, \alpha)$ be an averaging algebra, $r\in A\otimes A$, and $\beta: A\rightarrow A$
be a linear map. Then Eqs. \eqref{ybe1}-\eqref{ybe3} is called the {\rm Yang-Baxter equation
in $(A, \alpha)$ with respect to $\beta$}, or simply {\rm $\beta$-YBE in $(A, \alpha)$}.
If $\beta=\alpha$, these equations are called the {\rm YBE in averaging algebra $(A, \alpha)$}.
\end{defi}

Let $(A, \alpha)$ be an averaging algebra, $r\in A\otimes A$, $\Delta: A\rightarrow
A\otimes A$ given by Eq. \eqref{cobo}, and $\beta: A\rightarrow A$ be a linear map
such that $(A, \fl_{A}, \fr_{A}, \beta)$ is a bimodule over $(A, \alpha)$.
By the definition above, we get the following corollary.

\begin{cor}\label{cor:ybe-bi}
Let $(A, \alpha)$ be an averaging algebra, $r\in A\otimes A$, and $\beta: A\rightarrow A$
be a linear map. If $r$ is an antisymmetric solution of the $\beta$-YBE in $(A, \alpha)$,
then $(A, \Delta, \alpha, \beta)$ is an averaging ASI bialgebra, where $\Delta$ is given by
Eq. \eqref{cobo}.
\end{cor}

\begin{ex}\label{ex:YBE-bialg}
Let $(A, \alpha)$ be the 3-dimensional averaging algebra, which is given by $A=\Bbbk\{e_{1},
e_{2}, e_{3}\}$ with non-zero product $e_{1}e_{1}=e_{1}$, $e_{1}e_{2}=e_{2}=e_{2}e_{1}$ and
$\alpha(e_{1})=\alpha(e_{2})=e_{3}$, $\alpha(e_{3})=0$. Define a linear map $\beta: A
\rightarrow A$ by $\beta(e_{1})=e_{3}$, $\beta(e_{2})=-e_{3}$ and $\beta(e_{3})=0$.
Then $(A^{\ast}, \fr_{A}^{\ast}, \fl_{A}^{\ast}, \beta^{\ast})$ is a bimodule over
$(A, \alpha)$.

$(i)$ Let $r=e_{2}\otimes e_{3}-e_{3}\otimes e_{2}$. Then one can check that
$r$ is an antisymmetric solution of the $\beta$-YBE in $(A, \alpha)$. Thus $r$ induces
a comultiplication $\Delta: A\rightarrow A\otimes A$ by Eq. \eqref{cobo}, which is given by
$\Delta(e_{1})=-e_{2}\otimes e_{3}-e_{3}\otimes e_{2}$, $\Delta(e_{2})=\Delta(e_{3})=0$,
such that $(A, \Delta, \alpha, \beta)$ is an averaging ASI bialgebra.

$(ii)$ Let $r=e_{3}\otimes e_{3}$. Then one can check that $r$ is a symmetric solution
of the $\beta$-YBE in $(A, \alpha)$ and satisfies Eq. \eqref{coba1}. Thus $r$ also induces
a trivial comultiplication $\Delta: A\rightarrow A\otimes A$ by Eq. \eqref{cobo} such that
$(A, \Delta, \alpha, \beta)$ is an averaging ASI bialgebra.
\end{ex}

\begin{ex}\label{ex:YBE-bialg1}
Let $(A, \alpha)$ be the 3-dimensional commutative averaging algebra, which is
given by $A=\Bbbk\{e_{1}, e_{2}, e_{3}\}$ with non-zero product $e_{1}e_{1}=e_{1}$,
$e_{1}e_{2}=e_{2}=e_{2}e_{1}$ and $\alpha(e_{1})=e_{3}$, $\alpha(e_{3})=\alpha(e_{2})=0$.
Then $(A^{\ast}, \fr_{A}^{\ast}, \fl_{A}^{\ast}, \beta^{\ast})$ is a bimodule over
$(A, \alpha)$, where $\beta: A\rightarrow A$ is the zero map.

$(i)$ Let $r=e_{2}\otimes e_{3}-e_{3}\otimes e_{2}$. Then one can check that
$r$ is an antisymmetric solution of the $\beta$-YBE in $(A, \alpha)$. Thus $r$ induces
a comultiplication $\Delta: A\rightarrow A\otimes A$ by Eq. \eqref{cobo}, which is given by
$\Delta(e_{1})=-e_{2}\otimes e_{3}-e_{3}\otimes e_{2}$, $\Delta(e_{2})=\Delta(e_{3})=0$,
such that $(A, \Delta, \alpha, \beta)$ is an averaging ASI bialgebra.

$(ii)$ Let $r=e_{3}\otimes e_{3}$. Then one can check that $r$ is a symmetric solution
of the $\beta$-YBE in $(A, \alpha)$ and satisfies Eq. \eqref{coba1}. Thus $r$ also induces
a trivial comultiplication $\Delta: A\rightarrow A\otimes A$ by Eq. \eqref{cobo} such that
$(A, \Delta, \alpha, \beta)$ is an averaging ASI bialgebra.
\end{ex}

Let $V$ be a vector space. Through the isomorphism $V\otimes V\cong
\Hom(V^{\ast}, V)$, any $r=\sum_{i}x_{i}\otimes y_{i}\in V\otimes V$ can be view
as a map $r^{\sharp}: V^{\ast}\rightarrow V$, explicitly, $r^{\sharp}(\xi)=\sum_{i}
\langle\xi, x_{i}\rangle y_{i}$, for any $\xi\in V^{\ast}$.

\begin{pro}\label{pro:ybeo}
Let $(A, \alpha)$ be an averaging algebra, $r\in A\otimes A$ be antisymmetric,
$\beta: A\rightarrow A$ be a linear map. Then $r$ is a solution of $\beta$-YBE in
$(A, \alpha)$ if and only if $r^{\sharp}$ satisfies the following equations:
\begin{align}
\alpha r^{\sharp}&=r^{\sharp}\beta^{\ast},                          \label{ybeo1}\\
r^{\sharp}(\xi_{1})r^{\sharp}(\xi_{2})
&=r^{\sharp}\Big(\fr^{\ast}_{A}(r^{\sharp}(\xi_{1}))(\xi_{2})
+\fl^{\ast}_{A}(r^{\sharp}(\xi_{2}))(\xi_{1})\Big),                  \label{ybeo2}
\end{align}
for any $\xi_{1}, \xi_{2}\in A^{\ast}$.
\end{pro}

\begin{proof}
First, by \cite[Proposition 2.4.7]{Bai}, Eq. \eqref{ybe1} holds if and only if Eq.
(\ref{ybeo2}) holds. Denote $r=\sum_{i}x_{i}\otimes y_{i}$.
For any $\xi\in A^{\ast}$, note that
$$
r^{\sharp}(\beta^{\ast}(\xi))=\sum_{i}\langle\xi,\; \beta(x_{i})\rangle y_{i}
\qquad\mbox{ and }\qquad
\alpha(r^{\sharp}(\xi))=\sum_{i}\langle\xi,\; x_{i}\rangle\alpha(y_{i}),
$$
we get Eq. \eqref{ybe3} holds if and only if Eq. \eqref{ybeo1} holds.
Eq. \eqref{ybe3} holds if and only if Eq. \eqref{ybe2} holds, since $r$
is antisymmetric. The proof is complete.
\end{proof}

Now, let $(A, \alpha, \mathfrak{B})$ be a symmetric averaging Frobenius algebra.
Then under the natural bijection $\Hom(A\otimes A, \Bbbk)\cong\Hom(A, A^{\ast})$, the
bilinear form $\mathfrak{B}(-,-)$ corresponds to a linear map $\varphi:A\rightarrow A^{\ast}$,
which is given by $\langle\varphi(a_{1}),\; a_{2}\rangle=\mathfrak{B}(a_{1}, a_{2})$,
for any $a_{1}, a_{2}\in A$. For any $r\in A\otimes A$, define a linear map
$R_{r}: A\rightarrow A$, $a\mapsto r^{\sharp}(\varphi(a))$,
then we have the following proposition.

\begin{pro}\label{pro:yberb}
Let $(A, \alpha, \mathfrak{B})$ be a symmetric averaging Frobenius algebra and
$r\in A\otimes A$ be antisymmetric. Suppose that $\hat{\alpha}$ is the adjoint
of $\alpha$ with respect to $\mathfrak{B}(-,-)$. Then, $r$ is a solution of
$\hat{\alpha}$-YBE in $(A, \alpha)$ if and only if $R_{r}$ satisfies the following equations:
\begin{align}
\alpha R_{r}&=R_{r}\alpha,                                            \label{yberb1}\\
R_{r}(a_{1})R_{r}(a_{2})&=R_{r}(a_{1}R_{r}(a_{2})+R_{r}(a_{1})a_{2}), \label{yberb2}
\end{align}
for any $a_{1}, a_{2}\in A$.
Moreover, in this case, $(A, \Delta, \alpha, \hat{\alpha})$ is
an averaging ASI bialgebra, where $\Delta$ is defined by Eq. \eqref{cobo}.
\end{pro}

\begin{proof}
By \cite[Corollary 3.17]{Bai1}, Eq. \eqref{ybe1} holds if and only if $R_{r}$ satisfies
Eq. \eqref{yberb2}. Set $r=\sum_{i}x_{i}\otimes y_{i}$. For any $a\in A$, we have
$\alpha R_{r}(a)=\alpha r^{\sharp}(\varphi(a))$ and
$$
R_{r}\alpha(a)=\sum_{i}\mathfrak{B}(\alpha(a),\; x_{i})y_{i}
=\sum_{i}\mathfrak{B}(a,\; \hat{\alpha}(x_{i}))y_{i}
=r^{\sharp}\hat{\alpha}^{\ast}(\varphi(a)).
$$
Since $\varphi$ is a linear isomorphism, we get $\alpha r^{\sharp}
=r^{\sharp}\hat{\alpha}^{\ast}$ if and only if $\alpha R_{r}=R_{r}\alpha$.
Thus, the conclusion follows from Proposition \ref{pro:ybeo}.
\end{proof}

From Eqs. \eqref{yberb2} and \eqref{ybeo2}, we seem to see the shadows of the
Rota-Baxter operator and $\mathcal{O}$-operators. Next, we study the
$\mathcal{O}$-operators of averaging algebras.

\subsection{$\mathcal{O}$-operators of averaging algebras}\label{subsec:o-oper}

\begin{defi}\label{def:ooper}
Let $(A, \alpha)$ be an averaging algebra and $(M, \kl, \kr, \beta)$ be a
bimodule over $(A, \alpha)$. A linear map $P: M\rightarrow A$ is called an {\rm
$\mathcal{O}$-operator of $(A, \alpha)$ associated to $(M, \kl, \kr, \beta)$} if
$P$ satisfies
\begin{align}
\alpha P&=P\beta,                                           \label{oop1}\\
P(m_{1})P(m_{2})&=P\Big(\kl(P(m_{1}))(m_{2})+\kr(P(m_{2}))m_{1}\Big), \label{oop2}
\end{align}
for any $m_{1}, m_{2}\in M$.
\end{defi}

In the definition above, Eq. \eqref{oop2} means that $P$ is an $\mathcal{O}$-operator
of associative algebra $A$ associated to $(M, \kl, \kr)$. Let $(A, \alpha)$ be
an averaging algebra. Then the identity map $\id: A\rightarrow A$ is an
$\mathcal{O}$-operator of $(A, \alpha)$ associated to $(A, \fl_{A}, 0, \alpha)$
or $(A, 0, \fr_{A}, \alpha)$. Eqs. \eqref{ybeo1} and \eqref{ybeo2} mean that
$r^{\sharp}$ is an $\mathcal{O}$-operator of $(A, \alpha)$ associated to
$(A^{\ast}, \fl_{A}^{\ast}, \fr_{A}^{\ast}, \beta^{\ast})$ if $(A^{\ast},
\fl_{A}^{\ast}, \fr_{A}^{\ast}, \beta^{\ast})$ is a bimodule over $(A, \alpha)$.
In particular, in Definition \ref{def:ooper}, if the bimodule $(M, \kl, \kr, \beta)$
is just the regular bimodule $(A, \fl_{A}, \fr_{A}, \alpha)$, the operator
$P$ is called a {\it Rota-Baxter operator of weight $0$} on averaging algebra $(A, \alpha)$.
Then by Proposition \ref{pro:yberb}, we have

\begin{cor}\label{cor:RB}
Let $(A, \alpha, \mathfrak{B})$ be a symmetric averaging Frobenius algebra and
$r\in A\otimes A$ be antisymmetric. Suppose that $\hat{\alpha}$ is the adjoint
of $\alpha$ with respect to $\mathfrak{B}(-,-)$. Then, $r$ is a solution of
$\hat{\alpha}$-YBE in $(A, \alpha)$ if and only if $R_{r}$ is a Rota-Baxter operator
(of weight $0$) on the averaging algebra $(A, \alpha)$.
\end{cor}

For the $\mathcal{O}$-operator of associative algebras, we have


\begin{pro}[{\cite[Corollary 3.10]{BGN1}}]\label{pro:O-cons}
Let $A$ be an associative algebra and $(M, \kl, \kr)$ be a bimodule over $A$.
Let $P: M\rightarrow A$ be a linear map which is identified as an element in
$(A\ltimes M^{\ast})\otimes(A \ltimes M^{\ast})$ through $\Hom(M, A)\cong A\otimes M^{\ast}
\subset(A\ltimes M^{\ast})\otimes(A \ltimes M^{\ast})$. Then $r:=P-\tau(P)$ is an
antisymmetric solution of YBE in $A \ltimes M^{\ast}$ if and only if $P$ is an
$\mathcal{O}$-operator of $A$ associated to $M$.
\end{pro}

We will generalize the above construction to the context of averaging algebras,
showing that $\mathcal{O}$-operators of averaging algebras give antisymmetric solutions
of YBE in semidirect product averaging algebras and hence give rise to
averaging ASI bialgebras.

\begin{pro}\label{pro:admsemi}
Let $(A, \alpha)$ be an averaging algebra, $(M, \kl, \kr)$ be a bimodule over $A$,
and $\beta: A\rightarrow A$, $\gamma_{1}, \gamma_{2}: M\rightarrow M$ be linear maps.
Then the following conditions are equivalent.
\begin{enumerate}\itemsep=0pt
\item[$(i)$] There is an averaging algebra $(A\ltimes M, \alpha\oplus\gamma_{1})$ such that
     $(A\ltimes M, \fl_{A\ltimes M}, \fr_{A\ltimes M},\beta\oplus\gamma_{2})$
     is a bimodule over $(A\ltimes M,\alpha\oplus\gamma_{1})$;
\item[$(ii)$] There is an averaging algebra $(A\ltimes M^{\ast}, \alpha\oplus
     \gamma_{2}^{\ast})$ such that $(A\ltimes M^{\ast}, \fl_{A\ltimes M^{\ast}},
     \fr_{A\ltimes M^{\ast}}, \beta\oplus\gamma_{1}^{\ast})$
     is a bimodule over $(A\ltimes M^{\ast}, \alpha\oplus\gamma_{2}^{\ast})$;
\item[$(iii)$] The following conditions are satisfied:
\begin{enumerate}\itemsep=0pt
\item[(a)] $(M, \kl, \kr, \gamma_{1})$ is a bimodule over $(A, \alpha)$,
\item[(b)] $(A, \fl_{A}, \fr_{A}, \beta)$ is a bimodule over $(A, \alpha)$,
\item[(c)] $(M, \kl, \kr, \gamma_{2})$ is a bimodule over $(A, \alpha)$,
\item[(d)] for any $a\in A$ and $m\in M$,
\begin{align}
\kl(\beta(a))(\gamma_{1}(m))=\gamma_{2}(\kl(a)(\gamma_{1}(m)))
=\gamma_{2}(\kl(\beta(a))(m)),                       \label{admsemi1} \\
\kr(\beta(a))(\gamma_{1}(m))=\gamma_{2}(\kr(a)(\gamma_{1}(m)))
=\gamma_{2}(\kr(\beta(a))(m)).                        \label{admsemi2}
\end{align}
\end{enumerate}
\end{enumerate}
\end{pro}

\begin{proof}
$(i)\Leftrightarrow(iii)$. By Proposition \ref{pro:mod}, $(A\ltimes M,
\alpha\oplus\gamma_{1})$ is an averaging algebra if and only if
$(M, \kl, \kr, \gamma_{1})$ is a bimodule over $(A, \alpha)$. Moreover,
$(A\ltimes M, \fl_{A\ltimes M}, \fr_{A\ltimes M}, \beta\oplus\gamma_{2})$
is a bimodule over $(A\ltimes M, \alpha\oplus\gamma_{1})$ if and only if, for any
$a_{1}, a_{2}\in A$ and $m_{1}, m_{2}\in M$,
\begin{align}
\fl_{A\ltimes M}(\alpha(a_{1}), \gamma_{1}(m_{1}))(\beta(a_{2}), \gamma_{2}(m_{2}))
&=(\beta\oplus\gamma_{2})\big(\fl_{A\ltimes M}(\alpha(a_{1}),
\gamma_{1}(m_{1}))(a_{2}, m_{2})\big)\label{equm1}\\[-1mm]
&=(\beta\oplus\gamma_{2})\big(\fl_{A\ltimes M}(a_{1}, m_{1})(\beta(a_{2}),
\gamma_{2}(m_{2}))\big), \nonumber\\
\fr_{A\ltimes M}(\alpha(a_{1}), \gamma_{1}(m_{1}))(\beta(a_{2}), \gamma_{2}(m_{2}))
&=(\beta\oplus\gamma_{2})\big(\fr_{A\ltimes M}(\alpha(a_{1}),
\gamma_{1}(m_{1}))(a_{2}, m_{2})\big)\label{equm2}\\[-1mm]
&=(\beta\oplus\gamma_{2})\big(\fr_{A\ltimes M}(a_{1}, m_{1})(\beta(a_{2}),
\gamma_{2}(m_{2}))\big). \nonumber
\end{align}
Note that
\begin{align*}
\fl_{A\ltimes M}(\alpha(a_{1}), \gamma_{1}(m_{1}))(\beta(a_{2}), \gamma_{2}(m_{2}))
&=\big(\alpha(a_{1})\beta(a_{2}),\ \ \kl(\alpha(a_{1}))(\gamma_{2}(m_{2}))
+\kr(\beta(a_{2}))(\gamma_{1}(m_{1}))\big),\\
(\beta\oplus\gamma_{2})\big(\fl_{A\ltimes M}(\alpha(a_{1}),
\gamma_{1}(m_{1}))(a_{2}, m_{2})\big)&=\big(\beta(\alpha(a_{1})a_{2}),\ \
\gamma_{2}(\kl(\alpha(a_{1}))(m_{2}))+\gamma_{2}(\kr(a_{2})(\gamma_{1}(m_{1})))\big),\\
(\beta\oplus\gamma_{2})\big(\fl_{A\ltimes M}(a_{1}, m_{1})(\beta(a_{2}),
\gamma_{2}(m_{2}))\big)&=\big(\beta(a_{1}\beta(a_{2})),\ \ \gamma_{2}(\kl(a_{1})
(\gamma_{2}(m_{2})))+\gamma_{2}(\kr(\beta(a_{2}))(m_{1}))\big),
\end{align*}
we get Eq. \eqref{equm1} holds if and only if Eq. \eqref{adm1} hold for $(b)$ and $(c)$,
and \eqref{admsemi2} hold.
Similarly, we get Eq. \eqref{equm2} holds if and only if Eq. \eqref{adm2} holds for $(b)$
and $(c)$, and \eqref{admsemi1} holds.

$(ii)\Leftrightarrow(iii)$. By the proof of $(i)\Leftrightarrow(iii)$, we get that
$(A\ltimes M^{\ast}, \alpha\oplus\gamma_{2}^{\ast})$ is an averaging algebra and
$(A\ltimes M^{\ast}, \fl_{A\ltimes M^{\ast}}, \fr_{A\ltimes M^{\ast}},
\beta\oplus\gamma_{1}^{\ast})$ is a bimodule over $(A\ltimes M^{\ast},
\alpha\oplus\gamma_{2}^{\ast})$, if and only if
\begin{enumerate}\itemsep=0pt
\item[($a'$)] $(M^{\ast}, \kr^{\ast}, \kl^{\ast}, \gamma_{2}^{\ast})$ is a bimodule over
$(A, \alpha)$,
\item[($b'$)] $(A, \fl_{A}, \fr_{A}, \beta)$ is a bimodule over $(A, \alpha)$,
\item[($c'$)] $(M^{\ast}, \kr^{\ast}, \kl^{\ast}, \gamma_{1}^{\ast})$ is a bimodule over
$(A, \alpha)$,
\item[($d'$)] for any $a\in A$ and $\xi\in M^{\ast}$, we have
$\kl^{\ast}(\beta(a))(\gamma_{2}^{\ast}(\xi))
=\gamma_{1}^{\ast}(\kl^{\ast}(a)(\gamma_{2}^{\ast}(\xi)))
=\gamma_{1}^{\ast}(\kl^{\ast}(\beta(a))(\xi))$ and
$\kr^{\ast}(\beta(a))(\gamma_{2}^{\ast}(\xi))
=\gamma_{1}^{\ast}(\kr^{\ast}(a)(\gamma_{2}^{\ast}(\xi)))
=\gamma_{1}^{\ast}(\kr^{\ast}(\beta(a))(\xi))$.
\end{enumerate}
Note that $(a)\Leftrightarrow(c')$, $(b)=(b')$, $(c)\Leftrightarrow(a')$ and
$(d)\Leftrightarrow(d')$ by dual, we get the proof.
\end{proof}

\begin{thm}\label{thm:aybesemi}
Let $(A, \alpha)$ be an averaging algebra, $(M, \kl, \kr, \gamma_{1})$ be a bimodule
over $(A, \alpha)$, $\beta: A\rightarrow A$, $\gamma_{2}: M\rightarrow M$ and
$P: M\rightarrow A$ be linear maps. Then we have
\begin{enumerate}\itemsep=0pt
\item[$(i)$] The element $r:=P-\tau(P)$ is an antisymmetric solution of $(\beta\oplus
    \gamma_{2}^{\ast})$-YBE in the averaging algebra $(A\ltimes M^{\ast}, \alpha\oplus
    \gamma_{1}^{\ast})$ if and only if $P$ is an $\mathcal{O}$-operator of associative
    algebra $A$ associated to $(M, \kl, \kr)$ and $\alpha P=P\gamma_{2}$,
    $\beta P=P\gamma_{1}$.

\item[$(ii)$] Assume that $(M, \kl, \kr, \gamma_{2})$ is also a bimodule over $(A, \alpha)$.
    If $P$ is an $\mathcal{O}$-operator of $(A, \alpha)$ associated to $(M, \kl, \kr,
    \gamma_{2})$ and $P\gamma_{1}=\beta P$, then $r:=P-\tau(P)$ is an antisymmetric
    solution of $(\beta\oplus\gamma_{2}^{\ast})$-YBE in the averaging algebra $(A\ltimes
    M^{\ast}, \alpha\oplus\gamma_{1}^{\ast})$. If in addition, $(A, \fl, \fr, \beta)$ is a
    bimodule over $(A, \alpha)$ and Eqs. \eqref{admsemi1} and \eqref{admsemi2} are satisfied,
    then $(A\ltimes M^{\ast}, \fl_{A\ltimes M^{\ast}}, \fr_{A\ltimes M^{\ast}},
    \beta\oplus\gamma_{2}^{\ast})$ is a bimodule over the averaging algebra $(A\ltimes
    M^{\ast}, \alpha\oplus\gamma_{1}^{\ast})$. Therefore in this case, there is an averaging
    ASI bialgebra $(A\ltimes M^{\ast}, \Delta, \alpha\oplus\gamma_{1}^{\ast},
    \beta\oplus\gamma_{2}^{\ast})$, where the linear map $\Delta$ is defined by
    Eq. \eqref{cobo} for $r=P-\tau(P)$.
\end{enumerate}
\end{thm}

\begin{proof}
$(i)$ First, following from Proposition \ref{pro:O-cons}, we get $r:=P-\tau(P)$ is an
antisymmetric solution of YBE in associative algebra $A \ltimes M^{\ast}$ if and
only if $P$ is an $\mathcal{O}$-operator of associative algebra $A$ associated to
$(M, \kl, \kr)$. We need to show that $\big((\alpha\oplus\gamma_{1}^{\ast})\otimes\id-
\id\otimes(\beta\oplus\gamma_{2}^{\ast})\big)(r)=0$ if and only if $\alpha P=P\gamma_{2}$
and $\beta P=P\gamma_{1}$. Let $\{e_{1}, e_{2}, \cdots, e_{n}\}$ be a basis of $M$,
$\{e_{1}^{\ast}, e_{2}^{\ast}, \cdots, e_{n}^{\ast}\}$ be the dual basis.
Then $P=\sum_{i=1}^{n}P(e_{i})\otimes e_{i}^{\ast}\in(A\ltimes M^{\ast})\otimes(A\ltimes
M^{\ast})$, $r=P-\tau(P)=\sum_{i=1}^{n}(P(e_{i})\otimes e_{i}^{\ast}-e_{i}^{\ast}
\otimes P(e_{i}))$, and $\sum_{i=1}^{n}\gamma_{1}^{\ast}(e_{i}^{\ast})\otimes P(e_{i})
=\sum_{i=1}^{n}\sum_{j=1}^{n}\langle\gamma_{1}^{\ast}(e_{i}^{\ast}),\;
e_{j}\rangle e_{j}^{\ast}\otimes P(e_{i})=\sum_{j=1}^{n}e_{j}^{\ast}\otimes\sum_{i=1}^{n}
\langle e_{i}^{\ast}, \gamma_{1}(e_{j})\rangle P(e_{i})=\sum_{i=1}^{n}e_{i}^{\ast}\otimes
P\big(\sum_{j=1}^{n}\langle\gamma_{1}(e_{i}),\; e_{j}^{\ast}\rangle e_{j}\big)
=\sum_{i=1}^{n} e_{i}^{\ast}\otimes P(\gamma_{1}(e_{i}))$.
Similarly, we also have $\sum_{i=1}^{n}P(e_{i})\otimes\gamma_{2}^{\ast}(e_{i}^{\ast})
=\sum_{i=1}^{n}P(\gamma_{2}(e_{i}))\otimes e_{i}^{\ast}$. Then, we have
\begin{align*}
\big((\alpha\oplus\gamma_{1}^{\ast})\otimes\id\big)(r)
&=\sum_{i=1}^{n}\big(\alpha(P(e_{i}))\otimes e_{i}^{\ast}-e_{i}^{\ast}
\otimes P(\gamma_{1}(e_{i}))\big),\\[-1mm]
\big(\id\otimes(\beta\oplus\gamma_{2}^{\ast})\big)(r)
&=\sum_{i=1}^{n}\big((P(\gamma_{2}(e_{i}))\otimes e_{i}^{\ast}-e_{i}^{\ast}
\otimes\beta(P(e_{i})))\big).
\end{align*}
Thus, $\big((\alpha\oplus\gamma_{1}^{\ast})\otimes\id-\id\otimes(\beta\oplus\gamma_{2}^{\ast})
\big)(r)=0$ if and only if $\alpha P=P\gamma_{2}$ and $\beta P=P\gamma_{1}$.

$(ii)$ It follows from item $(i)$, Proposition \ref{pro:admsemi} and
Corollary \ref{cor:ybe-bi}.
\end{proof}

In particular, in the theorem above, if $\beta=\alpha$ and $\gamma_{1}=\gamma_{2}$,
we have the following corollary.

\begin{cor}\label{cor:qminp}
Let $(A, \alpha)$ be an averaging algebra, $(M, \kl, \kr, \gamma_{1})$ be a bimodule
over $(A, \alpha)$, $P: M\rightarrow A$ be an $\mathcal{O}$-operator of $(A, \alpha)$
associated to $(M, \kl, \kr, \gamma_{1})$. Then $r:=P-\tau(P)$ is an antisymmetric
solution of YBE in the averaging algebra $(A\ltimes M^{\ast}, \alpha\oplus
\gamma_{1}^{\ast})$, and so that, $(A\ltimes M^{\ast}, \Delta, \alpha\oplus\gamma_{1}^{\ast},
\alpha\oplus\gamma_{1}^{\ast})$ is an averaging ASI bialgebra, where the linear map
$\Delta$ is defined by Eq. \eqref{cobo} for $r=P-\tau(P)$.
\end{cor}

\subsection{Averaging dendriform algebras}\label{subsec:dend}

First, we recall the notion of dendriform algebras.

\begin{defi}\label{def:dend}
Let $A$ be a vector space and $\succ, \prec: A\otimes A\rightarrow A$ be two bilinear
operations. The triple $(A, \succ, \prec)$ is called a {\rm dendriform algebra} if
\begin{align*}
(a_{1}\prec a_{2})\prec a_{3}&=a_{1}\prec(a_{2}\prec a_{3}+a_{2}\succ a_{3}), \\
(a_{1}\succ a_{2}) \prec a_{3}&=a_{1}\succ(a_{2}\prec a_{3}), \\
(a_{1}\prec a_{2}+a_{1}\succ a_{2})\succ a_{3}&=a_{1}\succ(a_{2}\succ a_{3}),
\end{align*}
for any $a_{1}, a_{2}, a_{3}\in A$.
\end{defi}

For a dendriform algebra $(A, \succ, \prec)$, define two linear maps $\fl_{\succ},
\fr_{\prec}: A\rightarrow\End_{\Bbbk}(A)$ by
$$
\fl_{\succ}(a_{1})(a_{2})=a_{1}\succ a_{2}, \qquad\qquad
\fr_{\prec}(a_{1})(a_{2})=a_{2}\prec a_{1},
$$
for any $a_{1}, a_{2}\in A$. Then we have the following proposition.

\begin{pro}[\cite{Bai}]\label{pro:dend}
Let $(A, \succ, \prec)$ be a dendriform algebra. Then the bilinear operation
\begin{align}
a_{1}\cdot a_{2}:=a_{1}\succ a_{2}+a_{1}\prec a_{2},           \label{ass}
\end{align}
for any $a_{1}, a_{2}\in A$, defines an associative algebra $(A, \cdot)$, called the
{\rm associated associative algebra} of $(A, \succ, \prec)$. Moreover,
$(A, \fl_{\succ}, \fr_{\prec})$ is a bimodule over $(A, \cdot)$, and the identity
map $\id_{A}: A\rightarrow A$ is an $\mathcal{O}$-operator of $A$ associated
to bimodule $(A, \fl_{\succ}, \fr_{\prec})$.
\end{pro}

Now, we consider the notion of averaging dendriform algebras.

\begin{defi}\label{def:avdend}
An {\rm averaging operator} on a dendriform algebra $(A, \succ, \prec)$ is a
linear map $\alpha: A\rightarrow A$ satisfying
\begin{align*}
\alpha(a_{1})\succ\alpha(a_{2})&=\alpha(\alpha(a_{1})\succ a_{2})
=\alpha(a_{1}\succ\alpha(a_{2})),\\
\alpha(a_{1})\prec\alpha(a_{2})&=\alpha(\alpha(a_{1})\prec a_{2})
=\alpha(a_{1}\prec\alpha(a_{2})),
\end{align*}
for any $a_{1}, a_{2}\in A$.
A quadruple $(A, \succ, \prec, \alpha)$ is called an {\rm averaging dendriform algebra}
if $(A, \succ, \prec)$ is a dendriform algebra and $\alpha$ is an averaging operator on
$(A, \succ, \prec)$.
\end{defi}

We will generalize some results of dendriform algebras to the context of
averaging dendriform algebras.

\begin{pro}\label{pro:assda}
Let $(A, \succ, \prec, \alpha)$ be an averaging dendriform algebra.
Then $(A, \cdot, \alpha)$ is an averaging associative algebra, where the multiplication
is defined by Eq. \eqref{ass}, which is called the {\rm associated averaging algebra
of $(A, \succ, \prec, \alpha)$}. Moreover, $(A, \fl_{\succ}, \fr_{\prec}, \alpha)$
is a bimodule over $(A, \alpha)$, and the identity map $\id_{A}: A\rightarrow A$ is an
$\mathcal{O}$-operator of $(A, \alpha)$ associated to $(A, \fl_{\succ}, \fr_{\prec}, \alpha)$.
\end{pro}

\begin{proof}
Let $(A, \succ, \prec, \alpha)$ be an averaging dendriform algebra.
It is easy to see that $\alpha$ is also an averaging operator for the
multiplication is defined by Eq. \eqref{ass}. Moreover, it is straightforward
to show that Eqs. \eqref{adm1} and \eqref{adm2} hold for $\kl=\fl_{\succ}$
and $\kr=\fr_{\prec}$ if and only if $\alpha$ is an averaging operator on
$(A, \succ, \prec)$.
Thus, $(A, \fl_{\succ}, \fr_{\prec}, \alpha)$ is a bimodule over $(A, \alpha)$.
The last conclusion follows immediately.
\end{proof}

Recall that a Rota-Baxter operator $R$ on an associative algebra $(A, \cdot)$ gives a
dendriform algebra $(A, \succ, \prec)$, where
\begin{align}
a_{1}\succ a_{2}=R(a_{1})\cdot a_{2}, \qquad\qquad
a_{1}\prec a_{2}=a_{1}\cdot R(a_{2}),                   \label{rbd}
\end{align}
for any $a_{1}, a_{2}\in A$ \cite{Agu}. Let $(A, \alpha)$ be an averaging
algebra and $R$ be a Rota-Baxter operator on the associative algebra $A$.
If $R\alpha=\alpha R$, then one can check that $(A, \succ, \prec, \alpha)$
is an averaging dendriform algebra. More generally, for $\mathcal{O}$-operators of
averaging algebras, we have

\begin{pro}\label{pro:opd}
Let $P: M\rightarrow A$ be an $\mathcal{O}$-operator of an averaging algebra $(A, \alpha)$
associated to a bimodule $(M, \kl, \kr, \beta)$. Then there exists an averaging
dendriform algebra structure $(M, \succ, \prec, \beta)$ on $M$,
where $\succ$ and $\prec$ are defined by
$$
m_{1}\succ m_{2}:=\kl(P(m_{1}))(m_{2}) \qquad \text{and} \qquad
m_{1}\prec m_{2}:=\kr(p(m_{2}))(m_{1}),
$$
for any $m_{1}, m_{2}\in M$.
\end{pro}

\begin{proof}
First, since $P: M\rightarrow A$ is an $\mathcal{O}$-operator,
$\hat{P}:={\footnotesize\begin{pmatrix}0 & P \\ 0 & 0\end{pmatrix}}: A\oplus M
\rightarrow A\oplus M$ is a Rota-Baxter operator on $A\ltimes M$.
Thus, there is a dendriform algebra structure on the vector space $A\oplus V$,
which is defined by
\begin{align*}
(a_{1}, m_{1})\succ(a_{2}, m_{2}):=&(P(m_{1}),\; 0)(a_{2}, m_{2})
=(P(m_{1})a_{2},\; \kl(P(m_{1}))(m_{2})), \\
(a_{1}, m_{1})\prec(a_{2}, m_{2}):=&(a_{1}, m_{1})(P(m_{2}),\; 0)
=(a_{1}P(m_{2}),\; \kr(P(m_{2}))(m_{1})),
\end{align*}
for all $a_{1}, a_{2}\in A$ and $m_{1}, m_{2}\in M$. By the definition of
$\mathcal{O}$-operator, i.e., $\alpha P=P\beta$, we get $\hat{P}(\alpha\oplus\beta)
=(\alpha\oplus\beta)\hat{P}$. Thus, $(A\oplus V, \succ, \prec, \alpha\oplus\beta)$
is an averaging dendriform algebra. In particular, on the vector space $M$,
there is an averaging dendriform subalgebra $(V, \succ, \prec, \beta)$,
in which $\succ$ and $\prec$ are exactly defined in this proposition.
\end{proof}

At the end of this section, we illustrate a construction of
averaging ASI bialgebras from averaging dendriform algebras.

\begin{pro}\label{pro:avb-den}
Let $(A, \succ, \prec, \alpha)$ be an averaging dendriform algebra and
$(A, \alpha)$ be the associated averaging algebra. Let $\{e_{1}, e_{2}, \cdots, e_{n}\}$
be a basis of $A$ and $\{e_{1}^{\ast}, e_{2}^{\ast}, \cdots, e_{n}^{\ast}\}$ be
the dual basis. Then $r=\sum_{i=1}^{n}(e_{i}\otimes e_{i}^{\ast}-e_{i}^{\ast}
\otimes e_{i})$ is an antisymmetric solution of YBE in averaging algebra
$(A\ltimes A^{\ast},\; \alpha\oplus\alpha^{\ast})$. Therefore, there is an averaging
ASI bialgebra $(A\ltimes A^{\ast}, \Delta, \alpha\oplus\alpha^{\ast},
\alpha\oplus\alpha^{\ast})$, where the linear map $\Delta$ is defined by Eq. \eqref{cobo}
for $r=\sum_{i=1}^{n}(e_{i}\otimes e_{i}^{\ast}-e_{i}^{\ast}\otimes e_{i})$.
\end{pro}

\begin{proof}
By Proposition \ref{pro:assda}, we get that the identity map $\id_{A}$ is an
$\mathcal{O}$-operator of averaging algebra $(A, \alpha)$ associated to
$(A, \fl_{\succ}, \fr_{\prec}, \alpha)$. Note that $\id_{A}: A\rightarrow A$
is just $\sum^{n}_{i=1}e_{i}\otimes e_{i}^{\ast}$ by the isomorphism
$\Hom(A, A)\cong A\otimes A^{\ast}\subset(A\ltimes A^{\ast})\otimes(A \ltimes A^{\ast})$,
we get $r=\id_{A}-\tau(\id_{A})$. Thus, we obtain the conclusion by Corollary
\ref{cor:qminp}.
\end{proof}

\section{Factorizable averaging antisymmetric infinitesimal bialgebras}\label{sec:fact}
In this section, we establish the factorizable theories for averaging antisymmetric
infinitesimal bialgebras. First, we introduce some notations.
Let $A$ be vector space. For any $r\in A\otimes A$,
it can be written as the sum of symmetric $\mathfrak{s}(r)$ and skew-symmetric parts
$\mathfrak{a}(r)$, i.e., $\mathfrak{s}(r), \mathfrak{a}(r)\in A\otimes A$ satisfying
$\tau(\mathfrak{s}(r))=\mathfrak{s}(r)$, $\tau(\mathfrak{a}(r))=-\mathfrak{a}(r)$ and
$r=\mathfrak{s}(r)+\mathfrak{a}(r)$. For any $r\in A\otimes A$, we have defined a linear
map $r^{\sharp}: A^{\ast}\rightarrow A$ by
$$
\langle r^{\sharp}(\xi),\; \eta\rangle=\langle\xi\otimes\eta,\; r\rangle.
$$
Now we define another linear map $r^{\natural}: A^{\ast}\rightarrow A$ by
$$
\langle\xi,\; r^{\natural}(\eta)\rangle=-\langle \xi\otimes\eta,\; r\rangle,
$$
for any $\xi, \eta\in A^{\ast}$. If $A$ is an associative algebra,
then the associative algebra structure $\cdot_{r}$ on $A^{\ast}$ dual to the
comultiplication $\Delta$ defined by Eq. \eqref{cobo} is given by
$\xi\cdot_{r}\eta=\fr_{A}^{\ast}(r^{\sharp}(\xi))(\eta)
+\fl_{A}^{\ast}(r^{\natural}(\eta))(\xi)$, for any $\xi, \eta\in A^{\ast}$.

\begin{defi}\label{def:lr-inv}
Let $(A, \alpha)$ be an averaging algebra and $r\in A\otimes A$. Then $r$ is called
{\rm $(\fl, \fr)$-invariant} if
$$
(\id\otimes\fl_{A}(a)-\fr_{A}(a)\otimes\id)(r)=0,
$$
for any $a\in A$.
\end{defi}

That is, if $r\in A\otimes A$ $(\fl, \fr)$-invariant, then the comultiplication
$\Delta$ defined by Eq. (\ref{cobo}) is zero. The $(\fl, \fr)$-invariant condition
in an averaging algebra is the same as the $(\fl, \fr)$-invariant condition in
an associative algebra. We review some conclusions about $(\fl, \fr)$-invariant
condition in associative algebra.

\begin{pro}[\cite{SW}]\label{pro:lr-inv}
Let $A$ be an associative algebra and $r\in A\otimes A$.
\begin{enumerate}\itemsep=0pt
\item[$(i)$] $r$ is $(\fl, \fr)$-invariant if and only if $r^{\sharp}(\fr_{A}^{\ast}(a)(\xi))
     =ar^{\sharp}(\xi)$, for any $a\in A$ and $\xi\in A^{\ast}$.
\item[$(ii)$] Denote by $\mathcal{I}=r^{\sharp}-r^{\natural}:\; A^{\ast}\rightarrow A$. Then,
     $\mathfrak{s}(r)$ is $(\fl, \fr)$-invariant if and only if
     $\mathcal{I}\fr_{A}^{\ast}(a)=\fl_{A}(a)\mathcal{I}$, or
     $\mathcal{I}\fl_{A}^{\ast}(a)=\fr_{A}(a)\mathcal{I}$, for any $a\in A$.
\item[$(iii)$] If $\mathfrak{s}(r)$ is $(\fl, \fr)$-invariant, then, for any $a\in A$
     and $\xi, \eta\in A^{\ast}$, $\langle\xi,\; \mathfrak{s}(r)^{\sharp}(\eta)a\rangle
     =\langle\eta,\; a\mathfrak{s}(r)^{\sharp}(\xi)\rangle$. Therefore, the associative
     algebra multiplication $\cdot_{r}$ on $A^\ast$ reduces to
     $$
     \xi\cdot_{r}\eta=\fr_{A}^{\ast}(\mathfrak{a}(r)^{\sharp}(\xi))(\eta)
     +\fl_{A}^{\ast}(\mathfrak{a}(r)^{\sharp}(\eta))(\xi),
     $$
     for any $\xi, \eta\in A^{\ast}$.
\end{enumerate}
\end{pro}

\begin{pro}\label{pro:lr-inv-equ}
Let $(A, \alpha)$ be an averaging algebra and $r\in A\otimes A$.
If $\mathfrak{s}(r)$ is $(\fl, \fr)$-invariant, then $r$ is a solution of the
$\beta$-YBE in $(A, \alpha)$ if and only if $(A^{\ast}, \cdot_{r}, \beta^{\ast})$
is an averaging algebra and the linear maps $r^{\sharp}, r^{\natural}:\, (A^{\ast},
\cdot_{r}, \beta^{\ast})\rightarrow (A, \alpha)$ are averaging algebra homomorphisms.
\end{pro}

\begin{proof}
Assume $r$ is a solution of the $\beta$-YBE in $(A, \alpha)$.
First, by a direct calculation, one can show that $\beta^{\ast}$ is an averaging
operator on $(A^{\ast}, \cdot_{r})$, and so that $(A^{\ast},
\cdot_{r}, \beta^{\ast})$ is an averaging algebra.
Second, by \cite{SW}, we get that if $\mathfrak{s}(r)$ is $(\fl, \fr)$-invariant, then
$r$ satisfies the Yang-Baxter equation in associative algebra i.e., Eq. \eqref{ybe1},
if and only if $(A^{\ast}, \cdot_{r})$ is an associative algebra and the linear maps
$r^{\sharp}, r^{\natural}:\, (A^{\ast}, \cdot_{r})\rightarrow(A, \cdot)$ are
associative algebra homomorphisms. Finally, by Eq. \eqref{ybe2}, we get
\begin{align*}
\langle(\alpha\otimes\id-\id\otimes\beta)(r),\; \xi\otimes\eta\rangle
&=\langle r^{\sharp}(\alpha^{\ast}(\xi)),\; \eta\rangle
-\langle r^{\sharp}(\xi),\; \beta^\ast(\eta)\rangle\\
&=\langle r^{\sharp}(\alpha^{\ast}(\xi))-\beta(r^{\sharp}(\xi)),\ \ \eta\rangle\\
&=0,
\end{align*}
for any $\xi, \eta\in A^{\ast}$. That is, $r^{\sharp}\alpha^{\ast}=\beta r^{\sharp}$.
By the duality between $r^{\sharp}$ and $r^{\natural}$, we get $r^{\natural}\beta^{\ast}
=\alpha r^{\natural}$. Similarly, by Eq. \eqref{ybe3}, we get $r^{\sharp}\beta^{\ast}
=\alpha r^{\sharp}$. Thus, $r^{\sharp}$ and $r^{\natural}$ are averaging algebra
homomorphisms. Conversely, it is directly available from the above calculation.
\end{proof}

Now, we give the definition of factorizable averaging ASI bialgebra.

\begin{defi}\label{def:q-tri}
Let $(A, \alpha)$ be an averaging algebra. If $r\in A\otimes A$ is a solution of
the $\beta$-YBE in $(A, \alpha)$ and $\mathfrak{s}(r)$ is $(\fl, \fr)$-invariant,
then the averaging ASI bialgebra $(A, \Delta, \alpha, \beta)$
induced by $r$ is called a {\rm quasi-triangular averaging ASI bialgebra}.

The averaging ASI bialgebra $(A, \Delta, \alpha, \beta)$ induced by
$r$ is called a {\rm factorizable} if it is quasi-triangular and the linear map
$\mathcal{I}=r^{\sharp}-r^{\natural}:\; A^{\ast}\rightarrow A$ is a linear isomorphism
and $\mathcal{I}\beta^{\ast}=\alpha\mathcal{I}$.
\end{defi}

For convenience, we can consider the linear map $\mathcal{I}=r^{\sharp}-r^{\natural}:\;
A^{\ast}\rightarrow A$ as a composition of maps as follows:
$$
A^{\ast}\xrightarrow{\quad r^{\sharp}\oplus r^{\natural}\quad} A\oplus A
\xrightarrow{\quad(a_{1}, a_{2})\mapsto a_{1}-a_{2}\quad} A.
$$
The following result justifies the terminology of a factorizable averaging ASI bialgebra.

\begin{pro}\label{pro:fact}
Let $(A, \alpha)$ be an averaging algebra and $r\in A\otimes A$. Assume the averaging
ASI bialgebra $(A, \Delta, \alpha, \beta)$ induced by $r$ is factorizable. Then
$\Img(r^{\sharp}\oplus r^{\natural})$ is an averaging subalgebra of the
direct sum averaging algebra $A\oplus A$, which is isomorphic to the averaging
algebra $(A^{\ast}, \cdot_{r}, \beta^{\ast})$. Moreover, any $a\in A$ has an unique
decomposition $a=a_{+}+a_{-}$, where $a_{+}\in\Img(r^{\sharp})$ and $a_{-}\in\Img(r^{\sharp})$.
\end{pro}

\begin{proof}
Since $(A, \Delta, \alpha, \beta)$ is quasi-triangular, both $r^{\sharp}$ and $r^{\natural}$
are averaging algebra homomorphisms. Therefore, $\Img(r^{\sharp}\oplus r^{\natural})$
is an averaging subalgebra of the direct sum averaging algebra $(A\oplus A,
\alpha\oplus\alpha)$. Since $\mathcal{I}: A^\ast\rightarrow A$ is
a linear isomorphism, it follows that $r^{\sharp}\oplus r^{\natural}$ is injective,
and so that the averaging algebra $\Img(r^{\sharp}\oplus r^{\natural})$
is isomorphic to the averaging algebra $(A^{\ast}, \cdot_{r}, \beta^{\ast})$.
Moreover, since $\mathcal{I}$ is an isomorphism again, for any $a\in A$, we have
$$
a=(r^{\sharp}-r^{\natural})(\mathcal{I}^{-1}(a))=r^{\sharp}(\mathcal{I}^{-1}(a))
-r^{\natural}(\mathcal{I}^{-1}(a)),
$$
which implies that $a=a_{+}+a_{-}$, where $a_{+}=r^{\sharp}(\mathcal{I}^{-1}(a))$ and
$a_{-}=-r^{\natural}(\mathcal{I}^{-1}(a))$. The uniqueness also follows from the fact
that $\mathcal{I}$ is an isomorphism.
\end{proof}

Let $(A, \Delta, \alpha, \beta)$ be a factorizable averaging ASI bialgebra.
By Proposition \ref{pro:fact}, we get $\alpha(a)=\alpha(a_{+})+
\alpha(a_{-})$ for any $a\in A$. That is, each element in $\alpha(A)$
is factorizable in the set $\alpha(A)$.
Following, we will give a class of factorizable averaging ASI bialgebras by
the double of an averaging ASI bialgebra.
Let $(A, \Delta, \alpha, \beta)$ be an arbitrary averaging ASI bialgebra.
By Theorem \ref{thm:equ}, there exists a matched pair of averaging algebras
$((A, \alpha), (A^{\ast}, \beta^{\ast}), \fr_{A}^{\ast}, \fl_{A}^{\ast},
\fr_{A^{\ast}}^{\ast}, \fl_{A^{\ast}}^{\ast})$ and a double construction
of averaging Frobenius algebra $(A\bowtie A^{\ast}, \alpha\oplus\beta^{\ast},
\mathfrak{B}_{d})$ corresponding to it. Let $\{e_{1}, e_{2},\cdots, e_{n}\}$ be
a basis of $A$, $\{e_{1}^{\ast}, e_{2}^{\ast},\cdots, e_{n}^{\ast}\}$ be the
dual basis, and $r=\sum_{i=1}^{n}e_{i}\otimes e_{i}^{\ast}\in A\otimes A^{\ast}
\subset(A\oplus A^{\ast})\otimes(A\oplus A^{\ast})$. By Proposition \ref{pro:doublebia},
$(A\bowtie A^{\ast}, \Delta_{A\bowtie A^{\ast}}, \alpha\oplus\beta^{\ast},
\beta\oplus\alpha^{\ast})$ is an averaging ASI bialgebra induced by $r$.

\begin{pro}\label{pro:doub-fact}
With the above notations, the averaging ASI bialgebra $(A\bowtie A^{\ast},
\Delta_{A\bowtie A^{\ast}}, \alpha\oplus\beta^{\ast}, \beta\oplus\alpha^{\ast})$
induced by $r$ is factorizable.
\end{pro}

\begin{proof}
First, by the proof of Proposition \ref{pro:doublebia}, we get that $r:=\sum_{i=1}^{n}e_{i}
\otimes e_{i}^{\ast}$ is a solution of the $\beta$-YBE in $(A, \alpha)$. Second, since
$\mathfrak{s}(r)=\frac{1}{2}\sum_{i=1}^{n}(e_{i}\otimes e_{i}^{\ast}
+e_{i}^{\ast}\otimes e_{i})$, for any $(\xi, a)\in A^{\ast}\oplus A$
we have $\mathfrak{s}(r)^{\sharp}(\xi, a)=\frac{1}{2}(a, \xi)\in A\oplus A^{\ast}$.
Thus, for any $(a_{1}, \xi_{1}), (a_{2}, \xi_{2})\in A\oplus A^{\ast}$,
\begin{align*}
&\; (a_{1}, \xi_{1})\ast\mathfrak{s}(r)_{+}(\xi_{2}, a_{2})\\
=&\; \mbox{$\frac{1}{2}$}\Big(a_{1}a_{2}+\fr_{A^{\ast}}^{\ast}(\xi_{1})(a_{2})
+\fl_{A^{\ast}}^{\ast}(\xi_{2})(a_{1}),\ \ \xi_{1}\cdot_{A^{\ast}}\xi_{2}
+\fr_{A}^{\ast}(a_{1})(\xi_{2})+\fl_{A}^{\ast}(a_{2})(\xi_{1})\Big)\\
=&\; \mathfrak{s}(r)^{\sharp}\Big(\xi_{1}\cdot_{A^{\ast}}\xi_{2}
+\fr_{A}^{\ast}(a_{1})(\xi_{2})+\fl_{A}^{\ast}(a_{2})(\xi_{1}),\ \
a_{1}a_{2}+\fr_{A^{\ast}}^{\ast}(\xi_{1})(a_{2})+\fl_{A^{\ast}}^{\ast}(\xi_{2})(a_{1})\Big)\\
=&\; \mathfrak{s}(r)^{\sharp}\big(\fr_{A\bowtie A^{\ast}}^{\ast}(a_{1}, \xi_{1})
(\xi_{2}, a_{2})\big).
\end{align*}
By Proposition \ref{pro:lr-inv}, we get $\mathfrak{s}(r)$ is $(\fl_{A\bowtie A^{\ast}},
\fr_{A\bowtie A^{\ast}})$-invariant. Thus, $(A\bowtie A^{\ast}, \Delta,
\alpha\oplus\beta^{\ast}, \beta\oplus\alpha^{\ast})$ is a quasi-triangular
averaging ASI bialgebra. Finally, note that $r^{\sharp}, r^{\natural}:\,
A^{\ast}\oplus A\rightarrow A\oplus A^{\ast}$ are given by
$$
r^{\sharp}(\xi, a)=(0, \xi),\qquad \mbox{and} \qquad r^{\natural}(\xi, a)=(-a, 0),
$$
for any $a\in A$ and $\xi\in A^{\ast}$.
This implies that $\mathcal{I}(\xi, a)=(a, \xi)$. Thus, $\mathcal{I}$ is a linear
isomorphism, and so that, $(A\bowtie A^{\ast}, \Delta, \alpha\oplus\beta^{\ast},
\beta\oplus\alpha^{\ast})$ is a factorizable averaging ASI bialgebra.
\end{proof}

\begin{ex}\label{ex:double}
Consider the averaging algebra $(A, \alpha)$ defined with respect to a basis
$\Bbbk\{e_{1}, e_{2}, e_{3}\}$ given by Example \ref{ex:YBE-bialg}, i.e. $e_{1}e_{1}=e_1$,
$e_{1}e_{2}=e_{2}=e_{2}e_{1}$ and $\alpha(e_{1})=\alpha(e_{2})=e_{3}$, $\alpha(e_{3})=0$.
We have a 3-dimensional averaging ASI bialgebra $(A, \Delta, \alpha, \beta)$, where
linear map $\beta: A\rightarrow A$ is given by $\beta(e_{1})=e_{3}$,
$\beta(e_{2})=-e_{3}$, $\beta(e_{3})=0$ and the nonzero comultiplication $\Delta:
A\rightarrow A\otimes A$ is given by $\Delta(e_{1})=-e_{2}\otimes e_{3}-e_{3}\otimes e_{2}$.
Now, denote $\{e^{\ast}_{1}, e^{\ast}_{2}, e^{\ast}_{3}\}$ the dual basis of
$\{e_{1}, e_{2}, e_{3}\}$ and $r=e_{1}\otimes e^{\ast}_{1}+e_{2}\otimes e^{\ast}_{2}
+e_{3}\otimes e^{\ast}_{3}\in A\otimes A^{\ast}\subset(A\oplus A^{\ast})
\otimes(A\oplus A^{\ast})$. Considering the double averaging ASI bialgebra constructed
in Proposition \ref{pro:doublebia}, we get a 6-dimensional averaging ASI bialgebra
$(A\bowtie A^{\ast}, \Delta_{A\bowtie A^{\ast}}, \alpha\oplus\beta^{\ast},
\beta\oplus\alpha^{\ast})$, where $\alpha^\ast(e^{\ast}_{1})=\alpha^\ast(e^{\ast}_{2})=0$,
$\alpha^\ast(e^{\ast}_{3})=e^{\ast}_{1}+e^{\ast}_{2}$, $\beta^\ast(e^{\ast}_{1})=
\beta^\ast(e^{\ast}_{2})=0$, $\beta^\ast(e^{\ast}_{3})=e^{\ast}_{1}-e^{\ast}_{2}$, and
the nonzero multiplication and comultiplication are given by
\begin{align*}
&\qquad\qquad e_{1}\ast e_{1}=e_{1},
\qquad\qquad\qquad\qquad\qquad\qquad\; e_{1}\ast e_{2}=e_{2}\ast e_{1}=e_{2}, \\
& \qquad\qquad e_{2}^{\ast}\ast e_{3}^{\ast}=e_{3}^{\ast}\ast e_{2}^{\ast}=-e^{\ast}_{1},
\qquad\qquad\qquad\qquad e_{1}\ast e^{\ast}_{1}=e^{\ast}_{1}\ast e_{1}=e^{\ast}_{1},\\
& e_{1}\ast e_{2}^{\ast}=e_{2}^{\ast}\ast e_{1}=e_{2}^{\ast}-e_{3}, \qquad\qquad
e_{1}\ast e_{3}^{\ast}=e^{\ast}_{3}\ast e_{1}=-e_{2}, \qquad\qquad\quad
e_{2}\ast e_{2}^{\ast}=e_{2}^{\ast}\ast e_{2}=e^{\ast}_{1},\\
&\Delta_{A\bowtie A^{\ast}}(e_{1})=-e_{2}\otimes e_{3}-e_{3}\otimes e_{2}, \qquad
\Delta_{A\bowtie A^{\ast}}(e_{2}^{\ast})=-e^{\ast}_{1}\otimes e_{2}^{\ast}
-e_{2}^{\ast}\otimes e^{\ast}_{1},\qquad
\Delta_{A\bowtie A^{\ast}}(e^{\ast}_{1})=-e^{\ast}_{1}\otimes e^{\ast}_{1}.
\end{align*}
Then, one can check $(A\bowtie A^{\ast}, \Delta_{A\bowtie A^{\ast}}, \alpha\oplus\beta^{\ast},
\beta\oplus\alpha^{\ast})$ is a quasi-triangular averaging ASI bialgebra.
Moreover, note that $r+\tau(r)=e_{1}\otimes e^{\ast}_{1}+e_{2}\otimes e^{\ast}_{2}
+e_{3}\otimes e^{\ast}_{3}+e^{\ast}_{1}\otimes e_{1}+e^{\ast}_{2}\otimes e_{2}
+e^{\ast}_{3}\otimes e_{3}$, we get that the linear map $\mathcal{I}:
(A\bowtie A^\ast)^\ast\rightarrow A\bowtie A^\ast$ is given by
$$
\mathcal{I}(e_{1}^{\star})=e^{\ast}_{1},\quad\mathcal{I}(e_{2}^{\star})=e^{\ast}_{2},
\quad\mathcal{I}(e_{3}^{\star})=e^{\ast}_{3},\quad\mathcal{I}((e^{\ast}_{1})^\star)=e_{1},
\quad\mathcal{I}((e^{\ast}_{2})^{\star})=e_{2},\quad\mathcal{I}((e^{\ast}_{3})^{\star})=e_{3},
$$
where $\{e_{1}^{\star}, e_{2}^{\star}, e_{3}^{\star}, (e^{\ast}_{1})^{\star},
(e^{\ast}_{2})^{\star}, (e^{\ast}_{3})^{\star}\}$ in $(A\bowtie A^\ast)^\ast$ is a
dual basis of $\{e_{1}, e_{2}, e_{3}, e^{\ast}_{1}, e^{\ast}_{2}, e^{\ast}_{3}\}$ in
$A\bowtie A^\ast$. Hence, we get $\mathcal{I}$ is a linear isomorphism and
$\mathcal{I}(\beta^{\ast}\oplus\alpha)=(\alpha\oplus\beta^{\ast})\mathcal{I}$. Thus,
$(A\bowtie A^{\ast}, \Delta_{A\bowtie A^{\ast}}, \alpha\oplus\beta^{\ast},
\beta\oplus\alpha^{\ast})$ is a factorizable averaging ASI bialgebra.
\end{ex}

Next, we will give a characterization of factorizable averaging ASI bialgebras
by the Rota-Baxter operator on symmetric averaging Frobenius algebras.
Let $(A, \alpha)$ be an averaging algebra and $\lambda\in\Bbbk$. A linear map
$R: A\rightarrow A$ is called a {\it Rota-Baxter operator of weight $\lambda$}
on $(A, \alpha)$ if
\begin{align*}
R\alpha&=\alpha R,\\
R(a_{1})R(a_{2})&=R\big(R(a_{1})a_{2}+a_{1}R(a_{2})+\lambda a_{1}a_{2}\big),
\end{align*}
for any $a_{1}, a_{2}\in A$. Clearly, a Rota-Baxter operator of weight $0$ on $(A, \alpha)$
is an $\mathcal{O}$-operator of $(A, \alpha)$ associated to the regular bimodule.
Let $(A, \alpha, P)$ be an averaging algebra $(A, \alpha)$ with a Rota-Baxter
operator $R$ of weight $\lambda$. Then there is a new multiplication $\cdot_{R}$
on $A$ defined by
$$
a_{1}\cdot_{R}a_{2}=R(a_{1})a_{2}+a_{1}R(a_{2})+\lambda a_{1}a_{2},
$$
for any $a_{1}, a_{2}\in A$. Then, one can check that $(A, \cdot_{R}, \alpha)$
is also an averaging algebra, and $R$ is an averaging algebra homomorphism
from $(A, \cdot_{R}, \alpha)$ to $(A, \cdot, \alpha)$.

\begin{defi}\label{def:RB-aymm}
A linear map $R: A\rightarrow A$ is called a {\rm Rota-Baxter operator of weight
$\lambda$ on a symmetric averaging Frobenius algebra} $(A, \alpha, \mathfrak{B})$,
if $R$ is a Rota-Baxter operator of weight $\lambda$ on averaging algebra $(A, \alpha)$
and for any $a_{1}, a_{2}\in A$,
$$
\mathfrak{B}(R(a_{1}),\, a_{2})+\mathfrak{B}(a_{1},\, R(a_{2}))
+\lambda\mathfrak{B}(a_{1},\, a_{2})=0.
$$
\end{defi}

Let $R$ be a Rota-Baxter operator of weight $\lambda$ on a symmetric averaging
Frobenius algebra $(A, \alpha, \mathfrak{B})$. If we define a linear map
$\hat{R}: A\rightarrow A$ by $\mathfrak{B}(\hat{R}(a_{1}),\, a_{2})=
\mathfrak{B}(a_{1},\, R(a_{2}))$ for all $a_{1}, a_{2}\in A$.
Then, the equation in Definition \ref{def:RB-aymm} is equivalent to $R+\hat{R}+\lambda\id=0$.
The following theorem shows that there is a one-to-one correspondence between
factorizable averaging ASI bialgebras and symmetric averaging Frobenius algebras
with a Rota-Baxter operator of weight $\lambda$.

\begin{thm}\label{thm:RB-aymm}
Let $(A, \alpha)$ be an averaging algebra and $r\in A\otimes A$.
Suppose the averaging ASI bialgebra $(A, \Delta, \alpha, \beta)$ induced by $r$ is
factorizable and $\mathcal{I}=r^{\sharp}-r^{\natural}$. We define a bilinear form
$\mathfrak{B}_{\mathcal{I}}$ by $\mathfrak{B}_{\mathcal{I}}(a_{1}, a_{2})
=\langle\mathcal{I}^{-1}(a_{1}),\; a_{2}\rangle$, for any $a_{1}, a_{2}\in A$.
Then $(A, \alpha, \mathfrak{B}_{\mathcal{I}})$ is a symmetric averaging Frobenius algebra.
Moreover, the linear map $R=\lambda r^{\natural}\mathcal{I}^{-1}:\, A\rightarrow A$
is a Rota-Baxter operator of weight $\lambda$ on $(A, \alpha, \mathfrak{B}_{\mathcal{I}})$.

Conversely, for any symmetric averaging Frobenius algebra $(A, \alpha,
\mathfrak{B})$ with a Rota-Baxter operator $R$ of weight $\lambda$, we have
a linear isomorphism $\mathcal{I}_{\mathfrak{B}}: A^{\ast}\rightarrow A$ by
$\langle\mathcal{I}^{-1}_{\mathfrak{B}}(a_{1}),\; a_{2}\rangle=\mathfrak{B}
(a_{1}, a_{2})$, for any $a_{1}, a_{2}\in A$. If $\lambda\neq 0$, we define
$$
r^{\sharp}:=\mbox{$\frac{1}{\lambda}$}(R+\lambda\id)\mathcal{I}_{\mathfrak{B}}:\quad
A^{\ast}\longrightarrow A,
$$
and define $r\in A\otimes A$ by $\langle \xi\otimes\eta,\; r\rangle
=\langle\eta,\, r^{\sharp}(\xi)\rangle$, for any $\xi, \eta\in A^{\ast}$.
Then, $r$ satisfies the $\hat{\alpha}$-YBE in $(A, \alpha)$, and gives rise to
a factorizable averaging ASI bialgebra $(A, \Delta, \alpha, \hat{\alpha})$,
where $\Delta$ is given by Eq. (\ref{cobo}) and $\hat{\alpha}$ is the adjoint linear
operator of $\alpha$ under the nondegenerate bilinear form $\mathfrak{B}$.
\end{thm}

\begin{proof}
If the averaging ASI bialgebra $(A, \Delta, \alpha, \beta)$ induced by $r$ is
factorizable, then ASI bialgebra $(A, \Delta)$ is factorizable. By \cite[Theorem 4.6]{SW},
we get $\mathfrak{B}_{\mathcal{I}}(-,-)$ is a nondegenerate symmetric invariant
bilinear form on $A$, and $P$ is a Rota-Baxter operator of weight $\lambda$ on
symmetric averaging Frobenius algebra $(A, \mathfrak{B}_{\mathcal{I}})$.
Moreover, since $r^{\sharp}, r^{\natural}:\, (A^{\ast}, \cdot_r, \beta^{\ast})
\rightarrow(A, \cdot, \alpha)$ are averaging algebra homomorphisms, we get
$$
R\alpha=\lambda r^{\natural}\mathcal{I}^{-1}\alpha
=\lambda r^{\natural}\beta^{\ast}\mathcal{I}^{-1}
=\alpha\lambda r^{\natural}\mathcal{I}^{-1}=\alpha P.
$$
Thus, $R$ is a Rota-Baxter operator of weight $\lambda$ on $(A, \alpha,
\mathfrak{B}_{\mathcal{I}})$.

Conversely, since $\mathfrak{B}$ is symmetric, we have $\mathcal{I}_{\mathfrak{B}}^{\ast}
=\mathcal{I}_{\mathfrak{B}}$. Note that $\mathfrak{B}(a_{1},\, R(a_{2}))
+\mathfrak{B}(R(a_{1}),\, a_{2})+\lambda\mathfrak{B}(a_{1}, a_{2})=0$, that is,
$\langle\mathcal{I}_{\mathfrak{B}}^{-1}(a_{1}),\, R(a_{2})\rangle
+\langle\mathcal{I}_{\mathfrak{B}}^{-1}R(a_{1}),\, a_{2}\rangle
+\lambda\langle\mathcal{I}_{\mathfrak{B}}^{-1}(a_{1}),\, a_{2}\rangle=0$, we get
$R^{\ast}\mathcal{I}_{\mathfrak{B}}^{-1}+\mathcal{I}_{\mathfrak{B}}^{-1}R
+\lambda\mathcal{I}_{\mathfrak{B}}^{-1}=0$, and so that,
$\mathcal{I}_\mathfrak{B}R^{\ast}+R\mathcal{I}_\mathfrak{B}
+\lambda\mathcal{I}_{\mathfrak{B}}=0$. Thus, we have
$r^{\natural}=-(r^{\sharp})^{\ast}=-\frac{1}{\lambda}\big(\mathcal{I}_{\mathfrak{B}}R^{\ast}
+\lambda\mathcal{I}_{\mathfrak{B}}\big)=\frac{1}{\lambda}R\mathcal{I}_{\mathfrak{B}}$
and $\mathcal{I}_{\mathfrak{B}}=r^{\sharp}-r^{\natural}$, if $r\in A\otimes A$ is
defined as above. Since $\mathcal{I}_{\mathfrak{B}}$ is a linear isomorphism, we only
need to show that $(A, \Delta, \alpha, \hat{\alpha})$ is a quasi-triangular averaging ASI
bialgebra. By \cite[Theorem 4.6]{SW} again, we get $\mathfrak{s}(r)$ is $(\fl, \fr)$-invariant
and $r$ is a solution of Eq. \eqref{ybe1}. Thus, linear maps $r^{\sharp}, r^{\natural}:\,
(A^{\ast}, \cdot_{r})\rightarrow(A, \cdot)$ are associative algebra homomorphisms.
For any $a_{1}, a_{2}\in A$, since
$$
\langle\mathcal{I}_{\mathfrak{B}}^{-1}(\alpha(a_{1})),\; a_{2}\rangle
=\mathfrak{B}(\alpha(a_{1}),\; a_{2})=\mathfrak{B}(a_{1},\; \hat{\alpha}(a_{2})
=\langle\mathcal{I}_{\mathfrak{B}}^{-1}(a_{1}),\; \hat{\alpha}(a_{2})\rangle
=\langle\hat{\alpha}^{\ast}(\mathcal{I}_{\mathfrak{B}}^{-1}(a_{1})),\; a_{2}\rangle,
$$
we get $\mathcal{I}_{\mathfrak{B}}^{-1}\alpha=\hat{\alpha}^{\ast}
\mathcal{I}_{\mathfrak{B}}^{-1}$. Note that $r^{\sharp}=\frac{1}{\lambda}(R+\lambda\id)
\mathcal{I}_{\mathfrak{B}}$ and $R\alpha=\alpha R$, we have
$$
r^{\sharp}\hat{\alpha}^{\ast}=\mbox{$\frac{1}{\lambda}$}(R+\lambda\id)
\mathcal{I}_{\mathfrak{B}}\hat{\alpha}^{\ast}=\mbox{$\frac{1}{\lambda}$}(R+\lambda\id)
\alpha\mathcal{I}_{\mathfrak{B}}=\alpha\big(\mbox{$\frac{1}{\lambda}$}(R+\lambda\id)
\mathcal{I}_{\mathfrak{B}}\big)=\alpha r^{\sharp}.
$$
That is to say, $r^{\sharp}: (A^{\ast}, \cdot_{r}, \hat{\alpha}^{\ast})\rightarrow
(A, \cdot, \alpha)$ is an averaging algebra homomorphism. Similarly, one can check
that $r^{\natural}: (A^{\ast}, \cdot_{r}, \hat{\alpha}^{\ast})\rightarrow
(A, \cdot, \alpha)$ is also an averaging algebra homomorphism.
Thus, by Proposition \ref{pro:lr-inv-equ}, $r$ is a solution of the
$\hat{\alpha}$-YBE in averaging algebra $(A, \alpha)$, and so that,
$(A, \Delta, \alpha, \hat{\alpha})$ is a factorizable averaging ASI bialgebra.
\end{proof}

\begin{cor}\label{cor:oper}
Let $(A, \alpha)$ be an averaging algebra and $r\in A\otimes A$.
Suppose the averaging ASI bialgebra $(A, \Delta, \alpha, \beta)$ induced by $r$ is
factorizable and $\mathcal{I}=r^{\sharp}-r^{\natural}$. Then $-\lambda\id-R$ is also a
Rota-Baxter operator of weight $\lambda$ on symmetric averaging Frobenius
algebra $(A, \alpha, \mathfrak{B}_{\mathcal{I}})$, where $\mathfrak{B}_{\mathcal{I}}$
and $R$ are defined in Theorem \ref{thm:RB-aymm}.
\end{cor}

\begin{proof}
It can be obtained by direct calculations.
\end{proof}

Let $(A, \Delta, \alpha, \beta)$ and $(A', \Delta', \alpha', \beta')$ be two
averaging ASI bialgebras. A linear map $f: A\rightarrow A'$ is called a
{\it homomorphism of averaging ASI bialgebras} if $f: (A, \alpha)\rightarrow
(A', \alpha')$ is a homomorphism of averaging algebras and satisfies
$$
(f\otimes f)\Delta=\Delta'f,\qquad\mbox{ and }\qquad (f\otimes f)\beta=\beta'f.
$$
Moreover, if $f$ is a bijection, we call $f: (A, \Delta, \alpha, \beta)\rightarrow
(A', \Delta', \alpha', \beta')$ is an {\it isomorphism of averaging ASI bialgebras}.

\begin{cor}\label{cor:oper1}
Let $(A, \alpha)$ be an averaging algebra and $r\in A\otimes A$.
Suppose the averaging ASI bialgebra $(A, \Delta, \alpha, \beta)$ induced by $r$ is
factorizable, $\mathcal{I}=r^{\sharp}-r^{\natural}$ and $R=\lambda r^{\natural}
\mathcal{I}^{-1}:\, A\rightarrow A$, where $0\neq\lambda\in\Bbbk$. Then $(A, \cdot_{R},
\Delta_{\mathcal{I}}, \alpha, \beta)$ is an averaging ASI bialgebra, where
$$
\Delta_{\mathcal{I}}^{\ast}(\xi, \eta)=\mbox{$\frac{1}{\lambda}$}
\mathcal{I}^{-1}\big(\mathcal{I}(\xi)\mathcal{I}(\eta)\big),
$$
for any $\xi, \eta\in A^{\ast}$.	
Moreover, $\frac{1}{\lambda}\mathcal{I}: A^{\ast}\rightarrow A$ gives an averaging
ASI bialgebra isomorphism from $(A^{\ast}, \cdot_{r}, \Delta_{A^{\ast}}, \beta^{\ast},
\alpha^{\ast})$ to $(A, \cdot_{R}, \Delta_{\mathcal{I}}, \alpha, \beta)$, where
$\Delta_{A^{\ast}}^{\ast}(a_{1}, a_{2})=a_{1}a_{2}$ for any $a_{1}, a_{2}\in A$.
\end{cor}

\begin{proof}
By \cite[Corollary 4.9]{SW}, $\frac{1}{\lambda}\mathcal{I}: (A^{\ast}, \cdot_{r})
\rightarrow(A, \cdot_{R})$ is an isomorphism of associative algebras.
By Proposition \ref{pro:lr-inv-equ}, we get $(A^{\ast}, \cdot_{r}, \beta^{\ast})$
is an averaging algebra and $\mathcal{I}\beta^{\ast}=(r^{\sharp}-r^{\natural})\beta^{\ast}
=\alpha(r^{\sharp}-r^{\natural})=\alpha\mathcal{I}$. That is, $\frac{1}{\lambda}\mathcal{I}:
(A^{\ast}, \cdot_{r}, \beta^{\ast})\rightarrow(A, \cdot_{R}, \alpha)$ is an
isomorphism of averaging associative algebras. Moreover, since
$(\frac{1}{\lambda}\mathcal{I})^{\ast}=\frac{1}{\lambda}\mathcal{I}$,
we get $(\frac{1}{\lambda}\mathcal{I})^{\ast}(\Delta_{\mathcal{I}}^{\ast}(\xi, \eta))
=\big(\frac{1}{\lambda}\mathcal{I}(\xi)\big)\big(\frac{1}{\lambda}\mathcal{I}(\eta)\big)
=\big((\frac{1}{\lambda}\mathcal{I})^{\ast}(\xi)\big)
\big((\frac{1}{\lambda}\mathcal{I})^{\ast}(\eta)\big)$,
which means the map $(\frac{1}{\lambda}\mathcal{I}):\, (A, \Delta_{A^{\ast}},
\alpha^{\ast})\rightarrow(A^{\ast}, \Delta_{\mathcal{I}}, \beta)$ is an averaging
coalgebra isomorphism. Therefore, $(A, \cdot_{R}, \Delta_{\mathcal{I}}, \alpha, \beta)$
is also an averaging ASI bialgebra and $\frac{1}{\lambda}\mathcal{I}$ is
an isomorphism of averaging ASI bialgebra.
\end{proof}

\section{Perm bialgebras via commutative and cocommutative averaging
ASI bialgebras}\label{sec:permbia}

A perm algebra is a vector space $P$ with a bilinear operation such that
$p_{1}(p_{2}p_{3})=(p_{1}p_{2})p_{3}=(p_{2}p_{1})p_{3}$, for any $p_{1}, p_{2}, p_{3}\in P$.
Let $A$ be a commutative associative algebra and $\alpha: A\rightarrow A$ be an
averaging operator on $A$. Define a new binary operations
$\bullet: A\otimes A\rightarrow A$ by
$$
a_{1}\bullet a_{2}=\alpha(a_{1})a_{2},
$$
for all $a_{1}, a_{2}\in A$. Then $(A, \bullet)$ is a perm algebra, which is called
the {\it perm algebra induced by commutative averaging algebra} $(A, \alpha)$.
We generalize this construction to the context of bialgebras, that is,
we construct perm bialgebras from commutative and cocommutative averaging ASI bialgebras.
We establish the explicit relationships between them, as well as the equivalent
interpretation in terms of the corresponding double constructions and matched pairs.

In the following two sections, we always assume that $A$ is a commutative
associative algebra. In this case, we use $(M, \mu)$ to denote a bimodule
$(M, \kl, \kr)$ over $A$ since $\mu:=\kl=\kr$, and call $(M, \mu)$ a module
over $A$. Then $(M^{\ast}, \mu^{\ast})$ is again a module over $A$.
A bimodule over commutative averaging algebra is also called module over
commutative averaging algebra.

\subsection{Induced matched pairs and induced Manin triples of perm algebras}
\label{subsec:perm-mat}
We introduce the bimodule over a perm algebra.

\begin{defi}\label{def:perm-mod}
A {\rm bimodule of a perm algebra} $(P, \bullet)$ is a triple $(M, \kl, \kr)$, where $M$ is a
vector space and $\kl, \kr: P\rightarrow\End_{\Bbbk}(M)$ are linear maps satisfying
\begin{align*}
&\qquad\quad\kl(p_{1}\bullet p_{2})=\kl(p_{1})\kl(p_{2})=\kl(p_{2})\kl(p_{1}), \\
&\kr(p_{1}\bullet p_{2})=\kr(p_{2})\kr(p_{1})=\kr(p_{2})\kl(p_{1})=\kl(p_{1})\kr(p_{2}),
\end{align*}
for any $p_{1}, p_{2}\in P$.
\end{defi}

Let $(P, \bullet)$ be a perm algebra. Define $\tilde{\fl}_{P}, \tilde{\fr}_{P}:
P\rightarrow\End_{\Bbbk}(P)$ by $\tilde{\fl}_{P}(p_{1})(p_{2})=p_{1}\bullet p_{2}$
and $\tilde{\fr}_{P}(p_{1})(p_{2})=p_{2}\bullet p_{1}$, for any $p_{1}, p_{2}\in P$.
Obviously, $(P, \tilde{\fl}_{P}, \tilde{\fr}_{P})$ is a bimodule of perm algebra
$(P, \bullet)$, which is called the {\it regular bimodule} over $P$.
More generally, for a bimodule over a perm algebra, we have

\begin{pro}\label{pro:assda1}
Let $(P, \bullet)$ be a perm algebra, $M$ be a vector space and $\kl, \kr: P\rightarrow
\End_{\Bbbk}(M)$ be linear maps. Define a binary operation on $P\oplus M$ by
$$
(p_{1}, m_{1})(p_{2}, m_{2})=\Big(p_{1}\bullet p_{2},\;
\kl(p_{1})(m_{2})+\kr(p_{2})(m_{1})\Big),
$$
for any $p_{1}, p_{2}\in P$ and $m_{1}, m_{2}\in M$.
Then, $(M, \kl, \kr)$ is a bimodule over $P$ if and only if $P\oplus M$ with the product
as above is a perm algebra. In such a case, we call this perm algebra
the {\rm semidirect product} perm algebra of $(P, \bullet)$ by bimodule $(M, \kl, \kr)$,
and denote it by $P\ltimes M$.
\end{pro}

Considering the bimodule structure on dual spaces, by straightforward verification,
we have the following lemma.

\begin{lem}\label{lem:pdual-mod}
Let $(P, \bullet)$ be a perm algebra, $(M, \kl, \kr)$ be a bimodule over it.
Then $(M^{\ast}, \kl^{\ast}, \kl^{\ast}-\kr^{\ast})$ is also a bimodule over $(P, \bullet)$,
which is called the {\rm dual bimodule} of $(M, \kl, \kr)$.
In particular, $(P^{\ast}, \tilde{\fl}_{P}^{\ast}, \tilde{\fl}_{P}^{\ast}
-\tilde{\fr}_{P}^{\ast})$ is a bimodule over $(P, \bullet)$.
\end{lem}

A commutative averaging algebra induces a perm algebra. For a module over
a commutative averaging algebra, we also have the corresponding conclusion.

\begin{pro}\label{pro:assda2}
Let $(A, \alpha)$ be a commutative averaging algebra, and $(A, \bullet)$ be the
induced perm algebra. For any module $(M, \mu, \beta)$ over $(A, \alpha)$, we define
\begin{align*}
\kl: A\rightarrow\End_{\Bbbk}(M),\qquad \kl(a)(m)=\mu(\alpha(a))(m),\\
\kr: A\rightarrow\End_{\Bbbk}(M),\qquad \kr(a)(m)=\mu(a)(\beta(m)),
\end{align*}
for any $a\in A$ and $m\in M$. Then, $(M, \kl, \kr)$ is a bimodule over $(A, \bullet)$.
\end{pro}

\begin{proof}
Since for any $a\in A$ and $m\in M$, $\mu(\alpha(a))(\beta(m))=
\beta(\mu(\alpha(a))(m))=\beta(\mu(a)(\beta(m)))$, we get
\begin{align*}
&\kl(a_{1}\bullet a_{2})(m)=\mu(\alpha(a_{1}\bullet a_{2}))(m)
=\mu(\alpha(\alpha(a_{1})a_{2})(m))=\mu(\alpha(a_{1})\alpha(a_{2}))(m),\\
&\qquad\; \kl(a_{1})(\kl(a_{2})(m))=\mu(\alpha(a_{1}))(\mu(\alpha(a_{2}))(m))
=\mu(\alpha(a_{1})\alpha(a_{2}))(m),\\
&\qquad\; \kl(a_{2})(\kl(a_{1})(m))=\mu(\alpha(a_{2}))(\mu(\alpha(a_{2}))(m))
=\mu(\alpha(a_{1})\alpha(a_{2}))(m).
\end{align*}
for any $a_{1}, a_{2}\in A$ and $m\in M$,
That is, $\kl(a_{1}\bullet a_{2})=\kl(a_{1})\kl(a_{2})=\kl(a_{2})\kl(a_{1})$.
Similarly, we also have $\kr(a_{1}\bullet a_{2})=\kr(a_{2})\kr(a_{1})=
\kr(a_{2})\kl(a_{1})=\kl(a_{1})\kr(a_{2})$. Thus, $(M, \kl, \kr)$ is a
bimodule over $(A, \bullet)$.
\end{proof}

The bimodule $(M, \kl, \kr)$ in the proposition above is called the {\it induced bimodule}
by module $(M, \mu, \beta)$. Let $(A, \alpha)$ be a commutative averaging algebra
and $(A, \bullet)$ be the induced perm algebra of $(A, \alpha)$.
The bimodule $(A, \kl, \kr)$ over $(A, \bullet)$ induced by the regular module
$(A, \fl_{A}, \alpha)$ is exactly the regular bimodule $(A, \tilde{\fl}_{A},
\tilde{\fr}_{A})$ over $(A, \bullet)$.

\begin{pro}\label{pro:assdualrep}
Let $(A, \alpha)$ be a commutative averaging algebra and $(A, \bullet)$ be the
induced perm algebra. Suppose that $(M, \mu, \beta)$ is a module over $(A, \alpha)$,
and $(M, \kl, \kr)$ is the bimodule over $(A, \bullet)$ induced by $(M, \mu, \beta)$.
Then, the dual bimodule $(M^{\ast}, \kl^{\ast}, \kl^{\ast}-\kr^{\ast})$ of $(M, \kl, \kr)$
is the induced bimodule over the perm algebra $(A, \bullet)$
by $(M^{\ast}, \mu^{\ast}, \beta^{\ast})$ if and only if
$$
\beta(\mu(a)(m))=\mu(\alpha(a))(m)-\mu(a)(\beta(m)),
$$
for any $a\in A$ and $m\in M$.

In particular, when taking $(M, \mu, \beta)=(A, \fl_{A}, \beta)$, we get that
$(A^{\ast}, \tilde{\fl}_{A}^{\ast}, \tilde{\fl}_{A}^{\ast}-\tilde{\fr}_{A}^{\ast})$
is induced by $(A^{\ast}, \fl_{A}^{\ast}, \beta^{\ast})$ if and only if
$\beta(a_{1}a_{2})=\alpha(a_{1})a_{2}-a_{1}\beta(a_{2})$, for any $a_{1}, a_{2}\in A$.
\end{pro}

\begin{proof}
We only prove the first part. Denote by $(M^{\ast}, \kl', \kr')$ the bimodule
over $(A, \bullet)$ induced by $(M^{\ast}, \mu^{\ast}, \beta^{\ast})$.
Since $(M, \kl, \kr)$ is induced by $(M, \mu, \beta)$, for any $a\in A$, $m\in M$,
$\xi\in M^{\ast}$, we have
\begin{align*}
\langle\kl'(a)(\xi),\; m\rangle&=\langle\mu^{\ast}(\alpha(a))(\xi),\; m\rangle
=\langle\xi,\; \mu(\alpha(a))(m)\rangle=\langle\xi,\; \kl(a)(m)\rangle
=\langle\kl^{\ast}(a)(\xi),\; m\rangle,\\
\langle\kr'(a)(\xi),\; m\rangle
&=\langle\mu^{\ast}(a)(\beta^{\ast}(\xi)),\; m\rangle
=\langle\xi,\; \beta(\mu(a)(m))\rangle\\
\langle(\kl^{\ast}-\kr^{\ast})(a)(\xi),\; m\rangle
&=\langle\xi,\; \kl(a)(m)\rangle-\langle\xi,\; \kr(a)(m)\rangle
=\langle\xi,\; \mu(\alpha(a))(m)\rangle-\langle\xi,\; \mu(a)(\beta(m))\rangle,
\end{align*}
Thus, $(M^{\ast}, \kl', \kr')=(M^{\ast}, \kl^{\ast}, \kl^{\ast}-\kr^{\ast})$ if and
only if $\beta(\mu(a)(m))=\mu(\alpha(a))(m)-\mu(a)(\beta(m))$,
for any $a\in A$ and $m\in M$.
\end{proof}

Now we consider the relationship between the matched pair of commutative averaging algebras
and the matched pair of induced perm algebras.

\begin{defi}\label{def:matperm}
A {\rm matched pair of perm algebras} consists of perm algebras $(A, \bullet)$
and $(B, \bullet)$, together with linear maps $\kl_{A}, \kr_{A}:
A\rightarrow\End_{\Bbbk}(B)$ and $\kl_{B}, \kr_{B}: B\rightarrow\End_{\Bbbk}(A)$ such
that $A\oplus B$ is a perm algebra, where the multiplication is defined by
$$
(a_{1}, b_{1})(a_{2}, b_{2}):=\big(a_{1}\bullet a_{2}+\kl_{B}(b_{1})(a_{2})
+\kr_{B}(b_{2})(a_{1}),\quad b_{1}\bullet b_{2}+\kl_{A}(a_{1})(b_{2})
+\kr_{A}(a_{2})(b_{1})\big),
$$
for any $a_{1}, a_{2}\in A$ and $b_{1}, b_{2}\in B$. The matched pair of perm
algebras is denoted by $((A, \bullet), (B, \bullet)$, $\kl_{A}, \kr_{A}, \kl_{B}, \kr_{B})$
and the resulting perm algebra structure on $A\oplus B$ is denoted by
$(A, \bullet)\bowtie(B, \bullet)$.
\end{defi}

In a matched pair of perm algebras $((A, \bullet), (B, \bullet), \kl_{A}, \kr_{A},
\kl_{B}, \kr_{B})$, $(A, \kl_{B}, \kr_{B})$ is a bimodule over perm algebra
$(B, \bullet)$ and $(B, \kl_{A}, \kr_{A})$ is a bimodule over perm algebra
$(A, \bullet)$. In particular, if the multiplication of $B$ is trivial,
resulting perm algebra is just the semidirect product of $(A, \bullet)$ by
bimodule $(B, \kl_{A}, \kr_{A})$ defined in Proposition \ref{pro:assda2}.

\begin{pro}\label{pro:assmat}
Let $((A, \alpha), (B, \beta), \mu_{A}, \mu_{B})$ be a matched pair of
commutative averaging algebras, $(A, \bullet)$ and $(B, \bullet)$ be the
induced perm algebras of $(A, \alpha)$ and $(B, \beta)$ respectively.
Then $((A, \bullet), (B, \bullet)$, $\kl_{A}, \kr_{A}, \kl_{B}, \kr_{B})$ is a
matched pair of perm algebras, called {\rm the induced matched pair of perm algebras}
by $((A, \alpha), (B, \beta), \mu_{A}, \mu_{B})$, where $(B, \kl_{A}, \kr_{A})$
is the induced bimodule over $(A, \bullet)$ by $(B, \mu_{A}, \beta)$ and
$(A, \kl_{B}, \kr_{B})$ is the induced bimodule over $(B, \bullet)$ by $(A, \mu_B, \alpha)$.

Moreover, the perm algebra $(A, \bullet)\bowtie(B, \bullet)$ obtained by the matched
pair $((A, \bullet), (B, \bullet), \kl_{A}, \kr_{A},$ $ \kl_{B}, \kr_{B})$
is exactly the perm algebra induced by the commutative averaging algebra
$(A\bowtie B,\; \alpha\oplus\beta)$ obtained by the matched pair
$((A, \alpha), (B, \beta), \mu_{A}, \mu_{B})$.
\end{pro}

\begin{proof}
Let $(A\bowtie B,\; \alpha\oplus\beta)$ be the commutative averaging algebra obtained
by the matched pair $((A, \alpha), (B, \beta), \mu_{A}, \mu_{B})$, and
$(A\oplus B, \diamond)$ be the perm algebra induced by $(A\bowtie B,\; \alpha\oplus\beta)$.
Note that, for any $a_{1}, a_{2}\in A$ and $b_{1}, b_{2}\in B$,
\begin{align*}
&\;(a_{1}, b_{1})\diamond(a_{2}, b_{2})\\
=&\;(\alpha(a_{1}), \beta(b_{1}))(a_{2}, b_{2})\\
=&\;\Big(\alpha(a_{1})a_{2}+\mu_{B}(\beta(b_{1}))(a_{2})+\mu_{B}(b_{2})(\alpha(a_{1})),\ \
\beta(b_{1})b_{2}+\mu_{A}(\alpha(a_{1}))(b_{2})+\mu_{A}(a_{2})(\beta(b_{1}))\Big)\\
=&\;\Big(a_{1}\bullet a_{2}+\kl_{B}(b_{1})(a_{2})+\kr_{B}(b_{2})(a_{1}),\ \
b_{1}\bullet b_{2}+\kl_{A}(a_{1})(b_{2})+\kr_{A}(a_{2})(b_{1})\Big),
\end{align*}
we get $((A, \bullet), (B, \bullet), \kl_{A}, \kr_{A}, \kl_{B}, \kr_{B})$ is a
matched pair. And we have already shown that the perm algebra structure on
$A\oplus B$ obtained from this induced matched pair of perm algebras is exactly
the induced perm algebra by $(A\bowtie B\; \alpha\oplus\beta)$.
\end{proof}

In particular, if the multiplication of $B$ is trivial, we have the following corollary.

\begin{cor}\label{cor:asssemi}
Let $(A, \alpha)$ be a commutative averaging algebra and $(A, \bullet)$ be the
perm algebra induced by $(A, \alpha)$, and $(M, \mu, \beta)$ be a bimodule over
$(A, \alpha)$. Suppose that $(A\ltimes M, \alpha\oplus\beta)$ is the
semidirect product (commutative) averaging algebra of $(A, \alpha)$ by $(M, \mu, \beta)$.
Then the semidirect product perm algebra of $(A, \bullet)$ by $(M, \kl, \kr)$,
where $(M, \kl, \kr)$ is the induced bimodule by $(M, \mu, \beta)$,
is exactly the perm algebra induced by $(A\ltimes M, \alpha\oplus\beta)$.
\end{cor}

Moreover, by Propositions \ref{pro:assmat} and \ref{pro:assdualrep}, we have

\begin{cor}\label{cor:adjassmat}
Let $(A, \alpha)$ be a commutative averaging algebra. Suppose that there is a
linear map $\beta: A\rightarrow A$ such that $(A^{\ast}, \beta^{\ast})$ is a commutative
averaging algebra and $((A, \alpha),\; (A^{\ast},\; \beta^{\ast}),\; \fl_{A}^{\ast}, $ $
\fl_{A^{\ast}}^{\ast})$ is a matched pair of commutative averaging algebras.
Denote by $(A, \bullet)$ and $(A^{\ast}, \bullet)$ the perm algebras
induced by $(A, \alpha)$ and $(A^{\ast}, \beta^{\ast})$ respectively.
Then, $((A, \bullet)$, $(A^{\ast}, \bullet),\; \tilde{\fl}_{A}^{\ast},\;
\tilde{\fl}_{A}^{\ast}-\tilde{\fr}_{A}^{\ast},\; \tilde{\fl}_{A^{\ast}}^{\ast},\;
\tilde{\fl}_{A^{\ast}}^{\ast}-\tilde{\fr}_{A^{\ast}}^{\ast})$ is a matched pair of
perm algebras such that it is the induced matched pair by $((A, \alpha),\; (A^{\ast},
\beta^{\ast}),\; \fl_{A}^{\ast},\; \fl_{A^{\ast}}^{\ast})$ if and only if for any
$a_{1}, a_{2}\in A$, $\xi_{1}, \xi_{2}\in A^{\ast}$,
$$
\beta(a_{1}a_{2})=\alpha(a_{1})a_{2}-a_{1}\beta(a_{2})\qquad \mbox{and}\qquad
\alpha^{\ast}(\xi_{1}\cdot_{A^{\ast}}\xi_{2})=\beta^{\ast}(\xi_{1})\cdot_{A^{\ast}}
\xi_{2}-\xi_{1}\cdot_{A^{\ast}}\alpha^{\ast}(\xi_{2}).
$$
\end{cor}


Next, we consider the Manin triples of perm algebras induced by the double construction
of averaging Frobenius algebras. Recall that a bilinear form $\tilde{\mathfrak{B}}(-,-)$
on a perm algebra $(P, \bullet)$ is called {\it invariant} if
$$
\tilde{\mathfrak{B}}(p_{1}\bullet p_{2},\; p_{3})
=\tilde{\mathfrak{B}}(p_{1},\; p_{2}\bullet p_{3}-p_{3}\bullet p_{2}),
$$
for any $p_{1}, p_{2}, p_{3}\in P$.

\begin{defi}[\cite{LZB,Hou}]\label{def:perm-main}
A {\rm Manin triple of perm algebras} is a triple $((P, \diamond,
\tilde{\mathfrak{B}}),\; (P^{+}, \bullet),\; (P^{-}, \bullet))$, where $(P, \diamond)$
is a perm algebra and $\tilde{\mathfrak{B}}(-,-)$ is a nondegenerate antisymmetric
invariant bilinear form on $(P, \diamond)$ such that:
\begin{enumerate}\itemsep=0pt
\item[$(i)$] $(P^{+}, \bullet)$ and $(P^{-}, \bullet)$ are perm subalgebras of
     $(P, \diamond)$;
\item[$(ii)$] $P=P^{+}\oplus P^{-}$ as vector spaces;
\item[$(iii)$] $P^{+}$ and $P^{-}$ are isotropic with respect to $\tilde{\mathfrak{B}}(-,-)$.
\end{enumerate}
\end{defi}

\begin{pro}\label{pro:assFro}
Let $(A, \alpha)$ be a commutative averaging algebra and $(A, \bullet)$ be the induced
perm algebra by $(A, \alpha)$. Suppose that there is a linear map $\beta: A\rightarrow A$
such that $(A^{\ast}, \beta^{\ast})$ is a commutative averaging algebra, and $(A^{\ast},
\bullet)$ is the induced perm algebra by $(A^{\ast}, \beta^{\ast})$. Then there is a Manin
triple $((A\oplus A^{\ast}, \diamond, \tilde{\mathfrak{B}}_{d}),\; (A, \bullet),\;
(A^{\ast}, \bullet))$ of perm algebras such that the perm algebra $(A\oplus A^{\ast},
\diamond)$ is induced by $(A\bowtie A^{\ast}, \alpha\oplus\beta^{\ast})$, if and only if
for any $a_{1}, a_{2}\in A$,
\begin{align}
\beta(a_{1}a_{2})&=\alpha(a_{1})a_{2}-a_{1}\alpha(a_{2}),   \label{compa1}\\
\Delta\alpha&=(\beta\otimes\id)\Delta-(\id\otimes\beta)\Delta,  \label{compa2}
\end{align}
where the bilinear form $\tilde{\mathfrak{B}}_{d}(-,-)$ on $A\oplus A^{\ast}$ is
defined by $\tilde{\mathfrak{B}}_{d}((a_{1}, \xi_{1}),\; (a_{2}, \xi_{2})):=
\langle\xi_{2},\; a_{1}\rangle-\langle\xi_{1},\; a_{2}\rangle$, for any
$a_{1}, a_{2}\in A$ and $\xi_{1}, \xi_{2}\in A^{\ast}$.
\end{pro}

\begin{proof}
Clearly, the bilinear from $\tilde{\mathfrak{B}}_{d}(-,-)$ is nondegenerate antisymmetric.
Since the perm algebra $(A\oplus A^{\ast}, \diamond)$ is induced by $(A\bowtie A^{\ast},
\alpha\oplus\beta^{\ast})$, it is easy to see that $(A, \bullet)$ and $(A^{\ast}, \bullet)$
are perm subalgebras of $(A\oplus A^{\ast}, \diamond)$, $A$ and $A^{\ast}$ are isotropic
with respect to $\tilde{\mathfrak{B}}_{d}(-,-)$, and for any $a_{1}, a_{2}, a_{3}\in A$
and $\xi_{1}, \xi_{2}, \xi_{3}\in A^{\ast}$,
\begin{align*}
&\;\tilde{\mathfrak{B}}_{d}\big((a_{1}, \xi_{1})\diamond(a_{2}, \xi_{2}),\ \
(a_{3}, \xi_{3})\big)\\
=&\;\langle\xi_{3},\; \alpha(a_{1})a_{2}\rangle
+\langle\beta^{\ast}(\xi_{1})\cdot_{A^{\ast}}\xi_{3},\; a_{2}\rangle
+\langle\alpha^{\ast}(\xi_{2}\cdot_{A^{\ast}}\xi_{3}),\;a_{1}\rangle\\[-1mm]
&\qquad-\langle\beta^{\ast}(\xi_{1})\cdot_{A^{\ast}}\xi_{2},\; a_{3}\rangle
-\langle\xi_{2},\; \alpha(a_{1})a_{3}\rangle
-\langle\xi_{1},\; \beta(a_{2}a_{3})\rangle,\\
&\;\tilde{\mathfrak{B}}_{d}\big((a_{1}, \xi_{1}),\ \ (a_{2}, \xi_{2})\diamond(a_{3}, \xi_{3})
-(a_{3}, \xi_{3})\diamond(a_{2}, \xi_{2})\big)\\
=&\;\langle\beta^{\ast}(\xi_{2})\cdot_{A^{\ast}}\xi_{3},\; a_{1}\rangle
+\langle\xi_{3},\;a_{1}\alpha(a_{2})\rangle
+\langle\xi_{2},\; \beta(a_{1}a_{3})\rangle
-\langle\xi_{1},\; \alpha(a_{2})a_{3}\rangle\\[-1mm]
&\qquad-\langle\xi_{1}\cdot_{A^{\ast}}\beta^{\ast}(\xi_{2}),\; a_{3}\rangle
-\langle\alpha^{\ast}(\xi_{1}\cdot_{A^{\ast}}\xi_{3}),\; a_{2}\rangle
-\langle\xi_{2}\cdot_{A^{\ast}}\beta^{\ast}(\xi_{3}),\; a_{1}\rangle
-\langle\xi_{2},\; a_{1}\alpha(a_{3})\rangle\\[-1mm]
&\qquad\quad-\langle\xi_{3},\; \beta(a_{1}a_{2})\rangle
+\langle\xi_{1},\; a_{2}\alpha(a_{3})\rangle
+\langle\xi_{1}\cdot_{A^{\ast}}\beta^{\ast}(\xi_{3}),\; a_{2}\rangle
+\langle\alpha^{\ast}(\xi_{1}\cdot_{A^{\ast}}\xi_{2}),\; a_{3}\rangle,
\end{align*}
Thus, $\tilde{\mathfrak{B}}_{d}(-,-)$ is invariant on the perm algebra $(P, \diamond)$
if and only if Eq. \eqref{compa1} holds and $\alpha^{\ast}(\xi_{1}\cdot_{A^{\ast}}\xi_{2})
=\beta^{\ast}(\xi_{1})\cdot_{A^{\ast}}\xi_{2}-\xi_{1}\cdot_{A^{\ast}}\beta^{\ast}(\xi_{2})$
for any $\xi_{1}, \xi_{2}\in A^{\ast}$. Note that the equation above is just the dual
of Eq. \eqref{compa2}, we get the proof.
\end{proof}
\subsection{Induced perm bialgebras and solutions of Yang-Baxter equation}
\label{subsec:perm-bia}
Recall that a pair $(P, \bar{\Delta})$ is called a {\it perm coalgebra}, where $P$ is a
vector space and $\bar{\Delta}: P\rightarrow P\otimes P$ is a linear map such that
$$
(\bar{\Delta}\otimes\id)\bar{\Delta}=(\id\otimes\bar{\Delta})\bar{\Delta}
=(\tau\otimes\id)(\bar{\Delta}\otimes\id)\bar{\Delta}.
$$
The notion of a perm coalgebra is the dualization of the notion of a perm algebra, that is,
$(P, \bar{\Delta})$ is a finite-dimensional perm coalgebra if and only if $(P^{\ast},
\bar{\Delta}^{\ast})$ is a perm algebra.

\begin{lem} \label{lem:permcoa}
Let $(A, \Delta, \beta)$ be an averaging cocommutative coalgebra.
Then $(A, \bar{\Delta})$ is perm coalgebra, called the {\rm perm coalgebra induced}
by $(A, \Delta, \beta)$, where $\bar{\Delta}$ is defined by
$$
\bar{\Delta}=(\beta\otimes\id)\Delta.
$$
Moreover, $(A^{\ast}, \bar{\Delta}^{\ast})$ is exactly the perm algebra induced by
the averaging commutative algebra $(A^{\ast}, \Delta^{\ast}, \beta^{\ast})$.
\end{lem}

\begin{proof}
Since $(A, \Delta, \beta)$ is a cocommutative averaging coalgebra, we get a
commutative averaging algebra $(A^{\ast}, \Delta^{\ast}, \beta^{\ast})$.
Let $(A^{\ast}, \bullet)$ be the perm algebra induced by $(A^{\ast}, \Delta^{\ast},
\beta^{\ast})$, that is, for any $\xi_{1}, \xi_{2}\in A^{\ast}$,
$\xi_{1}\bullet\xi_{2}:=\Delta^{\ast}(\beta^{\ast}(\xi_{1})\otimes\xi_{2})$.
It is straightforward that $\bullet$ is just $\bar{\Delta}^{\ast}$,
the linear dual of $\bar{\Delta}$.
Thus, $(A^{\ast}, \bar{\Delta}^{\ast})$ is the perm algebra induced by $(A^{\ast},
\Delta^{\ast},\beta^{\ast})$, and so that, $(A, \bar{\Delta})$ is a perm coalgebra.
\end{proof}

\begin{defi}[\cite{LZB,Hou}]\label{def:perm-bia}
A {\rm perm bialgebra} is a triple $(P, \bullet, \bar{\Delta})$, where $(P, \bullet)$ is
a perm algebra and $(P, \bar{\Delta})$ is a perm coalgebra, such that for any
$p_{1}, p_{2}\in P$,
\begin{align}
\bar{\Delta}(p_{1}\bullet p_{2})&=((\tilde{\fl}-\tilde{\fr})(p_{1})\otimes\id)
\bar{\Delta}(p_{2})+(\id\otimes\,\tilde{\fr}(p_{2}))\bar{\Delta}(p_{1}),   \label{pbi1}\\
\tau(\tilde{\fr}(p_{2})\otimes\id)\bar{\Delta}(p_{1})
&=(\tilde{\fr}(p_{1})\otimes\id)\bar{\Delta}(p_{2}),             \label{pbi2}\\
\bar{\Delta}(p_{1}\bullet p_{2})&=(\id\otimes\,\tilde{\fl}(p_{1}))\bar{\Delta}(p_{2})
+((\tilde{\fl}-\tilde{\fr})(p_{2})\otimes\id)(\bar{\Delta}(p_{1})
-\tau\bar{\Delta}(p_{1})).                                       \label{pbi3}
\end{align}
\end{defi}

\begin{pro}\label{pro:asspoibi}
Let $(A, \alpha)$ be a commutative averaging algebra, $(A, \Delta, \beta)$ be a
cocommutative averaging coalgebra, $(A, \bullet)$ be the perm algebra induced by
$(A, \alpha)$ and $(A, \bar{\Delta})$ be the perm coalgebra induced by $(A, \Delta, \beta)$.
Suppose that $(A, \Delta, \alpha, \beta)$ is a commutative and cocommutative
averaging ASI bialgebra. Then $(A, \bullet, \bar{\Delta})$ is a perm bialgebra if and
only if, for any $a_{1}, a_{2}\in A$,
\begin{align}
&(\beta\fl_{A}(a_{2})\otimes\id)\Delta(\alpha(a_{1}))
-(\beta\fl_{A}(\alpha(a_{2}))\otimes\id)\Delta(a_{1})
+(\fl_{A}(a_{1})\alpha\beta\otimes\id)\Delta(a_{2})=0,     \label{indbi1}\\
&\qquad\qquad\quad(\id\otimes\fl_{A}(a_{2})\alpha\beta)\Delta(a_{1})
-(\fl_{A}(a_{1})\alpha\beta\otimes\id)\Delta(a_{2})=0,     \label{indbi2}\\
&(\beta\fl_{A}(a_{2})\otimes\id)\Delta(\alpha(a_{1}))
-(\fl_{A}(\alpha(a_{2}))\beta\otimes\id)\Delta(a_{1})
+(\fl_{A}(a_{2})\alpha\beta\otimes\id)\Delta(a_{1})         \label{indbi3}\\[-1mm]
&\qquad\qquad\qquad\qquad\qquad\;+(\fl_{A}(\alpha(a_{2}))\otimes\beta)\Delta(a_{1})
-(\fl_{A}(a_{2})\alpha\otimes\beta)\Delta(a_{1})=0.         \nonumber
\end{align}
\end{pro}

\begin{proof}
For any $a_{1}, a_{2}\in A$, by Eq. \eqref{bialg1}, we have
$$
\bar{\Delta}(a_{1}\bullet a_{2})=\bar{\Delta}(\alpha(a_{1})a_{2})
=(\beta\otimes\id)\Delta(\alpha(a_{1})a_{2})
=(\beta\fl_{A}(a_{2})\otimes\id)\Delta(\alpha(a_{1}))
+(\beta\otimes \fl_{A}(\alpha(a_{1})))\Delta(a_{2}).
$$
Moreover, since $A$ is commutative and cocommutative, we get
\begin{align*}
&\;((\tilde{\fl}-\tilde{\fr})(a_{1})\otimes\id)
\bar{\Delta}(a_{2})+(\id\otimes\tilde{\fr}(a_{2}))\bar{\Delta}(a_{1})\\
=&\;((\tilde{\fl}-\tilde{\fr})(a_{1})\otimes\id)(\beta\otimes\id)\Delta(a_{2})
+(\id\otimes\tilde{\fr}(a_{2}))(\beta\otimes\id)\Delta(a_{1})\\
=&\;(\beta\fl_{A}(\alpha(a_{1}))\otimes\id)\Delta(a_{2})
-(\fl_{A}(a_{1})\alpha\beta\otimes\id)\Delta(a_{2})
+(\beta\otimes\fl_{A}(a_{2})\alpha)\Delta(a_{1}).
\end{align*}
Note that the Eq. \eqref{bialg1} means $\big(\fl_{A}(a_{1})\otimes\id-\id\otimes
\,\fl_{A}(a_{1})\big)\Delta(a_{2})=\big(\fl_{A}(a_{2})\otimes\id-\id\otimes
\,\fl_{A}(a_{2})\big)\Delta(a_{1})$, we get Eq. \eqref{pbi1} holds if and only if
Eq. \eqref{indbi1} holds. Similarly, we can get Eq. \eqref{pbi2} holds if and only if
Eq. \eqref{indbi2} holds and Eq. \eqref{pbi3} holds if and only if Eq. \eqref{indbi3} holds.
The proof is finished.
\end{proof}

In Proposition \ref{pro:asspoibi}, if Eqs. \eqref{compa1} and \eqref{compa2} hold,
then we obtain $(\alpha\otimes\beta)\Delta=0$, $a_{1}\alpha(\beta(a_{2}))=0$
for any $a_{1}, a_{2}\in A$, and one can check that Eqs. \eqref{pbi1}-\eqref{pbi3} hold.
That is to say, in this case, $(A, \bullet, \bar{\Delta})$ is a perm bialgebra, which
is called {\it the induced perm bialgebra} by a commutative and cocommutative
averaging ASI bialgebra $(A, \Delta, \alpha, \beta)$.

\begin{ex}\label{ex:perm-bialg}
Consider the averaging ASI bialgebra $(A, \Delta, \alpha, \beta)$ given in Example
\ref{ex:YBE-bialg1}. That is, $A=\Bbbk\{e_{1}, e_{2}, e_{3}\}$, the non-zero product is
given by $e_{1}e_{1}=e_{1}$, $e_{1}e_{2}=e_{2}=e_{2}e_{1}$, the comultiplication is given
by $\Delta(e_{1})=-e_{2}\otimes e_{3}-e_{3}\otimes e_{2}$, $\Delta(e_{2})=\Delta(e_{3})=0$,
$\alpha(e_{1})=e_{3}$, $\alpha(e_{3})=\alpha(e_{2})=0$ and $\beta=0$.
Then one can check that the induced perm bialgebra $(A, \bullet, \bar{\Delta})$
by $(A, \Delta, \alpha, \beta)$ is trivial, i.e., the multiplication $\bullet$
and the comultiplication $\bar{\Delta}$ are zero.
\end{ex}

For perm bialgebras, we have the following Theorem.

\begin{thm}[\cite{LZB,Hou}]\label{thm:poibieq}
Let $(P, \bullet)$ be a perm algebra. Suppose that there is a perm algebra structure
$(P^{\ast}, \bullet)$ on $A^{\ast}$, and $\Delta: P\rightarrow P\otimes P$ is
the linear dual of $\bullet$ in $P^{\ast}$. Then the following conditions are equivalent:
\begin{enumerate}\itemsep=0pt
\item[$(i)$] $(P, \bullet, \Delta)$ is a perm bialgebra;
\item[$(ii)$] $((P, \bullet),\; (P^*, \bullet),\; \tilde{\fl}_{A}^{\ast},\;
      \tilde{\fl}_{A}^{\ast}-\tilde{\fr}_{A}^{\ast},\; \tilde{\fl}_{A^{\ast}}^{\ast},\;
      \tilde{\fl}_{A^{\ast}}^{\ast}-\tilde{\fr}_{A^{\ast}}^{\ast})$ is a matched pair
      of perm algebras;
\item[$(iii)$] There is a Manin triple of perm algebras $((A\bowtie A^{\ast}, \diamond,
     \tilde{\mathfrak{B}}_{d}),\; A,\; A^{\ast})$ associated to $(A, \bullet)$ and
     $(A^{\ast}, \bullet)$, where $\tilde{\mathfrak{B}}_d$ is defined in Proposition
     \ref{pro:assFro}.
\end{enumerate}
\end{thm}

Thus, for the induced matched pair of perm algebras, induced Manin triple of perm algebras
and induced perm bialgebras, 
we have
{\footnotesize$$
\xymatrix@C=1.12cm{
\txt{$((A, \alpha), (A^{\ast}, \alpha^{\ast}), \fl_{A}^{\ast}, \fl_{A^{\ast}}^{\ast})$ \\
{\tiny a matched pair of} \\ {\tiny commutative averaging algebras}}
\ar[d]_-{{\rm Cor.} \ref{cor:adjassmat}} \ar@{<->}[r]^-{{\rm Thm.}~\ref{thm:equ}}
& \txt{$(A, \Delta, \alpha, \alpha)$ \\ {\tiny a commutative and cocommutative} \\
{\tiny averaging ASI bialgebra}} \ar@{<->}[r]^-{{\rm Thm.}~\ref{thm:equ}}
\ar[d]_-{{\rm Pro.}~\ref{pro:asspoibi}}
& \txt{$(A\bowtie A^{\ast}, \alpha\oplus\alpha^{\ast},\mathfrak{B}_{d})$ \\
{\tiny a double construction of} \\ {\tiny commutative averaging Frobenius algebra}}
\ar[d]_-{{\rm Pro.}~\ref{pro:assFro}} \\
\txt{$(A, A^{\ast}, \tilde{\fl}_{A}^{\ast}, \tilde{\fl}_{A}^{\ast}-\tilde{\fr}_{A}^{\ast},
\tilde{\fl}_{A^{\ast}}^{\ast}, \tilde{\fl}_{A^{\ast}}^{\ast}
-\tilde{\fr}_{A^{\ast}}^{\ast})$ \\ {\tiny the induced matched pair} \\
{\tiny of perm algebras}} \ar@{<->}[r]^-{{\rm Thm.}~\ref{thm:poibieq}}
& \txt{$(A, \bullet, \bar{\Delta})$ \\ {\tiny the induced} \\
{\tiny perm bialgebra}} \ar@{<->}[r]^-{{\rm Thm.}~\ref{thm:poibieq}} &
\txt{$((A\oplus A^{\ast}, \bullet, \tilde{\mathfrak{B}}_{d}),\; A,\; A^{\ast})$ \\
{\tiny the induced Manin triple} \\ {\tiny of perm algebras}}
}
$$}

Next, we consider solutions of Yang-Baxter equation in perm algebras.
Let $(P, \bullet)$ be a perm algebra and $r\in P\otimes P$. Then equation
$$
r_{12}\bullet r_{23}-r_{13}\bullet r_{23}+r_{12}\bullet r_{13}-r_{13}\bullet r_{12}=0
$$
is called the {\it Yang-Baxter equation in perm algebra} $(P, \bullet)$.
The solutions of Yang-Baxter equation in perm algebras are closely related to
the perm bialgebras \cite{LZB}.

\begin{pro}\label{pro:aybe-pybe}
Let $(A, \alpha)$ be a commutative averaging algebra, $(A, \bullet)$ be the
perm algebra induced by $(A, \alpha)$, and $(M, \mu, \beta)$ be a bimodule over
$(A, \alpha)$. If Eq. \eqref{compa1} holds, then each solution of $\beta$-YBE in
$(A, \alpha)$ is a solution of YBE in the induced perm algebra $(A, \bullet)$.
\end{pro}

\begin{proof}
Suppose that $r=\sum_{i}x_{i}\otimes y_{i}\in A\otimes A$ is a solution of $\beta$-YBE
in $(A, \alpha)$. That is,
\begin{align*}
&\qquad\sum_{i,j}x_{i}x_{j}\otimes y_{i}\otimes y_{j}
+\sum_{i,j}x_{i}\otimes x_{j}\otimes y_{i}y_{j}
-\sum_{i,j}x_{j}\otimes x_{i}y_{j}\otimes y_{i}=0,\\[-1mm]
&\sum_{i}\alpha(x_{i})\otimes y_{i}=\sum_{i}x_{i}\otimes\beta(y_{i}),\qquad\qquad
\sum_{i}\beta(x_{i})\otimes y_{i}=\sum_{i}x_{i}\otimes\alpha(y_{i}).
\end{align*}
Since $\beta(a_{1}a_{2})=\alpha(a_{1})a_{2}-a_{1}\alpha(a_{2})$ for any $a_{1}, a_{2}\in A$,
we have
\begin{align*}
&\;r_{13}\bullet r_{23}+r_{13}\bullet r_{12}\\
=&\;\sum_{i,j}x_{i}\otimes x_{j}\otimes\alpha(y_{i})y_{j}
+\sum_{i,j}\alpha(x_{i})x_{j}\otimes y_{j}\otimes y_{i}\\[-1mm]
=&\;\sum_{i,j}\beta(x_{j})\otimes x_{i}y_{j}\otimes y_{i}
-\sum_{i,j}\beta(x_{i}x_{j})\otimes y_{i}\otimes y_{j}
+\sum_{i,j}\alpha(x_{i})x_{j}\otimes y_{j}\otimes y_{i}\\[-1mm]
=&\;\sum_{i,j}\beta(x_{j})\otimes x_{i}y_{j}\otimes y_{i}
-\sum_{i,j}\alpha(x_{j})x_{i}\otimes y_{i}\otimes y_{j}
+\sum_{i,j}\alpha(x_{i})x_{j}\otimes y_{i}\otimes y_{j}
+\sum_{i,j}\alpha(x_{i})x_{j}\otimes y_{j}\otimes y_{i}\\[-1mm]
=&\;\sum_{i,j}\beta(x_{j})\otimes x_{i}y_{j}\otimes y_{i}
+\sum_{i,j}\alpha(x_{i})x_{j}\otimes y_{i}\otimes y_{j}\\[-1mm]
=&\;\sum_{i,j}x_{i}\otimes\alpha(y_{i})x_{j}\otimes y_{j}
+\sum_{i,j}\alpha(x_{i})x_{j}\otimes y_{i}\otimes y_{j}\\
=&\;r_{12}\bullet r_{23}+r_{12}\bullet r_{13}.
\end{align*}
This means that $r$ is a solution of YBE in the induced perm algebra $(A, \bullet)$.
\end{proof}

\begin{ex}\label{ex:YBE-perm}
Let $(A, \alpha)$ be the 3-dimensional commutative averaging algebra, which is given
by $A=\Bbbk\{e_{1}, e_{2}, e_{3}\}$ with non-zero product $e_{1}e_{1}=e_{2}$,
$e_{1}e_{2}=e_{3}=e_{2}e_{1}$ and $\alpha(e_{1})=e_{2}+e_{3}$, $\alpha(e_{2})=-e_{2}$,
$\alpha(e_{3})=e_{3}$. Define linear map $\beta: A\rightarrow A$ by $\beta(e_{1})=-e_{1}$,
$\beta(e_{2})=0$, $\beta(e_{3})=e_{3}$. Then $(A^{\ast}, \fr_{A}^{\ast}, \fl_{A}^{\ast},
\beta^{\ast})$ is a bimodule over $(A, \alpha)$, Eq. \eqref{compa1} holds and the induced
perm algebra $(A, \bullet)$ is given by $e_{1}\bullet e_{1}=e_{3}=e_{2}\bullet e_{1}$.
Let $r=e_{3}\otimes e_{3}$. Then one can check that $r$ is a symmetric solution
of the $\beta$-YBE in $(A, \alpha)$ and satisfies Eq. \eqref{coba1}. It is easy to
see that $r$ is also a solution of the YBE in perm algebra $(A, \bullet)$.
\end{ex}

Recently, Lin, Zhou and Bai constructed Lie biagebra by using a perm algebra and a pre-Lie
algebra \cite{LZB}. Recall that a {\it pre-Lie algebra} $(Q, \circ)$ is a vector space
$Q$ with a binary operation $\circ: Q\otimes Q\rightarrow Q$ such that
$$
(q_{1}\circ q_{2})\circ q_{3}-q_{1}\circ(q_{2}\circ q_{3})
=(q_{2}\circ q_{1})\circ q_{3}-q_{2}\circ(q_{1}\circ q_{3}),
$$
for any $q_{1}, q_{2}, q_{3}\in Q$. Let $(Q, \circ)$ be a pre-Lie algebra.
A bilinear from $\omega(-, -)$ is called {\it invariant}, if
$\omega(q_{1}\circ q_{2},\; q_{3})=-\omega(q_{2},\; q_{1}\circ q_{3}-q_{3}\circ q_{1})$.
A pre-Lie algebra $(Q, \circ)$ with an antisymmetric nondegenerate invariant bilinear
form $\omega(-,-)$ is called a {\it quadratic pre-Lie algebra}.
Let $(P, \bullet)$ be a perm algebra and $(Q, \circ)$ be a pre-Lie algebra. Define
a binary operation $[-,-]: (P\otimes Q)\otimes(P\otimes Q)\rightarrow(P\otimes Q)$ by
$$
[p_{1}\otimes q_{1},\; p_{2}\otimes q_{2}]=(p_{1}\bullet p_{2})\otimes(q_{1}\circ q_{2})
-(p_{2}\bullet p_{1})\otimes(q_{2}\circ q_{1}),
$$
for any $p_{1}, p_{2}\in P$ and $q_{1}, q_{2}\in Q$. Then one can check that
$(P\otimes Q,\; [-,-])$ is a Lie algebra, which is called {\it the induced Lie algebra}
from $(P, \bullet)$ and $(Q, \circ)$. In \cite{LZB}, the authors have extended this
conclusion to the Lie bialgebra, and constructed some solutions of the classical Yang-Baxter
equation in Lie algebra from solutions of YBE in perm algebra. Let $(L, [-,-])$ be
a Lie algebra and $r\in L\otimes L$. The classical Yang-Baxter equation
in Lie algebra $(L, [-,-])$ is given by
$$
[r_{12}, r_{13}]+[r_{12}, r_{23}]+[r_{13}, r_{23}]=0.
$$

\begin{pro}[\cite{LZB}]\label{pro:pr-lieybe}
Let $(P, \bullet)$ be a perm algebra, $(Q, \circ, \omega)$ be a quadratic pre-Lie
algebra, and $(P\otimes Q,\; [-,-])$ be the induced Lie algebra from $(P, \bullet)$
and $(Q, \circ)$. Let $\{e_{1}, e_{2},\cdots, e_{n}\}$ be a basis of $Q$, and $\{f_{1},
f_{2},\cdots, f_{n}\}$ be the dual basis with respect to $\omega(-,-)$. If $r=\sum_{i}
x_{i}\otimes y_{i}\in P\otimes P$ is a solution of the perm algebra $(P, \bullet)$, then
$$
\tilde{r}=\sum_{i, j}(x_{i}\otimes e_{j})\otimes(y_{i}\otimes f_{j})
$$
is a solution of the classical Yang-Baxter equation in $(P\otimes Q,\; [-,-])$.
\end{pro}

Thus, by Propositions \ref{pro:aybe-pybe} and \ref{pro:pr-lieybe}, we have

\begin{pro}\label{pro:av-lieybe}
Let $(A, \alpha)$ be a commutative averaging algebra, $(M, \mu, \beta)$ be a bimodule
over it, and $(Q, \circ, \omega)$ be a quadratic pre-Lie algebra.
\begin{enumerate}\itemsep=0pt
\item[$(i)$] If define a $[-,-]: (A\otimes Q)\otimes(A\otimes Q)\rightarrow(A\otimes Q)$ by
  $$
  [a_{1}\otimes q_{1},\; a_{2}\otimes q_{2}]=(\alpha(a_{1})a_{2})\otimes(q_{1}\circ q_{2})
  -(\alpha(a_{2})a_{1})\otimes(q_{2}\circ q_{1}),
  $$
  for any $a_{1}, a_{2}\in A$ and $q_{1}, q_{2}\in Q$, then $(A\otimes Q,\; [-,-])$ is
  a Lie algebra.
\item[$(ii)$] Let $\{e_{1}, e_{2},\cdots, e_{n}\}$ be a basis of $Q$, and $\{f_{1},
f_{2},\cdots, f_{n}\}$ be the dual basis with respect to $\omega(-,-)$. If Eq.
\eqref{compa1} holds, each solution $r=\sum_{i}x_{i}\otimes y_{i}$ of $\beta$-YBE
in $(A, \alpha)$ gives a solution
$$
\tilde{r}=\sum_{i, j}(x_{i}\otimes e_{j})\otimes(y_{i}\otimes f_{j})
$$
of the classical Yang-Baxter equation in $(A\otimes Q,\; [-,-])$.
\end{enumerate}
\end{pro}

\begin{ex}\label{ex:YBE-Lie}
Let $(A=\Bbbk\{e_{1}, e_{2}, e_{3}\}, \alpha)$ be the 3-dimensional commutative
averaging algebra given in Example \ref{ex:YBE-perm}. Then we get a perm algebra
$(A, \bullet)$ with nonzero product $e_{1}\bullet e_{1}=e_{3}=e_{2}\bullet e_{1}$.
Considering quadratic pre-Lie algebra $(Q=\Bbbk\{q_{1}, q_{2}\}, \omega)$, where the
nonzero product is given by $q_{1}q_{2}=q_{1}$, $q_{2}q_{2}=q_{2}$, and $\omega(q_{1},
q_{2})=-\omega(q_{2}, q_{1})=1$, $\omega(q_{1}, q_{1})=\omega(q_{2}, q_{2})=0$,
we obtain a Lie algebra $(A\otimes Q,\; [-,-])$, where the nonzero product is given
by $[e_{1}\otimes q_{1},\; e_{1}\otimes q_{2}]=e_{2}\otimes q_{1}$,
$[e_{2}\otimes q_{1},\; e_{1}\otimes q_{2}]=e_{3}\otimes q_{1}$,
$[e_{2}\otimes q_{2},\; e_{1}\otimes q_{2}]=e_{3}\otimes q_{2}$. Moreover,
by the symmetric solution $r=e_{3}\otimes e_{3}$ of the $\beta$-YBE in $(A, \alpha)$,
we get an antisymmetric solution $\tilde{r}=(e_{3}\otimes q_{1})\otimes(e_{3}\otimes q_{2})-
(e_{3}\otimes q_{2})\otimes(e_{3}\otimes q_{1})$ of the classical Yang-Baxter
equation in $(A\otimes Q,\; [-,-])$.
\end{ex}

\bigskip
\noindent
{\bf Acknowledgements. } This work was financially supported by National
Natural Science Foundation of China (No.11771122).

%

 \end{document}